%% file: StoDCuP_Revised.tex
\documentclass{amsart}

 \usepackage{amsmath}
\usepackage{amsfonts}
\usepackage{amssymb}
\usepackage[dvips]{graphicx}
\usepackage{color}
\usepackage{float} 
\usepackage{fullpage}
\usepackage{epsfig} 

\usepackage{url}

\include{def}

\def\dom{{\rm dom}}
\def\Dom{{\rm Dom}}

\def\ri {{\rm ri}}

\def\R{{\mathbb{R}}}
 
\newcommand{\transp}{{\scriptscriptstyle \top}}

\begin{document}

\title[]{Stochastic Dynamic Cutting Plane for multistage stochastic convex programs}

\maketitle

\vspace*{0.5cm}


\begin{center}
\begin{tabular}{ccc}
\begin{tabular}{c}
Vincent Guigues\\
School of Applied Mathematics, FGV\\
Praia de Botafogo, Rio de Janeiro, Brazil\\ 
{\tt vincent.guigues@fgv.br}
\end{tabular}&

&
\begin{tabular}{c}
Renato D.C.\ Monteiro\\
Georgia Institute of Technology\\
Atlanta, Georgia 30332, USA,\\
{\tt renato.monteiro@isye.gatech}\\
\end{tabular}
\end{tabular}
\end{center}

\date{}

\begin{abstract} We introduce StoDCuP (Stochastic Dynamic Cutting Plane),
an extension of the Stochastic Dual Dynamic Programming (SDDP) algorithm 
to solve multistage stochastic convex optimization problems. At each iteration, the algorithm builds lower bounding
affine functions not only for the cost-to-go functions, as SDDP does, but also for some or all nonlinear cost and constraint
functions. We show the almost sure convergence of StoDCuP.
We also introduce an inexact variant of StoDCuP where all subproblems are solved approximately (with bounded errors) and show the almost sure convergence 
of this variant for vanishing errors. Finally, numerical experiments are presented 
on nondifferentiable
multistage stochastic programs
where Inexact StoDCuP computes
a good approximate policy quicker than
StoDCuP while SDDP and the previous
inexact variant of SDDP 
combined with  
Mosek library to solve subproblems
were not able to solve the 
differentiable reformulation of the
problem.
\end{abstract}



\section{Introduction}

Risk-neutral multistage stochastic programs (MSPs) aim at minimizing the expected value of the total cost
over a given optimization period of $T$ stages while satisfying almost surely for every stage some
constraints depending on an underlying stochastic process. These optimization problems are useful for many
real-life applications but are challenging to solve, see for instance \cite{shadenrbook} and references therein for a thorough discussion on MSPs. Popular solution methods for MSPs are based on decomposition techniques such as Approximate Dynamic Programming \cite{powellbook},
Lagrangian relaxation, or Stochastic Dual Dynamic Programming (SDDP) \cite{pereira}.
SDDP is a sampling-based extension
of \cite{birgemulti}, itself a multistage extension of the L-shaped
method \cite{vwets}. The SDDP method
builds linearizations of the convex cost-to-go
functions at trial points 
computed on scenarios of the underlying stochastic
process
generated
randomly along iterations.
The use of such cutting plane
models for the objective function
in the context of deterministic convex
optimization dates back to
 Kelley's cutting plane method \cite{kelley}
and has later been extended in many
variants such as 
subgradient \cite{kiwiel1},
bundle \cite{kiwiel2,lembundle},
and level \cite{lemnestnem} variants.
Kelley's algorithm was also generalized
by Benders to solve \cite{benders}
mixed-variables programming problems.
Recently, several enhancements of SDDP 
have been proposed, see for instance \cite{shapsddp}, \cite{guiguesrom10}, \cite{philpmatos}, \cite{kozmikmorton} for risk-averse variants,
\cite{philpot}, \cite{lecphilgirar12}, \cite{guiguessiopt2016} for convergence analysis, \cite{shapding} for the application of SDDP
to periodic stochastic programs,
and \cite{ruialex}, \cite{guiguesinexact2018} 
to speed up the convergence of the method. In particular, in \cite{guiguesinexact2018}, Inexact SDDP was proposed, which incorporates inexact cuts in SDDP
(for both linear and nonlinear programs). The idea of Inexact SDDP is to allow us to solve approximately some
or all primal and dual subproblems in the forward and backward passes of SDDP. This extension and the study of Inexact SDDP was motivated
by the following reasons:
\begin{itemize}
\item[(i)] solving to a very high accuracy nonlinear programs can take a significant amount of time or may even be impossible
whereas linear programs (of similar sizes) can be solved exactly or to high accuracy quicker. Examples of convex but challenging
to solve subproblems include
semidefinite programs \cite{hsdpbook},
quadratically constrained
quadratic programs with degenerate
quadratic forms (see the numerical
experiments of Section \ref{sec:numexp}), or some high dimensional 
nondifferentiable problems.
For subproblems where it is difficult
or impossible
to get optimal solutions, if we
are able to provide a feasible
primal-dual solution, we should
be able to derive an extension of
SDDP, i.e., cuts for the cost-to-go
functions, from approximate
subproblem 
primal-dual solutions. Therefore one has to study how 
to extend the SDDP algorithm to still derive
valid cuts and a converging Inexact SDDP or an Inexact SDDP with controlled accuracy when only approximate primal and dual solutions
are computed for nonlinear MSPs.
\item[(ii)] As explained in \cite{guiguesinexact2018}, numerical experiments 
(see for instance \cite{guiguesejor17, guiguesbandarra18, guilejtekregsddp}) 
indicate that very loose cuts are computed in the first iterations of SDDP
and it may be useful to compute with less accuracy these cuts for the first
iterations. Using this strategy, it was shown in \cite{guiguesinexact2018} that for several instances of a portfolio problem, Inexact SDDP
can converge (i.e., satisfy the stopping criterion) quicker than SDDP.
\end{itemize}

In this paper, we extend \cite{guiguesinexact2018} in two ways:

\begin{itemize}
 \item a natural way of taking advantage of observation (i) above in the context of SDDP applied to nonlinear problems, consists in linearizing some or all nonlinear objective and
constraint functions of the subproblems solved along the iterations of the method at the optimal solutions of these subproblems. When all nonlinear functions are
linearized, all subproblems solved
in the iterations of SDDP are linear
programs which allows us to
avoid having to solve 
difficult problems
that cannot be solved with high accuracy. However, to the best of our knowledge, this variant of SDDP, that we term as StoDCuP (Stochastic Dynamic Cutting Plane) has not been proposed and studied
so far in the literature (SDDP does
build linearizations for the cost-to-go functions
but not for some or all of  the remaining nonlinear objective and constraint functions). In this context, the goal of this paper is to propose and study StoDCuP. 
\item As far as (ii) is concerned, it is interesting
to notice that it is easy to incorporate inexact cuts in StoDCuP (i.e., to derive an inexact variant of StoDCuP), control the quality of 
these cuts (see Lemma \ref{lemmacutsbis}), and show the convergence of this method (see Theorem \ref{convproofStoDCuPinex} below). This comes
from the fact that we can easily compute a cut for the value function of a linear program (and in StoDCuP all subproblems solved are linear programs) from any feasible primal-dual solution
since the corresponding dual objective is linear, see Proposition 2.1 in \cite{guiguesinexact2018}. On the contrary, deriving valid (inexact) cuts from approximate primal-dual solutions of the subproblems solved in SDDP applied to nonlinear problems and showing the convergence
of the corresponding variant of Inexact SDDP is technical and the computation of inexact cuts may require solving additional subproblems, see \cite{guiguesinexact2018} for details. Moreover, Inexact SDDP from 
\cite{guiguesinexact2018} applies
to differentiable multistage
convex stochastic programs
while 
both StoDCuP and Inexact StoDCuP
apply to more general differentiable
or  nondifferentiable multistage
convex stochastic programs.
\end{itemize}

The outline of the paper is the following. To ease the presentation and analysis of StoDCuP, we start in Section \ref{sec:dcup} with its deterministic counterpart,
called DCuP (Dynamic Cutting Plane) which solves  convex Dynamic Programming equations linearizing cost-to-go, constraint, and objective functions.
Starting with the deterministic case allows us to focus on the differences between traditional Dual Dynamic Programming and its convergence analysis
with DCuP and its convergence analysis. 
In Section \ref{stodcup}, we introduce forward StoDCuP and prove the almost sure convergence of the method. 
In Section \ref{variantsstodcup}, we present 
 Inexact StoDCuP, an inexact variant
 of StoDCuP which builds
inexact cuts on the basis of approximate primal-dual solutions of the subproblems solved along the iterations of the method. We also prove
the almost sure convergence of Inexact StoDCuP for vanishing noises. 
Our convergence proofs of 
DCuP and StoDCuP are based
on the convergence analysis of
SDDP for nonlinear problems
in \cite{guiguessiopt2016}
but additional technical results
are needed due to the linearizations
of cost and constraint functions,
see Lemma \ref{lm:2.1}-(c),(d), 
Lemma \ref{lm:2.2}, Lemma \ref{lipcontQt}-(b), 
Theorem \ref{convprooftheorem}-(i),(ii),
Lemma \ref{lipcontQtstob}-(c), Lemmas \ref{lemmaft}, \ref{lemmacuts}, and Theorem \ref{convproofStoDCuP}. Finally, numerical experiments are presented in Section \ref{sec:numexp}
on nondifferentiable
multistage stochastic programs.
Two variants of  Inexact StoDCuP 
are presented: with and without cut selection
strategies.
In all instances tested, at least one inexact variant
computes
a good approximate policy quicker than
StoDCuP while SDDP and the previous
inexact variant of SDDP from \cite{guiguesinexact2018}
combined with  
Mosek library to solve subproblems
were not able to solve a
differentiable reformulation of the
problem (recall that such reformulation is
necessary to use the inexact variant of SDDP
from \cite{guiguesinexact2018} which applies
to differentiable stochastic programs).

\par We will use the following notation:
\begin{itemize}
\item For a real-valued convex function $f$, we denote
by $\ell_{f}(\cdot; x_0)$ an arbitrary lower bounding linearization of $f$ at $x_0$, i.e.,
$\ell_{f}(\cdot; x_0)= f(x_0 ) + s_f( x_0 )^{\top} (\cdot - x_0)$ where $s_f( x_0)$
is an arbitrary subgradient of $f$ at $x_0$.
 \item The domain of a point to set operator $T: A \rightrightarrows B$ is given by
 Dom($T$)$=\{a \in A : T(a) \neq \emptyset\}$.
 \item For vectors $x, y \in \mathbb{R}^n$,
 $\langle x , y \rangle = x^\top y$ is the usual scalar
 product between $x$ and $y$.
 \item For $a \in \mathbb{R}^n$,  $\bar B(a;{\varepsilon})
 =\{x \in \mathbb{R}^n : \|x-a\|_2 \leq \varepsilon \}$.
 \item The domain of a  convex function $f:X \rightarrow (-\infty,\infty]$
 is $\mbox{dom}(f)=\{x \in X : f(x)<\infty\}$.
 \item The relative interior ri $X$ of a set $X$
 is the set $\{x \in X: \exists \varepsilon>0 : 
\bar B(x;{\varepsilon}) \cap \mbox{Aff}(X) \subset X\}$.
\item The subdifferential of the convex function $f:X \rightarrow (-\infty,\infty]$ at $x$ is
$$
\partial f(x)=\{s : f(y) \geq f(x) + \langle s, y-x \rangle\;
\forall y \in X\}.
$$
\item The indicator function $\delta_X(\cdot)$ of the
set $X$ is given by $\delta_X(x)=0$ if $x \in X$
and $\delta_X(x)=\infty$ otherwise.
\item A function $f:\mathbb{R}^n \rightarrow (-\infty,\infty]$
is proper if there is $x$ such that $f(x)$ is finite.
\item {\textbf{e}} is a vector of ones whose dimension
depends on the context.
\end{itemize}

\section{The DCuP (Dynamic Cutting Plane) algorithm}\label{sec:dcup}

\subsection{Problem formulation and assumptions}

Given $x_0 \in \mathbb{R}^n$, consider the optimization problem
\begin{equation}\label{defpb}
\left\{
\begin{array}{l}
\displaystyle \inf_{x_1,\ldots,x_T \in \mathbb{R}^n}\;\sum_{t=1}^{T} f_t(x_{t}, x_{t-1}) \\
g_t(x_t, x_{t-1}) \leq 0,\;\;\displaystyle A_{t} x_{t} + B_{t} x_{t-1} = b_t, \;t=1,\ldots,T,\\
x_t \in \mathcal{X}_t,\; t=1,\ldots,T,\\
\end{array}
\right.
\end{equation}
where $A_t$ and $B_t$ are matrices of appropriate dimensions,
$f_t: \R^n \times \R^n \to (-\infty,\infty]$ and
 $g_t:  \R^n \times \R^n \to  (-\infty,\infty]^p$.
 In this problem, for each step $t$, we have nonlinear 
and linear coupling constraints,  $g_t(x_t, x_{t-1}) \leq 0$ and $A_{t} x_{t} + B_{t} x_{t-1} = b_t$
respectively, and  set constraints 
$x_t \in \mathcal{X}_t$. 
Without loss of generality, nonlinear noncoupling constraints $h_{t}(x_t) \leq 0$ can be 
dealt with by incorporating them 
into the constraint $g_t(x_t,x_{t-1}) \le 0$. 
For convenience, we use the short notation
\begin{equation}
X_t( x_{t-1} ) :=
\{ x_t \in \mathcal{X}_t:\;g_t(x_t, x_{t-1}) \leq 0,\;\;\displaystyle A_{t} x_{t} + B_{t} x_{t-1} = b_t \}
\label{eq:Xt}
\end{equation}
and
\begin{equation}
X_t^0( x_{t-1} ) = X_t( x_{t-1} ) \cap \ri \, \mathcal{X}_t.
\label{eq:riXt}
\end{equation}


With this notation, the dynamic programming equations corresponding to problem \eqref{defpb} are
\begin{equation} \label{eq:Qt}
\mathcal{Q}_{t}(x_{t-1})=
\left\{
\begin{array}{l}
\displaystyle \inf_{x_t \in \mathbb{R}^n}\;F_t(x_{t}, x_{t-1}):=f_t(x_{t}, x_{t-1}) + \mathcal{Q}_{t+1}(x_{t})\\
x_t \in X_t( x_{t-1} ),
\end{array}
\right.
\end{equation}
for $t=1,\ldots,T$, and $\mathcal{Q}_{T+1} \equiv 0$.
The cost-to-go function $\mathcal{Q}_{t+1}(x_{t})$ represents the optimal
total cost for time steps $t+1, \ldots,T$, starting from state $x_t$ at the beginning of step $t+1$.
Clearly, it follows from the above definition that
\beq \label{eq:simobs}
\Dom(X^0_t) \subset \Dom(X_t) \quad \forall t=1,\ldots,T.
\eeq


Setting $\mathcal{X}_0=\{x_0\}$, the following assumptions are made throughout this section.\\

{\bf Assumption (H1):}
\begin{itemize}
\item[1)] For $t=1,\ldots,T$:
\begin{itemize}
\item[a)] $\mathcal{X}_{t} \subset \mathbb{R}^n$ is nonempty, convex, and compact;
\item[b)] $f_t$ is a proper lower-semicontinuous convex function such that
$\mathcal{X}_{t} \small{\times} \mathcal{X}_{t-1} \subset \inte (\dom(f_t))$;
\item[c)] each of the $p$ components $g_{t i},i=1,\ldots,p$, of $g_t$ is a 
proper lower-semicontinuous convex function
such that $\mathcal{X}_{t} \small{\times} \mathcal{X}_{t-1} \subset \inte (\dom(g_{ti}))$.
\end{itemize}
\item[2)] $X_1(x_0) \ne \emptyset$ and
$\mathcal{X}_{t-1} \subset \inte \left[  \Dom(X_{t}^0) \right]$
for every $t=2,\ldots,T$.\\
\end{itemize}

The following simple lemma states a few consequences of the above assumptions.

\begin{lemma} \label{lm:2.1}
The following statements hold:
\begin{itemize}
\item[(a)]
for every $t=1,\ldots,T$, $\mathcal{Q}_{t+1}$ is a convex function such that
\[
\mathcal{X}_{t} \subset \inte \left( \dom(\mathcal{Q}_{t+1}) \right);
\]
\item[(b)]
for every $t=1,\ldots,T$, $\mathcal{Q}_{t+1}$ is Lipschitz continuous on $\mathcal{X}_{t}$;
\item[(c)] for every $t=1,\ldots,T$, $i=1,\ldots,p$, 
and $(x_t,x_{t-1}) \in \mathcal{X}_{t} \small{\times} \mathcal{X}_{t-1}$,
\[
\partial f_t(x_t,x_{t-1}) \ne \emptyset, \quad \partial g_{ti}(x_t,x_{t-1})  \ne \emptyset;
\]
\item[(d)] for every $t=1,\ldots,T$, $i=1,\ldots,p$, the sets
\begin{align*}
\bigcup \left\{ \partial f_t(x_t,x_{t-1}) :  (x_t,x_{t-1}) \in \mathcal{X}_{t} \small{\times} \mathcal{X}_{t-1} \right\}, \quad
\bigcup \left\{ \partial g_{ti}(x_t,x_{t-1}) :  (x_t,x_{t-1}) \in \mathcal{X}_{t} \small{\times} \mathcal{X}_{t-1} \right\}
\end{align*}
are bounded.
\end{itemize}
\end{lemma}

\begin{proof}
(a)
The proof is by backward  induction on $t$. The result clearly holds for $t=T$ since
$\mathcal{Q}_{T+1} \equiv 0$.
Assume now that $\mathcal{Q}_{t+1}$ is a convex function such that
$\mathcal{X}_t \subset \inte \left( \dom(\mathcal{Q}_{t+1}) \right)$ for some $2 \le t \le T$.
Then, condition 1) of Assumption (H1) implies that the function
$(x_t,x_{t-1}) \mapsto F_t(x_t,x_{t-1}) + \delta_{X_t(x_{t-1})}(x_t)$ is convex. 
This conclusion together with the definition of $\mathcal{Q}_{t}$ and 
the discussion following Theorem 5.7 of \cite{Rockafellar70} then imply that $\mathcal{Q}_{t}$ is a convex function.
Moreover, conditions 1)b) and 2) of Assumption (H1) and relation \eqref{eq:simobs}  imply that there exists $\varepsilon>0$ such that for every
 $x_{t-1} \in  \mathcal{X}_{t-1} + \bar B(0,\varepsilon)$,
 \[
 \dom(f_t(\cdot,x_{t-1})) \supset \mathcal{X}_t, \quad X_t(x_{t-1}) \ne \emptyset.
 \]
The induction hypothesis, the latter observation, and relations \eqref{eq:Xt} and \eqref{eq:Qt},
 then imply that  
 \[
  X_{t}( x_{t-1} )  \cap \dom(F_{t}(\cdot,x_{t-1}))  = X_{t}( x_{t-1} )  \cap   \dom(f_{t}(\cdot,x_{t-1})) \cap  \dom(\mathcal{Q}_{t+1})
  \supset  X_{t}( x_{t-1} )  \cap \mathcal{X}_t = X_{t}( x_{t-1} ) \ne \emptyset
  \]
  for  every
$x_{t-1} \in  \mathcal{X}_{t-1} + \bar B(0,\varepsilon)$. Since by \eqref{eq:Qt},
\begin{align*}
\dom(\mathcal{Q}_{t}) &= \{ x_{t-1} \in \R^n : 
 X_{t}( x_{t-1} )  \cap \dom(F_{t}(\cdot,x_{t-1})) \ne \emptyset  \},
\end{align*}
we then conclude that $\mathcal{X}_{t-1} + \bar B(0,\varepsilon) \subset \dom(\mathcal{Q}_{t})$, and hence that
$\mathcal{X}_{t-1} \subset \inte (\dom(\mathcal{Q}_{t}))$. We have thus proved that (a) holds.

b) This statement follows from statement a) and  Theorem 10.4 of \cite{Rockafellar70}.

c-d) These two statements follow from conditions 1)a), 1)b) and 1)c) of Assumption (H1) together with
Theorem 23.4 and 24.7 of \cite{Rockafellar70}.
\end{proof}

\subsection{Forward DCuP}
Before formally describing the DCuP algorithm, we give some motivation for it.
At iteration $k \ge 1$ and stage $t=1,\ldots,T$, the algorithm uses the following approximation to the
function $\mathcal{Q}_{t}(\cdot)$ defined in \eqref{eq:Qt}:
\begin{equation}\label{approxqt}
{\underline{\mathcal{Q}}}_t^{k-1}  (x_{t-1}) = \min \left\{ f_t^{k-1}(x_t, x_{t-1} ) + \mathcal{Q}^{k-1}_{t+1} ( x_t ) : x_t \in X_t^{k-1} ( x_{t-1} ) \right \} 
\end{equation}
where
\begin{equation} \label{Xtkd}
X_t^{k-1} ( x_{t-1} ) =\{ x_t \in \mathcal{X}_t: \;g_t^{k-1}(x_t, x_{t-1}) \leq 0,\;\displaystyle A_{t} x_{t} + B_{t} x_{t-1} = b_t \}
\end{equation}
and
$f_t^{k-1}, g_t^{k-1}$, and $\mathcal{Q}_{t+1}^{k-1}$ are polyhedral functions minorizing
$f_t, g_t$ and $\mathcal{Q}_{t+1}$, respectively, i.e.,
\begin{equation} \label{eq:lowerfcts}
f_t^{k-1} \le f_t, \quad g_t^{k-1} \le g_t, \quad \mathcal{Q}^{k-1}_{t+1} \le  \mathcal{Q}_{t+1}.
\end{equation}
For $t=T+1$, we actually assume that $\mathcal{Q}_{T+1}^{k-1} \equiv 0$, and hence that
$\mathcal{Q}_{T+1}^{k} = \mathcal{Q}_{T+1}$. Moreover, we also assume that
${\underline{\mathcal{Q}}}_{T+1}^{k-1} \equiv 0$, and hence
${\underline{\mathcal{Q}}}_{T+1}^{k-1} = \mathcal{Q}_{T+1}$.


Observe that for every $k \ge 0$, $t=1,\cdots,T$, and
$x_{t-1} \in \mathcal{X}_{t-1}$, relations \eqref{Xtkd} and \eqref{eq:lowerfcts} imply that
\begin{equation} \label{??}
X_t(x_{t-1}) \subset X_t^{k} ( x_{t-1} ) \subset \mathcal{X}_{t}
\end{equation}
and 
$$
f_t^k(\cdot,x_{t-1}) + \mathcal{Q}_{t+1}^k(\cdot ) \leq f_t(\cdot, x_{t-1}) + \mathcal{Q}_{t+1}(\cdot),
$$
and hence that
\begin{equation}\label{lowboundqt}
{\underline{\mathcal{Q}}}_t^k \le \mathcal{Q}_t, \quad \forall \, t=1,2,\ldots,T,\,\,\forall \, k \geq 0.
\end{equation}
At iteration $k$, feasible points
$x_1^k,\ldots,x_T^k$ are computed recursively as follows:
for $t=1,\ldots,T$, $x_t^k$ is set to be an optimal solution of
subproblem \eqref{approxqt} with $x_{t-1}=x_{t-1}^k$ with the convention 
that $x_0^k = x_0$. These points in turn are used to
compute new affine functions minorizing $f_t$, $g_t$ and $\mathcal{Q}_{t}$ which
are then added to the bundle of affine functions describing
$f_t^{k-1}$, $g_t^{k-1}$, and 
$\mathcal{Q}_{t}^{k-1}$ to obtain new lower bounding approximations
$f_t^k, g_t^k$, and $\mathcal{Q}_{t}^k$ for
$f_t, g_t$ and $\mathcal{Q}_{t}$, respectively.

The precise description of DCuP algorithm is as follows.\\
\rule{\linewidth}{1pt}
\par {\textbf{DCuP (Dynamic Cutting Plane) with linearizations computed in a forward pass.}}\\
\rule{\linewidth}{1pt}
\par {\textbf{Step 0. Initialization.}} 
For every $t=1,\ldots,T$, let affine functions $f_t^0$ and $g_t^0$ such that $f_t^0 \leq f_t$ and $g_t^0 \leq g_t$,
and a piecewise linear function
$\mathcal{Q}^{0}_t : \mathcal{X}_{t-1} \to \mathbb{R}$ such that  $\mathcal{Q}_t^0 \leq \mathcal{Q}_t$
be given.  We write  $\mathcal{Q}_t^0$
as 
$\mathcal{Q}_t^0(x_{t-1})=\theta_t^0 +
\langle \beta_t^0 , x_{t-1} \rangle$,
set
$\mathcal{Q}_{T+1}^0 \equiv 0$, and $k=1$.\\
\par {\textbf{Step 1. Forward pass.}} 
Set $\mathcal{C}_{T+1}^k = \mathcal{Q}_{T+1}^k \equiv 0$ and
$x_0^{k} = x_0$. For $t=1,2,\ldots,T$, do:
\begin{itemize}
\item[a)]
find an optimal solution $x_t^{k}$ of
\begin{equation}\label{forreg10}
{\underline{\mathcal{Q}}}_t^{k-1} ( x_{t-1}^{k} )   = 
\left\{
\begin{array}{l}
\displaystyle \inf_{x_t \in \mathbb{R}^n}  f_t^{k-1}(x_t, x_{t-1}^{k} ) + \mathcal{Q}^{k-1}_{t+1}( x_t )\\
x_t \in X_t^{k-1}( x_{t-1}^{k}  ),
\end{array}
\right.
\end{equation}
where $X_t^{k}(\cdot)$ is as in \eqref{Xtkd};
\item[b)]
compute function values and subgradients
of $f_t$ and $g_{t i}$, $i=1,\ldots,p$,  at $( x_t^{k}, x_{t-1}^{k} )$,
and let 
$\ell_{f_t}(\cdot; (x_t^{k}, x_{t-1}^{k} ) ) $ and $\ell_{g_{t i}}(\cdot; (x_t^{k}, x_{t-1}^{k} ))$
denote the corresponding linearizations;
\item[c)]
set
\begin{align}
f_t^{k} &= \max\Big( f_t^{k-1} \,,\, \ell_{f_t}\left((\cdot,\cdot); ( x_t^{k}, x_{t-1}^{k}   ) \right ) \Big), \label{eq:ftk0}\\
g_{ti}^{k} &= \max\Big( g_{t i}^{k-1} \,,\, \ell_{g_{t i}}\left((\cdot,\cdot); (x_t^{k}, x_{t-1}^{k} ) \right) \Big), \quad \forall i=1,\ldots,p, \label{eq:gtk0}
\end{align}
and define $g_t^{k} := ( g_{t 1}^{k},\ldots,g_{t p}^{k} )$;
\item[d)]
if $t \ge 2$, then compute $\beta_t^{k} \in \partial {\underline{\mathcal{Q}}}_t^{k-1} ( x_{t-1}^{k}  )$
and denote the corresponding linearization of ${\underline{\mathcal{Q}}}_t^{k-1}$ as
\[
\mathcal{C}_t^{k} (\cdot ) := {\underline{\mathcal{Q}}}_t^{k-1} ( x_{t-1}^{k} ) + 
\langle \beta_t^{k} ,  \cdot - x_{t-1}^{k} \rangle;
\]
moreover, set
\beq \label{eq:Qtk}
\mathcal{Q}_{t}^{k} = \max\{ \mathcal{Q}^{k-1}_{t}, \mathcal{C}_t^{k}  \};
\eeq
\end{itemize}
\par {\textbf{Step 2.}} Set $k \leftarrow k+1$ and go to Step 1.\\
\rule{\linewidth}{1pt}

We now make a few remarks about DCuP.
First, Lemma \ref{lm:2.1}(c) guarantees  the existence of the subgradients,
and hence the linearizations, of
the functions $f_t$ and $g_{ti}$, $i=1,\ldots,p$, at any point
$(x_t,x_{t-1}) \in \mathcal{X}_t \times \mathcal{X}_{t-1}$, and hence that the functions
$f_t^{k}$ and $g_t^{k}$ in Step 1 are well-defined.
Second, in view of the definition of $x_t^{k}$ in Step a), we have that
$x_t^{k} \in X_t^{k-1}( x_{t-1}^{k}  ) \subset \mathcal{X}_{t}$ for every $t=1,\ldots, T$ and
$k \ge 0$.
Third, Lemma \ref{lm:2.2}(b) below and the previous remark guarantee
the existence of the subgradient $\beta_t^{k}$ in Step d).
Fourth, we dicuss in Subsection \ref{subsection:subg} ways of computing  this subgradient.


In the remaining part of this subsection, we provide the convergence analysis of DCuP.
The following result states some  basic properties about the functions involved in DCuP.

\begin{lemma} \label{lm:2.2}
The following statements hold:
\begin{itemize}
\item[(a)]
 for every $k\ge 1$ and $t=1,\ldots,T$, we have
\begin{align}
& f_t^{k} \le f_t^{k+1} \le f_t, \quad  g_t^{k} \le g_t^{k+1} \le g_t, \quad 
 \label{l2.2a}  \\
&X_t(x_{t-1}) \subset X_t^{k+1} ( x_{t-1} ) \subset X_t^{k} ( x_{t-1} ) \subset \mathcal{X}_{t-1}
\quad \forall x_{t-1} \in \R^n,  \label{l2.2b}  \\
&\mathcal{Q}_{t+1}^{k} \le \mathcal{Q}_{t+1}^{k+1} \le \mathcal{Q}_{t+1}, \label{l2.2c} \\
&{\underline{\mathcal{Q}}}_t^{k} \le {\underline{\mathcal{Q}}}_t^{k+1} \le \mathcal{Q}_t.  \label{l2.2d} 
\end{align}
\item[(b)]
For every $k \ge 1$ and $t=2,\ldots,T$, function ${\underline{\mathcal{Q}}}_{t}^{k}$ is convex and
$\inte(\dom({\underline{\mathcal{Q}}}_{t}^k))
\supset \mathcal{X}_{t-1}$; 
as a consequence, $\partial {\underline{\mathcal{Q}}}_{t}^k(x_{t-1}) \ne \emptyset$
 for every $x_{t-1} \in \mathcal{X}_{t-1}$. 
\end{itemize}
\end{lemma}

\begin{proof}
(a)
Relations \eqref{l2.2a} and \eqref{l2.2b}  follow immediately from
 the initialization of DCuP described in step 0,
the recursive definitions of
$f_t^{k}$ and $g_t^k$ 
in \eqref{eq:ftk0} and \eqref{eq:gtk0},
respectively, the definition of $X_t^k(\cdot)$ in \eqref{Xtkd}, 
and the fact that 
\[
\ell_{f_t}((\cdot,\cdot); (x_t^{k}, x_{t-1}^{k} ) ) \le f_t(\cdot,\cdot),\
\quad \ell_{g_{t i}}((\cdot.\cdot); (x_t^{k}, x_{t-1}^{k} )) \le g_{t i}(\cdot.\cdot).
\]
Next note that the inequalities in \eqref{l2.2d}  follow immediately from the respective ones in \eqref{l2.2a},
\eqref{l2.2b} and \eqref{l2.2c}, together with
 relations  \eqref{eq:Qt} and \eqref{forreg10}.
 It then remains to show that the inequalities in \eqref{l2.2c} hold.
 Indeed,  the inequality $\mathcal{Q}_{t+1}^{k} \le \mathcal{Q}_{t+1}^{k+1}$
 follows immediately from \eqref{eq:Qtk} with $t=t+1$.
We will now show that  inequalities $\mathcal{Q}_{t}^{k} \le \mathcal{Q}_{t}$ for every $t=2,\ldots,T+1$
implies that
$\mathcal{Q}_{t}^{k+1} \le \mathcal{Q}_{t}$ for every $t=2,\ldots,T+1$,
and hence that the second inequality in \eqref{l2.2c} follows
from a simple inductive argument on $k$.
Indeed, first observe that the inequality $\mathcal{Q}_{t+1}^{k} \le \mathcal{Q}_{t+1}$ implies that
${\underline{\mathcal{Q}}}_{t}^{k} \le {\mathcal{Q}}_{t}$.
Next observe that the construction of $\mathcal{C}_t^{k+1}$ in Step d) of DCuP implies that
 $\mathcal{C}_t^{k+1} \le {\underline{\mathcal{Q}}}_t^{k}$, and hence that
 $\mathcal{C}_t^{k+1} \le {\mathcal{Q}}_t$. It then follows from \eqref{eq:Qtk} and the inequality
 $\mathcal{Q}_{t}^{k} \le \mathcal{Q}_{t}$  that $\mathcal{Q}_{t}^{k+1} \le \mathcal{Q}_{t}$.
 We have thus shown that $\mathcal{Q}_{t}^{k} \le \mathcal{Q}_{t}$ for every $t=2,\ldots,T+1$
 implies that $\mathcal{Q}_{t}^{k+1} \le \mathcal{Q}_{t}$ for every $t=2,\ldots,T$. Since the latter inequality for
 $t=T+1$ is straightforward and $\mathcal{Q}_{t}^{0} \le \mathcal{Q}_{t}$ for $t=2,\ldots,T$, \eqref{l2.2c} follows.
 

(b) 
The assertion that ${\underline{\mathcal{Q}}}_{t}^{k}$ is a convex function follows from the fact that
$\mathcal{Q}_{t+1}^k$ is convex and the same arguments used in Lemma \ref{lm:2.1} to show that
$\mathcal{Q}_t$ is convex. The assertion that
$\dom(\underline{\mathcal{Q}}_{t}^k) \supset \mathcal{X}_{t-1}$
follows from the fact that by \eqref{l2.2d} we have $\underline{\mathcal{Q}}_{t}^k \le \mathcal{Q}_t$, and hence
that 
\[
\inte \left( \dom(\underline{\mathcal{Q}}_{t}^k) \right) \supset \inte \left(\dom(\mathcal{Q}_t) \right) \supset \mathcal{X}_{t-1},
\]
where the last inclusion is due to Lemma \ref{lm:2.1}(a).
\end{proof}

The following technical result is useful to establish uniform Lipschitz continuity of convex functions.
\begin{lemma} \label{lm:2.3}
Assume that $\phi^-$ and $\phi^+$ are proper convex functions such that $\phi^- \le \phi^+$.
Then, for any nonempty compact set $K \subset \inte ( \dom(\phi^+))$, there exists a scalar $L=L(K) \ge 0$ satisfying the following property:
any convex function $\phi$ such that $\phi^- \le \phi \le \phi^+$ is $L$-Lipschitz continuous on $K$.
\end{lemma}

\begin{proof}
Let $\phi$ be a convex function such that $\phi^- \le \phi \le \phi^+$ and let
$K \subset \inte ( \dom(\phi^+))$ be a nonempty compact set.
Since $\phi^-$ and $\phi^+$ are  proper, it then follows that $\phi$ is proper and $\dom(\phi) \supset \dom(\phi^+)$, and hence that
$\inte(\dom(\phi^{-}) ) \supset \inte(\dom(\phi) )  \supset \inte(\dom(\phi^+) ) \supset K$.
Hence, in view of Theorem 23.4 of \cite{Rockafellar70}, we conclude that
$\partial \phi(x) \ne \emptyset$ for every $x \in K$. We now claim that there exists $L$ such that
$\|\beta\| \le L$ for every $\beta \in \partial \phi(x)$ and $x \in K$. This claim in turn can be easily seen to imply that the conclusion of
the lemma holds. To show the claim, let $x \in K$ and $0 \ne \beta \in \partial \phi(x)$ be given.
The inclusion $K \subset \inte ( \dom(\phi^+))$ implies the existence of $\varepsilon >0$ such that
$K_\varepsilon := K + \bar B(0;\varepsilon) \subset \inte(\dom(\phi^+))$. Let
\[
y_\varepsilon:= x+ \varepsilon \frac{\beta}{\|\beta\|}, \quad \theta^+ := \max_{y \in K_\varepsilon} \phi^+(y),
\quad \theta^-:=\min_{y \in K} \phi^-(y).
\]
Clearly, $y_\varepsilon \in K_\varepsilon$ due to the definition of $K_\varepsilon$ and the facts that $x \in K$ and $\|y_\varepsilon-x\| \le \varepsilon$.
Moreover, using the fact that every proper convex function is continuous in the interior of its domain, we then
conclude that the proper convex functions $\phi^+$ and $\phi^-$ are continuous on $K_\varepsilon$ and $K$, respectively,
since these two sets lie in the interior of their domains, respectively.
Hence, it follows from Weierstrass' theorem that $\theta^+$ and $\theta^-$
are both finite due to the compactness of $K$ and $K_\varepsilon$,
respectively.
Using the facts that $x \in K$, $y_\varepsilon \in K_\varepsilon$,
$\beta \in \partial \phi(x)$ and $\phi^+ \ge \phi$, the definitions of $\theta^+$ and $\theta^-$, and the definition of subgradient,
it then follows that
\begin{align*}
\theta^+ \ge \phi^+(y_\varepsilon) \ge \phi(y_\varepsilon) \ge \phi(x) + \inner{\beta}{y_\varepsilon-x} = \phi(x) + \varepsilon \|\beta\| \ge \theta^-  + \varepsilon \|\beta\|
\end{align*}
and hence that the claim holds with $L=(\theta^+-\theta^-)/\varepsilon$.
\end{proof}

\begin{lemma}\label{lipcontQt}
The following statements hold:
\begin{itemize}
\item[(a)] 
For each $t=2,\ldots,T$, there exist $L_t \ge 0$ such that
the functions $\mathcal{Q}^k_{t}$ and $\underline{\mathcal{Q}}_t^k$
 are $L_t$-Lipschitz continuous on $\mathcal{X}_{t-1}$ for every $k \ge 1$;
 \item[(b)] For each $t=1,\ldots,T$, there exist $\hat L_t\ge 0$ such that
 the functions $f_t^k$ and $g_{t i}^k$ are $\hat L_t$-Lipschitz continuous functions on $\mathcal{X}_t \times \mathcal{X}_{t-1}$
 for every $k \ge 1$ and $i=1,\ldots,p$.
 \end{itemize}
\end{lemma}

\begin{proof}
Let $t \in \{2,\ldots,T\}$ be given. The existence of $L_t$ satisfying (a) follows
from Lemmas \ref{lm:2.1} and \ref{lm:2.2}, and applying Lemma \ref{lm:2.3} twice, the first time  with
$K=\mathcal{X}_{t-1}$, $\phi^+=\mathcal{Q}_t$ and $\phi^-=\mathcal{Q}_t^0$, and the second time
with $K=\mathcal{X}_{t-1}$, $\phi^+=\mathcal{Q}_t$ and $\phi^-=\underline{\mathcal{Q}}_t^0$.
Moreover, the existence of $\hat L_t$ satisfying (b) follows
from Lemma \ref{lm:2.2}, and applying Lemma \ref{lm:2.3} twice, the first time with
$K=\mathcal{X}_t \times \mathcal{X}_{t-1}$, $\phi^+=f_t$ and $\phi^-=f_t^0$, and the second time
with $K=\mathcal{X}_t \times \mathcal{X}_{t-1}$, $\phi^+=g_{ti}$ and $\phi^-=g_{ti}^0$ for $i=1,\ldots,p$.
\end{proof}

We now state a result whose proof is given in Lemma 5.2 of  \cite{lecphilgirar12}.
Even though the latter result assumes convexity of the functions involved in its statement,
its proof does not make use of this assumption.
For this reason, we state the result here in a slightly more general way than it is stated in Lemma 5.2 of  \cite{lecphilgirar12}.


\begin{lemma}\label{techlemmasequence}
Lemma 5.2 in \cite{lecphilgirar12}.
Assume that $Y \subset \R^n$ is a compact set, $f : \R^n \to (-\infty,\infty]$ is a function and $\{f_k: \R^n \to (\infty,\infty]\}_{k=1}^\infty$ is a sequence of functions such that, for some integer $k_0>0$ and scalar $L>0$, we have:
\begin{itemize}
\item[(a)]
$
f^{k-k_0}(y) \leq f^{k}(y) \leq f(y)<\infty$ for every $k \ge k_0 + 1$ and $ y \in Y$;
\item[(b)]
$f^k$ is $L$-Lipschitz continuous on $Y$ for every $k \ge 1$.
\end{itemize}
Then, for any infinite sequence $\{y^k\}_{k=1}^\infty \subset Y$, we have
$$
\lim_{k \rightarrow +\infty}[f(y^k)-f^k(y^k)]=0 \Longleftrightarrow \lim_{k \rightarrow +\infty}[f(y^k)-f^{k-k_0}(y^k)]=0.
$$
\end{lemma}

We are now ready to provide the main result of this subsection, i.e., the convergence analysis of DCuP.

\begin{thm}\label{convprooftheorem}  Let Assumption (H1) hold. Define
{\small{
$$
\mathcal{H}(t)
\left\{
\begin{array}{ll}
(i) &\displaystyle \varlimsup_{k \rightarrow +\infty} g_{t i}(x_t^{k}, x_{t-1}^{k} )  \le 0, \quad i=1,\ldots,p,\\
(ii)& \displaystyle \lim_{k \rightarrow + \infty} \mathcal{Q}_t( x_{t-1}^{k} ) - {\underline{\mathcal{Q}}}_t^{k-1}  (x_{t-1}^{k} )=
\displaystyle \lim_{k \rightarrow + \infty} \mathcal{Q}_t( x_{t-1}^{k} ) - {\underline{\mathcal{Q}}}_t^k  (x_{t-1}^{k} )=0,\\
(iii) & \displaystyle \lim_{k \rightarrow + \infty} \mathcal{Q}_t( x_{t-1}^{k} ) - \sum_{\tau=t}^T f_{\tau}(x_{\tau}^{k}, x_{\tau-1}^{k} ) =0,\\
(iv)& 
\displaystyle \lim_{k \rightarrow + \infty} \mathcal{Q}_t( x_{t-1}^{k} ) - \mathcal{Q}_t^k  (x_{t-1}^{k} ) =0. 
\end{array}
\right.
$$
}}
Then $\mathcal{H}(t)$-(i) holds for $t=1,\ldots,T$,
$\mathcal{H}(t)$-(ii),(iii) hold for $t=1,\ldots,T+1$,
and $\mathcal{H}(t)$-(iv) holds for $t=2,\ldots,T+1$. 
In particular, the limit of the sequence of upper bounds $(\sum_{t=1}^T f_t(x_t^{k} , x_{t-1}^{k} ) )_{k \geq 1}$ 
and of lower bounds 
$({\underline{\mathcal{Q}}}_1^{k-1}  (x_{0} ))_{k \geq 1}$
is the optimal value
$\mathcal{Q}_1( x_0 ) $ of \eqref{defpb} and 
any accumulation point of the sequence $(x_1^{k},\ldots,x_T^{k})$ is an optimal solution to \eqref{defpb}.
\end{thm}
\begin{proof} We first  prove $\mathcal{H}(t)$-(i) for $t=1,\ldots,T$.
Let $t \in \{1,\ldots,T\}$ be given and define  the sequence $\{y_t^k\}$
as $y_t^{k}=(x_t^{k} , x_{t-1}^{k})$ for every $k \geq 1$.
In view of Lemma \ref{lm:2.2},
we have 
$g_{t i}(y_t^k) \geq g_{t i}^{k}(y_t^k) \geq \ell_{g_{t i}}(y_t^k ; y_t^k )=g_{t i}(y_t^k)$, and hence
\begin{equation}\label{tangeantefhd}
g_t^{k} ( y_t^k   )  = g_t( y_t^k ), \;\forall \;k \geq 1.
\end{equation}
Due to Lemma \ref{lipcontQt}-(b), functions $g_{t i}^k$ are convex 
$\hat L_t$-Lipschitz continuous on $\mathcal{X}_t \small{\times} \mathcal{X}_{t-1}$.
Therefore, recalling \eqref{tangeantefhd}, we can apply Lemma \ref{techlemmasequence} to 
$f=g_{t i}$, $f^k=g_{t i}^{k}$, $y^k=y_t^k$, $Y=\mathcal{X}_t \small{\times} \mathcal{X}_{t-1}$ for $i=1,\ldots,p$,
to obtain 
\begin{equation}\label{limghTd}
\begin{array}{l}
\lim_{k \rightarrow +\infty} g_{t}( x_t^{k} , x_{t-1}^{k} )-g_t^{k-1}( x_t^{k} , x_{t-1}^{k} )=0.\\
\end{array}
\end{equation}
The latter conclusion together with the fact that $x_t^{k} \in X_t^{k-1}(x_{t-1}^{k} )$,
and hence  $g_{t}^{k-1}( x_t^{k} ,x_{t-1}^{k} ) \leq 0$, for every $k \ge 1$, then implies that $\mathcal{H}(t)$-(i) holds.

Let us now show  $\mathcal{H}(1)$-(ii), (iii) and
$\mathcal{H}(t)$-(ii)-(iii), (iv) for $t=2,\ldots,T+1$
by backward  induction on $t$. $\mathcal{H}(T+1)$-(ii), (iii), (iv) trivially holds.
Now,  fix $t \in \{1,\ldots,T\}$ and assume that $\mathcal{H}(t+1)$-(ii), (iii), (iv) holds.
We will show
that $\mathcal{H}(t)$-(ii), (iii) holds and that $\mathcal{H}(t)$-(iv) holds if $t \geq 2$.
Indeed,
since $f_t \geq f_t^{k} \geq \ell_{f_t}(\cdot ; y_t^{k} )$ and $f_t(y^k_t) = \ell_{f_t}(y_t^k ; y_t^{k} )$, we conclude that
$f_t^{k} (y_t^k  )  = f_t( y_t^k )$ for every $k \ge 1$,
and hence that $\lim_{k \rightarrow +\infty} f_t( y_t^k )-f_t^{k}( y_t^k )=0$.
Recalling by Lemma \ref{lipcontQt}-(b) that $f_t^k$ is $\hat L_t$-Lipschitz continuous on
$\mathcal{X}_{t} \small{\times} \mathcal{X}_{t-1}$ and using Lemma \ref{techlemmasequence}
with $f=f_t$, $f^k=f_t^{k}$, $(y^k)=( y_t^k )$, and
$Y=\mathcal{X}_{t} \small{\times} \mathcal{X}_{t-1}$, we conclude that
\begin{equation}\label{limftd}
\lim_{k \rightarrow +\infty} f_t( x_t^{k} , x_{t-1}^{k} )-f_t^{k-1}( x_t^{k} , x_{t-1}^{k} )=0.
\end{equation}
Moreover, by the induction hypothesis $\mathcal{H}(t+1)$-(iv), we have
$
\lim_{k \rightarrow +\infty}  \mathcal{Q}_{t+1}^{k}(  x_t^{k}  ) - \mathcal{Q}_{t+1}( x_{t}^{k} ) = 0.
$
Recalling by Lemma \ref{lipcontQt}-(a) that functions 
$\mathcal{Q}_t^k$ are $L_t$-Lipschitz continuous on $\mathcal{X}_{t-1}$, we can use
Lemma \ref{techlemmasequence} with $k_0=1$, $f=\mathcal{Q}_{t+1}, f^k=\mathcal{Q}_{t+1}^k, y^k=x_{t}^{k}$ and 
$Y=\mathcal{X}_t$, to obtain 
\begin{equation} \label{eq:Qktt}
\lim_{k \rightarrow +\infty}  \mathcal{Q}_{t+1}^{k-1}(  x_t^{k}  ) - \mathcal{Q}_{t+1}( x_{t}^{k} ) = 0.
\end{equation}
Now, using Lemma \ref{lm:2.2}, we easily see that the objective function $f_{t}^{k-1}(\cdot , x_{t-1}^{k} )  +  \mathcal{Q}_{t+1}^{k-1}( \cdot)$ and feasible region $ X^{k-1}_t( x_{t-1}^{k} )$ of \eqref{forreg10}
satisfies $ f_{t}^{k-1}(\cdot , x_{t-1}^{k} )  +  \mathcal{Q}_{t+1}^{k-1}( \cdot) \le F_t(\cdot,x_{t-1}^k) $
and $X_t^{k-1} (x_{t-1}^{k} ) \supseteq X_t( x_{t-1}^{k} )$.  Since $x_t^k$ is an optimal solution of \eqref{forreg10} and $\mathcal{Q}_t( x_{t-1}^{k} )$
is the optimal value of $\min \{ F_t(x_t,x_{t-1}^k): x_t \in X_t( x_{t-1}^{k} )\}$ due to \eqref{eq:Qt}, we then conclude that
$f_{t}^{k-1}(x_t^{k} , x_{t-1}^{k} )  +  \mathcal{Q}_{t+1}^{k-1}(  x_t^{k}  ) \leq \mathcal{Q}_t( x_{t-1}^{k} )$. Hence, we conclude that
\begin{align*}
0 & \ge \varlimsup_{k \rightarrow +\infty}  f_{t}^{k-1}(x_t^{k} , x_{t-1}^{k} )  +  \mathcal{Q}_{t+1}^{k-1}(  x_t^{k}  ) - \mathcal{Q}_t( x_{t-1}^{k} ) 
 = \varlimsup_{k \rightarrow +\infty}  f_{t}(x_t^{k} , x_{t-1}^{k} )  +  \mathcal{Q}_{t+1}(  x_t^{k}  ) - \mathcal{Q}_t( x_{t-1}^{k} )
\end{align*}
where the equality is due to \eqref{limftd} and \eqref{eq:Qktt}.
We now claim that
\begin{equation}\label{anteschaintd}
\lim_{k \rightarrow +\infty}  f_{t}(x_t^{k} , x_{t-1}^{k} )  +  \mathcal{Q}_{t+1}(  x_t^{k}  ) - \mathcal{Q}_t( x_{t-1}^{k} )=0. 
\end{equation}
Indeed, assume by contradiction that the above claim does not hold. Then, it follows from the last conclusion before the claim that
\begin{equation}\label{liminftsupd}
\begin{array}{lll}
\varliminf_{k \rightarrow +\infty} \;  f_{t}(x_t^{k} , x_{t-1}^{k} )  +  \mathcal{Q}_{t+1}(  x_t^{k}  ) - \mathcal{Q}_t( x_{t-1}^{k} ) < 0.
\end{array}
\end{equation}
Since
$\{(x_t^{k} , x_{t-1}^{k})\}$ is a sequence lying in the compact set $\mathcal{X}_t \small{\times} \mathcal{X}_{t-1}$,
it has a subsequence $\{(x_t^{k} , x_{t-1}^{k})\}_{k \in K}$ converging to some $(x_t^*, x_{t-1}^*) \in \mathcal{X}_t \small{\times} \mathcal{X}_{t-1}$.
Hence, in view of $\mathcal{H}(t)$-(i), \eqref{liminftsupd}, and the fact that $f_t$ and $g_t$ are lower semi-continuous on $\mathcal{X}_t \small{\times} \mathcal{X}_{t-1}$ and
$\mathcal{Q}_t$ (resp. $\mathcal{Q}_{t+1}$) is lower semi-continuous on $\mathcal{X}_{t-1}$ (resp. $\mathcal{X}_t$), we conclude that
$$
 g_t(x^*_t,x^*_{t-1}) \le 0, \quad f_{t}(x_t^{*} , x_{t-1}^{*} )  +  \mathcal{Q}_{t+1}(  x_t^{*}  ) -\mathcal{Q}_t( x_{t-1}^* )< 0
$$
and hence that $x_t^* \in X_t( x_{t-1}^* )$ (recall that $\mathcal{X}_t$ is closed) and $F_t(x_t^{*} , x_{t-1}^{*} )< \mathcal{Q}_t( x_{t-1}^* )$ due to the
definition of $X_t$ and $F_t$ in \eqref{eq:Xt} and \eqref{eq:Qt}, respectively.
Since this contradicts the definition of 
$\mathcal{Q}_t$ in \eqref{eq:Qt},
the above claim follows.
Combining 
$$
\begin{array}{lll}
0 \leq \mathcal{Q}_t( x_{t-1}^{k} ) - {\underline{\mathcal{Q}}}_t^{k} ( x_{t-1}^{k} ) &  \leq & \mathcal{Q}_t( x_{t-1}^{k} ) - {\underline{\mathcal{Q}}}_t^{k-1} ( x_{t-1}^{k} ),\\
&= & \mathcal{Q}_t( x_{t-1}^{k} ) - f_t^{k-1} (x_t^{k} , x_{t-1}^{k} )-\mathcal{Q}_{t+1}^{k-1}(x_t^{k} ) \mbox{ [by definition of }x_t^k]
\end{array}
$$
with relations \eqref{limftd}, \eqref{eq:Qktt}, \eqref{anteschaintd} we obtain $\lim_{k \rightarrow +\infty} \mathcal{Q}_t( x_{t-1}^{k} ) - {\underline{\mathcal{Q}}}_t^{k} ( x_{t-1}^{k} ) = 0$.
Also observe that 
$$
\begin{array}{l}
\displaystyle \lim_{k \rightarrow +\infty} \mathcal{Q}_t(x_{t-1}^{k} )  - \sum_{\tau=t}^T f_\tau(x_\tau^{k}, x_{\tau-1}^{k} )\\
 = \underbrace{\displaystyle \lim_{k \rightarrow +\infty} \mathcal{Q}_t(x_{t-1}^{k} )  - f_t(x_t^k, x_{t-1}^k) - \mathcal{Q}_{t+1}( x_t^k )}_{=0 \mbox{ by }\eqref{anteschaintd}} + 
\underbrace{\displaystyle \lim_{k \rightarrow +\infty}  \mathcal{Q}_{t+1}( x_t^k ) - \sum_{\tau=t+1}^T f_\tau(x_\tau^{k}, x_{\tau-1}^{k} )}_{=0 \mbox{ using }\mathcal{H}(t+1)-(iii)},\\
=0,
\end{array}
$$
and we have shown $\mathcal{H}(t)$-(ii),(iii).

Finally, if $t \geq 2$, $\mathcal{H}(t)$-(iv) follows from
\begin{eqnarray*}
0 \leq \mathcal{Q}_t( x_{t-1}^{k} ) - \mathcal{Q}_t^k ( x_{t-1}^{k} )  & \leq  &   \mathcal{Q}_t( x_{t-1}^{k} ) - \mathcal{C}_t^k ( x_{t-1}^{k} )\mbox{ since }\mathcal{Q}_t^k \geq \mathcal{C}_t^k,\\
& = & \mathcal{Q}_t( x_{t-1}^{k} ) - f_t^{k-1} (x_t^{k} , x_{t-1}^{k} )-\mathcal{Q}_{t+1}^{k-1}(x_t^{k} ) \mbox{ [by definition of }x_t^k]
\end{eqnarray*}
combined with relations \eqref{limftd}, \eqref{eq:Qktt}, \eqref{anteschaintd}. 
\hfill
\end{proof}

\if{

\subsection{Forward-Backward DCuP}

It is also possible to have for each iteration both a forward and backward
pass and compute cuts for $\mathcal{Q}_t$ in the backward passes and cuts
for $f_t, g_t$ in both the forward and backward passes. The corresponding
extension of the algorithm is given below and the convergence of this variant
of DCuP is given in Theorem \ref{convproofstodcupbf}.\\
\rule{\linewidth}{1pt}
\par {\textbf{Forward-Backward DCuP (Dynamic Cutting Plane) with linearizations computed in  forward and backward passes.}}\\
\rule{\linewidth}{1pt}
\par {\textbf{Step 0. Initialization.}} Let $\mathcal{Q}^{0}_t : \mathcal{X}_{t-1} \to \mathbb{R},\,t=2,\ldots,T+1$, be affine functions 
satisfying 
$\mathcal{Q}_t^0 \leq \mathcal{Q}_t, t=2,\ldots,T,$
and $\mathcal{Q}_{ T+1}^0 \equiv 0$, and let $f_t^0, g_t^0: \mathcal{X}_{t}\small{\times}\mathcal{X}_{t-1} \to \mathbb{R},\,t=1,\ldots,T$, be 
affine functions such that $f_t^0 \leq f_t, g_t^0 \leq g_t$. Set $k=1$.\\
\par {\textbf{Step 1. Forward pass.}} Setting $x_0^{2k-1} = x_0$, for $t=1,2,\ldots,T$,
compute an optimal solution $x_t^{2k-1}$ of
\begin{equation}\label{forreg}
\left\{
\begin{array}{l}
\displaystyle \min_{x_t}  \;f_t^{2k-2}(x_t, x_{t-1}^{2k-1} ) + \mathcal{Q}^{k-1}_{t+1}( x_t )\\
x_t \in X_t^{2k-2}( x_{t-1}^{2k-1}  ),
\end{array}
\right.
\end{equation}
where $X_t^{2k-2}$ is given by \eqref{Xtkd} with $k-1$
replaced by $2k-2$.
Compute $f_t(x_t^{2k-1}, x_{t-1}^{2k-1} )$, $g_t(x_t^{2k-1}, x_{t-1}^{2k-1})$, 
and subgradients
of $f_t$, $g_{t i}$, $i=1,\ldots,p$, at $( x_t^{2k-1}, x_{t-1}^{2k-1} )$ 
with corresponding
linearizations $\ell_{f_t}(\cdot; (x_t^{2k-1}, x_{t-1}^{2k-1} ) ) $ and $\ell_{g_{t i}}(\cdot; (x_t^{2k-1}, x_{t-1}^{2k-1} )) $.
Define
\begin{eqnarray}
f_t^{2k-1} &=& \max\Big( f_t^{2k-2}, \ell_{f_t}(\cdot; ( x_t^{2k-1}, x_{t-1}^{2k-1}   ) ) \Big), \label{eq:ftk0bb}\\
g_t^{2k-1} &=& (g_{t 1}^{2k-1},\ldots,g_{t p}^{2k-1})\mbox{ where } g_{t i}^{2 k-1} = \max\Big( g_{t i}^{2k-2},  \ell_{g_{t i}}(\cdot; (x_t^{2k-1}, x_{t-1}^{2k-1} )) \Big),i=1,\ldots,p. \label{eq:gtk0bb}
\end{eqnarray}
\par {\textbf{Step 2. Backward pass.}} For $t=T, T-1, \ldots,1$,
solve the problem 
\begin{equation}\label{backpassd} 
{\underline{\mathcal{Q}}}_t^{k} ( x_{t-1}^{2k-1} ) := \left\{
\begin{array}{l}
\displaystyle \min_{x_t}  \;f_t^{2k-1}(x_t, x_{t-1}^{2k-1} ) + \mathcal{Q}^{k}_{t+1}( x_t )\\
x_t \in X_t^{2k-1}( x_{t-1}^{2k-1}  ).
\end{array}
\right.
\end{equation}
\par Denoting by $x_t^{2k}$ an optimal solution of \eqref{backpassd}, 
compute $f_t( x_t^{2k}, x_{t-1}^{2k-1} )$, $g_t( x_t^{2k}, x_{t-1}^{2k-1})$,
and subgradients
of $f_t$ and $g_{t i}$, $i=1,\ldots,p$, at $( x_t^{2 k}, x_{t-1}^{2k-1} )$,
with corresponding
linearizations 
$
\ell_{f_t}(\cdot; ( x_t^{2 k}, x_{t-1}^{2k-1} ) )$, $\ell_{g_{t i}}(\cdot; ( x_t^{2k}, x_{t-1}^{2k -1} )).
$
Define
\begin{eqnarray}
f_t^{2k} &=& \max\Big( f_t^{2k-1}, \ell_{f_t}(\cdot; ( x_t^{2k}, x_{t-1}^{2k-1}   ) ) \Big), \label{eq:ftk0b}\\
g_t^{2k} &=& (g_{t 1}^{2k},\ldots,g_{t p}^{2k})\mbox{ where } g_{t i}^{2 k} = \max\Big( g_{t i}^{2k-1},  \ell_{g_{t i}}(\cdot; (x_t^{2k}, x_{t-1}^{2k-1} )) \Big),i=1,\ldots,p. \label{eq:gtk0b}
\end{eqnarray}
\par If $t \geq 2$, take a subgradient $\beta_t^k$ of ${\underline{\mathcal{Q}}}_t^k (  \cdot )$ at $x_{t-1}^{2k-1}$, and store the new cut
\[
\mathcal{C}_t^k ( x_{t-1} ) := {\underline{\mathcal{Q}}}_t^k ( x_{t-1}^{2k-1} ) + \langle \beta_t^k ,  x_{t-1} - x_{t-1}^{2k-1} \rangle
\]
for $\mathcal{Q}_t$,
making up the new approximation
$\mathcal{Q}_{t}^k = \max\{ \mathcal{Q}^{k-1}_{t}, \mathcal{C}_t^k  \}$.
\par {\textbf{Step 4.}} Do $k \leftarrow k+1$ and go to Step 1.\\
\rule{\linewidth}{1pt}

\begin{thm}\label{convproofstodcupbf} 
Let Assumption (H1) hold. Define
{\small{
$$
\mathcal{H}(t)
\left\{
\begin{array}{ll}
(i) & \lim_{k \rightarrow +\infty} \max( g_t(x_t^{2k-1}, x_{t-1}^{2k-1} ) , 0 )  = 0,\; \lim_{k \rightarrow +\infty} \max( g_t(x_t^{2k}, x_{t-1}^{2k-1} ) , 0 )  = 0,\\ 
(ii)&\lim_{k \rightarrow + \infty} \mathcal{Q}_t( x_{t-1}^{2k-1} ) - {\underline{\mathcal{Q}}}_t^k  (x_{t-1}^{2k-1} ) =0,\\
(iii)    &  \lim_{k \rightarrow + \infty} \mathcal{Q}_t( x_{t-1}^{2k-1} ) - \sum_{\tau=t}^T f_{\tau}(x_{\tau}^{2k-1}, x_{\tau-1}^{2k-1} ) =0,\\
(iv)& \lim_{k \rightarrow + \infty} \mathcal{Q}_t( x_{t-1}^{2k-1} ) -  \mathcal{Q}_t^{k} ( x_{t-1}^{2k-1} ) =0.
\end{array}
\right.
$$
}}
Then $\mathcal{H}(t)$-(i) holds for $t=1,\ldots,T$,
$\mathcal{H}(t)$-(ii),(iii) hold for $t=1,\ldots,T+1$,
and $\mathcal{H}(t)$-(iv) holds for $t=2,\ldots,T+1$. Moreover, the limit of the sequence $(\sum_{t=1}^T f_t(x_t^{2k-1} , x_{t-1}^{2k-1} ) )_{k \geq 1}$ is the optimal value
$\mathcal{Q}_1( x_0 ) $ of \eqref{defpb} and 
any accumulation point of the sequence $(x_1^{2k-1},\ldots,x_T^{2k-1})$ is an optimal solution to \eqref{defpb}.
\end{thm}
\begin{proof} For $t=1,\ldots,T$, let us define the sequence $(y_t^k)_k$
by $y_t^{2k}=(x_t^{2k} , x_{t-1}^{2k-1})$ and $y_t^{2k-1}=(x_t^{2k-1} , x_{t-1}^{2k-1})$ for all $k \geq 1$.
Let us first show $\mathcal{H}(t)$-(i), $t=1,\ldots,T$.
Let $t \in  \{1,\ldots,T\}$.  
Since $x_t^{2k-1} \in X_t^{2k-2}(x_{t-1}^{2k-1} )$, we have $g_t^{2k-2}( x_t^{2k-1} ,x_{t-1}^{2k-1} ) \leq 0$ and therefore
if $g_t(x_t^{2k-1}, x_{t-1}^{2k-1} ) \geq 0$ we have
\begin{equation}\label{limaxg}
\max(g_t(x_t^{2k-1}, x_{t-1}^{2k-1} )  , 0  ) = g_t(x_t^{2k-1}, x_{t-1}^{2k-1} ) \leq g_t(x_t^{2k-1}, x_{t-1}^{2k-1} ) - g_t^{2k-2}(x_t^{2k-1}, x_{t-1}^{2k-1} ).
\end{equation}
Since $g_t^{2k-2} \leq g_t$, the above relation also holds when $g_t(x_t^{2k-1}, x_{t-1}^{2k-1} ) \leq 0$ and therefore
holds for every $k \geq 1$. Similarly, since $x_t^{2k} \in X_t^{2k-1}(x_{t-1}^{2k-1} )$ we have $g_t^{2k-1}( x_t^{2k} ,x_{t-1}^{2k-1} ) \leq 0$
which implies 
\begin{equation}\label{limaxg2}
\max(g_t(x_t^{2k}, x_{t-1}^{2k-1} )  , 0  ) \leq g_t(x_t^{2k}, x_{t-1}^{2k-1} ) - g_t^{2k-1}(x_t^{2k}, x_{t-1}^{2k-1} )
\end{equation}
for all $k \geq 1$.
Next,  $g_{t i} \geq g_{t i}^{k} \geq \ell_{g_{t i}}(\cdot ; y_t^k )$, which implies
\begin{equation}\label{tangeantefh}
g_t^{k} ( y_t^k   )  = g_t( y_t^k ), \;\forall \;k \geq 1.
\end{equation}
Using \eqref{tangeantefh} and applying Lemma \ref{techlemmasequence} to 
$f=g_t$, $f^k=g_t^{k}$, $y^k=y_t^k$ (observe that the assumptions  of the lemma are satisfied), 
we obtain 
\begin{equation}\label{limghT}
\left\{
\begin{array}{l}
\lim_{k \rightarrow +\infty} g_t( x_t^{2k-1} , x_{t-1}^{2k-1} )-g_t^{2k-2}( x_t^{2k-1} , x_{t-1}^{2k-1} )=0,\\
\lim_{k \rightarrow +\infty} g_t( x_t^{2k} , x_{t-1}^{2k-1} )-g_t^{2k-1}( x_t^{2k} , x_{t-1}^{2k-1} )=0.
\end{array}
\right.
\end{equation}
Combining \eqref{limaxg}, \eqref{limaxg2}, and \eqref{limghT}  we get 
\begin{equation}\label{limfing}
 \lim_{k \rightarrow +\infty} \max( g_t(x_t^{2k-1}, x_{t-1}^{2k-1} )  , 0  ) = 0,\;\lim_{k \rightarrow +\infty} \max(g_t(x_t^{2k}, x_{t-1}^{2k-1} ), 0  ) = 0,
\end{equation}
which achieves the proof of $\mathcal{H}(t)$-(i).

Let us now show $\mathcal{H}(1)$-(ii), (iii) and
$\mathcal{H}(t)$-(ii)-(iii), (iv) for $t=2,\ldots,T+1$
by backward  induction on $t$. Clearly, $\mathcal{H}(T+1)$-(ii),(iii), (iv) holds.
Now,  fix $t \in \{1,\ldots,T\}$ and assume that $\mathcal{H}(t+1)$-(ii), (iii), (iv) holds
We will show
that $\mathcal{H}(t)$-(ii), (iii) holds and that $\mathcal{H}(t)$-(iv) holds if $t \geq 2$.
Since $f_t \geq f_t^{k} \geq \ell_{f_t}(\cdot ; y_t^{k} )$, 
we have 
$f_t( y_t^{k} ) \geq f_t^{k} (  y_t^k    )  \geq  \ell_{f_t}(  y_t^k  ; y_t^k ) = f_t( y_t^k )$
and therefore for all $k \geq 1$, 
\begin{equation}\label{tangeantef}
f_t^{k} (y_t^k  )  = f_t( y_t^k ).
\end{equation}
From \eqref{tangeantef}, $\lim_{k \rightarrow +\infty} f_t( y_t^k )-f_t^{k}( y_t^k )=0$ and
applying Lemma \ref{techlemmasequence} to $f=f_t$, $f^k=f_t^{k}$, and $(y^k)=( y_t^k )$ (observe that the assumptions  of the lemma are satisfied),
we obtain 
\begin{equation}\label{limfT}
\left\{
\begin{array}{l}
\lim_{k \rightarrow +\infty} f_t( x_t^{2k-1} , x_{t-1}^{2k-1} )-f_t^{2k-2}( x_t^{2k-1} , x_{t-1}^{2k-1} )=0,\\
\lim_{k \rightarrow +\infty} f_t( x_t^{2k} , x_{t-1}^{2k-1} )-f_t^{2k-1}( x_t^{2k} , x_{t-1}^{2k-1} )=0.
\end{array}
\right.
\end{equation}
\if{
Clearly relations $g_t \geq g_t^{2k-2}$ imply $X_t^{2k-2} (x_{t-1}^{2k-1} ) \supseteq X_t( x_{t-1}^{2k-1} )$
and since $x_t^{2k-1} \in X_t^{2k-2} (x_{t-1}^{2k-1} )$, we deduce 
$f_{t}^{2k-2}(x_t^{2k-1} , x_{t-1}^{2k-1} ) \leq \mathcal{Q}_t( x_{t-1}^{2k-1} )$.
Since the sequence $f_{t}^{2k-2}(x_t^{2k-1} , x_{t-1}^{2k-1} ) - \mathcal{Q}_t( x_{t-1}^{2k-1} )$ is bounded it has a finite
limit superior which satisfies
$$
\varlimsup_{k \rightarrow +\infty}  f_{t}^{2k-2}(x_t^{2k-1} , x_{t-1}^{2k-1} ) - \mathcal{Q}_t( x_{t-1}^{2k-1} ) \leq 0
$$
and using \eqref{limfT}, this implies
\begin{equation}\label{limsupT}
\varlimsup_{k \rightarrow +\infty}  f_{t}(x_t^{2k-1} , x_{t-1}^{2k-1} ) - \mathcal{Q}_t( x_{t-1}^{2k-1} ) \leq 0.
\end{equation}
Let us now show by contradiction that 
\begin{equation}\label{liminfT}
\varliminf_{k \rightarrow +\infty}  f_{t}(x_t^{2k-1} , x_{t-1}^{2k-1} ) - \mathcal{Q}_t( x_{t-1}^{2k-1} ) \geq  0.
\end{equation}
Assume that \eqref{liminfT} does not hold.
It follows that there exists a subsequence $( x_t^{k}, x_{t-1}^{k}  )_{k \in K}$ of $(x_t^{2k-1} , x_{t-1}^{2k-1})$ in 
$\mathcal{X}_t \small{\times} \mathcal{X}_{t-1}$
such that 
$$
\lim_{k \rightarrow +\infty, k \in K}  f_{t}(x_t^{k} , x_{t-1}^{k} ) - \mathcal{Q}_t( x_{t-1}^{k} ) < 0.
$$
Since the sequence $(x_t^{k}, x_{t-1}^{k})_{k \in K}$ belongs to the compact set $\mathcal{X}_t \small{\times} \mathcal{X}_{t-1}$, we can assume
without loss of generality that it converges to some $(x_t^* , x_{t-1}^* ) \in \mathcal{X}_t \small{\times} \mathcal{X}_{t-1}$.
Using continuity of $f_t$ and $\mathcal{Q}_t$ this implies
\begin{equation}\label{fTlessQT}
f_t(x_t^* , x_{t-1}^* ) < \mathcal{Q}_t ( x_{t-1}^* ).
\end{equation}
Using and the continuity of $g$ and \eqref{limfing} 
we also have 
$$
\max( g_t(x_t^* , x_{t-1}^* ) , 0 ) = \lim_{k \rightarrow +\infty, k \in K} \max( g_t(x_t^{k} , x_{t-1}^{k} ) , 0 ) \leq 0.
$$
Therefore $g_t(x_{t}^* , x_{t-1}^* ) \leq 0$ and similarly using \eqref{limfinh} and the continuity of $h_t$ we obtain
$h_T( x_T^* ) \leq 0$. We deduce that
$x_T^* \in X_T(x_{T-1}^* ) $ which is a contradiction with \eqref{fTlessQT}.
Therefore \eqref{liminfT} holds and we have shown that 
$$
\lim_{k \rightarrow +\infty}  \mathcal{Q}_T( x_{T-1}^{2k-1}  ) - f_{T}(x_T^{2k-1}  , x_{T-1}^{2k-1} )=
\lim_{k \rightarrow +\infty}  \mathcal{Q}_T( x_{T-1}^{2k-1} ) - f_{T}^{2k-2}(x_T^{2k-1}  , x_{T-1}^{2k-1} )=0.
$$
Now note that compared to problem \eqref{forreg}, the 
optimization problem \eqref{backpassd} has a smaller feasible set and a 
larger objective function, therefore the optimal value 
${\underline{\mathcal{Q}}}_T^k ( x_{T-1}^{2k-1} )$ of \eqref{backpass}
cannot be less than the optimal value of \eqref{forreg}:
\begin{equation}\label{backaboveforw}
{\underline{\mathcal{Q}}}_T^k ( x_{T-1}^{2k-1} ) \geq  f_T^{2k-2}( x_T^{2k-1} , x_{T-1}^{2k - 1} ). 
\end{equation}
Finally, $\mathcal{H}(T)-(ii)$ follows from the following chain of inequalities:
\begin{eqnarray*}
0 \leq \mathcal{Q}_T( x_{T-1}^{2k-1} ) - \mathcal{Q}_T^k ( x_{T-1}^{2k-1} )  & \leq  &   \mathcal{Q}_T( x_{T-1}^{2k-1} ) - \mathcal{C}_T^k ( x_{T-1}^{2k-1} )\mbox{ since }\mathcal{Q}_T^k \geq \mathcal{C}_T^k,\\
& = & \mathcal{Q}_T( x_{T-1}^{2k-1} ) - {\underline{\mathcal{Q}}}_T^k ( x_{T-1}^{2k-1} )\mbox{ by definition of }\mathcal{C}_T^k,\\
& \leq & \mathcal{Q}_T( x_{T-1}^{2k-1} ) - f_T^{2k-2}( x_T^{2k-1} , x_{T-1}^{2k - 1} )\mbox{ using }\eqref{backaboveforw}.
\end{eqnarray*}
This completes the proof of $\mathcal{H}(T)$. Let us now assume that $\mathcal{H}(t+1)$ holds for some $t \in \{2,\ldots,T-1\}$ and let us show
$\mathcal{H}(t)$. As above, using once again Lemma \ref{techlemmasequence} and the fact that 
$f_t \geq f_t^{k} \geq f_t^{k-1}$, $f_t(y_t^k )=f_t^{k}( y_t^k )$,
$g_t \geq g_t^{k} \geq g_t^{k-1}$, $g_t( y_t^k )=g_t^{k}(y_t^k)$,  we get $\mathcal{H}(t)-(i)$ and
\begin{equation}\label{limft}
\lim_{k \rightarrow +\infty} f_t( x_t^{2k-1} , x_{t-1}^{2k-1} )-f_t^{2k-2}( x_t^{2k-1} , x_{t-1}^{2k-1} )=0.
\end{equation}
}\fi
Using the fact that $x_t^{2k-1} \in X_t^{2k-2} (x_{t-1}^{2k-1} ) \supseteq X_t( x_{t-1}^{2k-1} )$,
the relation 
$f_{t}^{2k-2}(\cdot , x_{t-1}^{2k-1} )  +  \mathcal{Q}_{t+1}^{k-1}(\cdot ) \leq 
f_{t}(\cdot , x_{t-1}^{2k-1} )  +  \mathcal{Q}_{t+1}( \cdot )$,
and recalling defintions of $x_t^{2k-1}$ and $\mathcal{Q}_t$, we get 
$$f_{t}^{2k-2}(x_t^{2k-1} , x_{t-1}^{2k-1} )  +  \mathcal{Q}_{t+1}^{k-1}(  x_t^{2k-1}  ) \leq \mathcal{Q}_t( x_{t-1}^{2k-1} ).
$$
Since the sequence $f_{t}^{2k-2}(x_t^{2k-1} , x_{t-1}^{2k-1} )  +  \mathcal{Q}_{t+1}^{k-1}(  x_t^{2k-1}  ) - \mathcal{Q}_t( x_{t-1}^{2k-1} )$ is bounded it has a finite
limit superior which satisfies
\begin{equation}\label{limftQt}
\varlimsup_{k \rightarrow +\infty}  f_{t}^{2k-2}(x_t^{2k-1} , x_{t-1}^{2k-1} )  +  \mathcal{Q}_{t+1}^{k-1}(  x_t^{2k-1}  ) - \mathcal{Q}_t( x_{t-1}^{2k-1} ) \leq 0.
\end{equation}
The induction hypothesis gives 
$$
\lim_{k \rightarrow +\infty}  \mathcal{Q}_{t+1}^{k}(  x_t^{2k-1}  ) - \mathcal{Q}_{t+1}( x_{t}^{2k-1} ) = 0.
$$
Applying Lemma \ref{techlemmasequence} to $f=\mathcal{Q}_{t+1}, f^k=\mathcal{Q}_{t+1}^k, y^k=x_{t}^{2k-1}$ (observe that the assumptions  of the lemma are satisfied), we obtain
$$
\lim_{k \rightarrow +\infty}  \mathcal{Q}_{t+1}^{k-1}(  x_t^{2k-1}  ) - \mathcal{Q}_{t+1}( x_{t}^{2k-1} ) = 0.
$$
Together with \eqref{limfT}, \eqref{limftQt} this relation, implies
\begin{equation}\label{ftkqtk}
\begin{array}{l}
\varlimsup_{k \rightarrow +\infty}  f_{t}^{2k-2}(x_t^{2k-1} , x_{t-1}^{2k-1} )  +  \mathcal{Q}_{t+1}^{k-1}(  x_t^{2k-1}  ) - \mathcal{Q}_t( x_{t-1}^{2k-1} ) \\ 
=   \varlimsup_{k \rightarrow +\infty}  f_{t}^{2k-1}(x_t^{2k} , x_{t-1}^{2k-1} )  +  \mathcal{Q}_{t+1}^{k}(  x_t^{2k-1}  ) - \mathcal{Q}_t( x_{t-1}^{2k-1} )  \\
 = \varlimsup_{k \rightarrow +\infty}  f_{t}(x_t^{2k-1} , x_{t-1}^{2k-1} )  +  \mathcal{Q}_{t+1}(  x_t^{2k-1}  ) - \mathcal{Q}_t( x_{t-1}^{2k-1} ) \leq 0.
\end{array}
\end{equation}
Let us now show by contradiction that 
\begin{equation}\label{liminftsup}
\begin{array}{lll}
\varliminf_{k \rightarrow +\infty} \;  f_{t}(x_t^{2k-1} , x_{t-1}^{2k-1} )  +  \mathcal{Q}_{t+1}(  x_t^{2k-1}  ) - \mathcal{Q}_t( x_{t-1}^{2k-1} ) \geq 0.
\end{array}
\end{equation}
If \eqref{liminftsup} does not hold, using the fact that
$(x_t^{2k-1} , x_{t-1}^{2k-1})$ is a sequence from the compact set $\mathcal{X}_t \small{\times} \mathcal{X}_{t-1}$
and the lower semicontinuity of $f_t, g_t, \mathcal{Q}_{t+1}, \mathcal{Q}_t$, we can find a 
subsequence $(x_t^{k} , x_{t-1}^{k})_{k \in K}$ converging to some $(x_t^*, x_{t-1}^*) \in \mathcal{X}_t \small{\times} \mathcal{X}_{t-1}$
such that 
$$
f_t( x_t^* , x_{t-1}^* ) + \mathcal{Q}_{t+1}(x_t^* ) < \mathcal{Q}_t( x_{t-1}^* ),
$$
$g_t( x_t^* , x_{t-1}^* ) \leq 0$, and $x_t^* \in X_t( x_{t-1}^* )$,
which is a contradiction. Therefore \eqref{liminftsup} must hold and we have shown  
\begin{equation}\label{anteschaint}
\lim_{k \rightarrow +\infty} \mathcal{Q}_t(x_{t-1}^{2k-1} )  - f_t^{2k-2}(x_t^{2k-1}, x_{t-1}^{2k-1} ) - \mathcal{Q}_{t+1}^{k-1}(x_t^{2k-1} )=
\lim_{k \rightarrow +\infty} \mathcal{Q}_t(x_{t-1}^{2k-1} )  - f_t(x_t^{2k-1}, x_{t-1}^{2k-1} ) - \mathcal{Q}_{t+1}(x_t^{2k-1} )=0. 
\end{equation}
As before, note that the optimal value of \eqref{backpassd} is larger than the optimal value of \eqref{forreg}, i.e.,
\begin{equation}\label{lbchain}
{\underline{\mathcal{Q}}}_t^k ( x_{t-1}^{2k-1} ) \geq  f_t^{2k-2} (x_t^{2k-1} , x_{t-1}^{2k-1} )+\mathcal{Q}_{t+1}^{k-1}(x_t^{2k-1} ),
\end{equation}
implying
\begin{eqnarray}\label{eqhti3}
0 \leq \mathcal{Q}_t( x_{t-1}^{2k-1} ) - {\underline{\mathcal{Q}}}_t^k ( x_{t-1}^{2k-1} )
 \leq  \mathcal{Q}_t( x_{t-1}^{2k-1} ) - f_t^{2k-2} (x_t^{2k-1} , x_{t-1}^{2k-1} )-\mathcal{Q}_{t+1}^{k-1}(x_t^{2k-1} ),
\end{eqnarray}
which, together with \eqref{anteschaint}, gives $\mathcal{H}(t)$-(ii).
Next,
$$
\begin{array}{l}
\displaystyle \lim_{k \rightarrow +\infty} \mathcal{Q}_t(x_{t-1}^{2k-1} )  - \sum_{\tau=t}^T f_\tau(x_\tau^{2k-1}, x_{\tau-1}^{2k-1} )\\
 = \underbrace{\displaystyle \lim_{k \rightarrow +\infty} \mathcal{Q}_t(x_{t-1}^{2k-1} )  - f_t(x_t^{2k-1}, x_{t-1}^{2k-1}) - \mathcal{Q}_{t+1}( x_t^{2k-1} )}_{=0 \mbox{ by }\eqref{anteschaint}} + 
\underbrace{\displaystyle \lim_{k \rightarrow +\infty}  \mathcal{Q}_{t+1}( x_t^{2k-1} ) - \sum_{\tau=t+1}^T f_\tau(x_\tau^{2k-1}, x_{\tau-1}^{2k-1} )}_{=0 \mbox{ using }\mathcal{H}(t+1)-(iii)},\\
=0,
\end{array}
$$
and we obtain $\mathcal{H}(t)$-(iii). Finally, for $t \geq 2$,
\begin{eqnarray*}
0 \leq \mathcal{Q}_t( x_{t-1}^{2k-1} ) - \mathcal{Q}_t^k ( x_{t-1}^{2k-1} )  & \leq  &   \mathcal{Q}_t( x_{t-1}^{2k-1} ) - \mathcal{C}_t^k ( x_{t-1}^{2k-1} )\mbox{ since }\mathcal{Q}_t^k \geq \mathcal{C}_t^k,\\
& = & \mathcal{Q}_t( x_{t-1}^{2k-1} ) - {\underline{\mathcal{Q}}}_t^k ( x_{t-1}^{2k-1} )\mbox{ by definition of }\mathcal{C}_t^k
\end{eqnarray*}
which combines with \eqref{eqhti3} to show $\mathcal{H}(t)$-(iv).$\hfill$
\end{proof}

}\fi

\subsection{Computation of the subgradient in Step d) of DCuP} \label{subsection:subg}

This subsection explains how to compute
a subgradient $\beta_t^k$ of 
${\underline{\mathcal{Q}}}_t^{k-1} (\cdot )$ at $x_{t-1}^{k}$ in Step d) of DCuP.
 
Observe that we can express 
${\underline{\mathcal{Q}}}_t^{k-1}$ as
\begin{equation}\label{forreg100}
{\underline{\mathcal{Q}}}_t^{k-1} ( x_{t-1} )   = 
\left\{
\begin{array}{l}
\displaystyle \min_{x_t \in \mathbb{R}^n, f, \theta \in \mathbb{R}}  f + \theta\\
x_t \in \mathcal{X}_t,\\
f \geq \ell_{f_t}(x_t, x_{t-1} ,(x_t^j, x_{t-1}^j)),j=1,\ldots,k-1,\\
\theta \geq {\underline{\mathcal{Q}}}_{t+1}^{i-1} ( x_{t}^{i} ) + 
\langle \beta_{t+1}^i ,  x_{t} - x_{t}^{i} \rangle,\;i=1,\ldots,k-1,\\
\ell_{g_{t i}}(x_t, x_{t-1}, (x_t^j, x_{t-1}^j)) \leq 0,\;j=1,\ldots,k-1,i=1,\ldots,p,\\
A_t x_t + B_t x_{t-1} = b_t.
\end{array}
\right.
\end{equation}
Due to Assumption (H1)-2), for every 
$x_{t-1} \in \mathcal{X}_{t-1}$, there exists
$x_t \in \mbox{ri}( \mathcal{X}_t )$
such that 
$A_t x_t + B_t x_{t-1} = b_t$
and 
$g_t(x_t, x_{t-1}) \leq 0$, which implies
that for every $i=1,\ldots,p$, and $j=1,\ldots,k-1$,
we have
$$\ell_{g_{t i}}(x_t, x_{t-1}, (x_t^j, x_{t-1}^j))
\leq g_{t i}(x_t, x_{t-1}) \leq 0$$
and therefore  Slater constraint qualification holds
for problem \eqref{forreg100} for every 
$x_{t-1} \in \mathcal{X}_{t-1}$.
Next observe that 
due to the compactness of 
$\mathcal{X}_t$
the objective function  
of \eqref{forreg100} bounded 
from below on the feasible set. 
It follows that the optimal value of 
\eqref{forreg100} is finite and by the Duality Theorem, we can write problem \eqref{forreg100} as the 
optimal value of the corresponding dual problem.
To write this dual, it is convenient to rewrite 
${\underline{\mathcal{Q}}}_t^{k-1}$ on $\mathcal{X}_{t-1}$ as
\begin{equation}\label{forreg100b}
{\underline{\mathcal{Q}}}_t^{k-1} ( x_{t-1} )   = 
\left\{
\begin{array}{l}
\displaystyle \min_{x_t \in \mathbb{R}^n, f, \theta \in \mathbb{R}}  f + \theta\\
x_t \in \mathcal{X}_t,\\
f {\textbf{e}} \geq A_t^{k-1} x_t + B_t^{k-1} x_{t-1} + C_t^{k-1},\\
\theta {\textbf{e}} \geq \theta_{t+1}^{0:k-1} + \beta_{t+1}^{0:k-1} x_t,\\
D_t^{k-1} x_t + E_t^{k-1} x_{t-1} + H_t^{k-1} \leq 0,\\
A_t x_t + B_t x_{t-1}= b_t,
\end{array}
\right.
\end{equation}
where \textbf{e} is a vector of ones of dimension $k-1$ and
$A_t^{k-1}, B_t^{k-1}, D_t^{k-1}, E_t^{k-1}, \beta_{t+1}^{1:k-1}$ 
(resp. $C_t^{k-1}, H_t^{k-1}, \theta_{t+1}^{1:k-1}$) are matrices (resp. vectors) of appropriate dimensions. 
In particular, $\beta_{t+1}^{0:k-1}$ is a matrix
with $k$ rows with $(i+1)$-th row equal to 
$(\beta_{t+1}^i)^{\top}$ and $\theta_{t+1}^{0:k-1}$ is
a vector of size $k$ with first
component equal to $\theta_{t+1}^0$
and for $i \geq 2$ component
$i$ given by 
$\theta_{t+1}^{i-1} = {\underline{\mathcal{Q}}}_{t+1}^{i-2} ( x_{t}^{i-1} ) - 
\langle \beta_{t+1}^{i-1} , x_{t}^{i-1} \rangle$.

We now write the dual of \eqref{forreg100b} as 
\begin{equation}\label{dualqbartk}
{\underline{\mathcal{Q}}}_t^{k-1} ( x_{t-1} )   = 
\left\{
\begin{array}{l}
\displaystyle \max_{\alpha, \mu,  \delta, \lambda} 
h_{t, x_{t-1}}(\alpha , \lambda , \mu , \delta  )  \\
\alpha \geq 0, \mu \geq 0, \delta \geq 0, \lambda,
\end{array}
\right.
\end{equation}
where dual function $h_{t, x_{t-1}}$ 
is given by 
\begin{equation}\label{defdualfh}
h_{t, x_{t-1}}(\alpha , \lambda , \mu , \delta )=
\left\{
\begin{array}{l}
\displaystyle \min_{x_t \in \mathbb{R}^n, f, \theta \in \mathbb{R}} L_{t, x_{t-1}}(x_t, f, \theta ;\alpha , \lambda , \mu , \delta)  \\
x_t \in \mathcal{X}_t,
\end{array}
\right.
\end{equation}
with Lagrangian $L_{t, x_{t-1}}(x_t,f,\theta;\alpha , \lambda , \mu , \delta)$ given by
$$
\begin{array}{lll}
L_{t, x_{t-1}}(x_t, f, \theta ;\alpha , \lambda , \mu , \delta) 
& = & f+\theta +  
\langle  \alpha ,
A_t^{k-1} x_t + B_t^{k-1} x_{t-1} + C_t^{k-1} - f {\textbf{e}} \rangle 
+ \langle \lambda , A_t x_t + B_t x_{t-1} -b_t \rangle \\
&  & + \langle \mu , D_t^{k-1} x_t + E_t^{k-1} x_{t-1} + H_t^{k-1}  \rangle + \langle  
\delta  ,   \theta_{t+1}^{1:k-1} + \beta_{t+1}^{1:k-1} x_t  - \theta {\textbf{e}} \rangle.
\end{array}
$$
With this notation, we have the following
characterization of $\partial {\underline{\mathcal{Q}}}_t^{k-1}(x_{t-1}^k )$:
\begin{lemma} \label{lemmsubdiffqtk}
Let Assumption (H1) hold. Then the subdifferential of ${\underline{\mathcal{Q}}}_t^{k-1}$
at $x_{t-1}^k$ is the set of points 
of form
\begin{equation}\label{formulasubgradientqtbar}
B_t^{\top} \lambda + (B_t^{k-1})^{\top} \alpha 
+ (E_{t}^{k-1})^{\top} \mu
\end{equation}
where $(\alpha, \lambda, \mu)$ is such 
that there is $\delta$ satisfying 
$(\alpha, \lambda, \mu, \delta)$ is an optimal solution of dual
problem  \eqref{dualqbartk} written for 
$x_{t-1}=x_{t-1}^k$.
\end{lemma}
\begin{proof} Defining
$$
\mathcal{S}_k = \mathcal{X}_t \small{\times} \mathbb{R} \small{\times} \mathbb{R} \small{\times} \mathbb{R}^n  
\cap C_k \cap D,
$$
where 
$$
\begin{array}{lll}
C_k & = &  \left\{ 
(x_t, f, \theta, x_{t-1}) : 
\left\{
\begin{array}{l}
A_t^{k-1} x_t + B_t^{k-1} x_{t-1} + C_t^{k-1} \leq f {\textbf{e}},\\
\theta_{t+1}^{0:k-1} + \beta_{t+1}^{0:k-1} x_t \leq \theta {\textbf{e}},\\
D_t^{k-1} x_t + E_t^{k-1} x_{t-1} + H_t^{k-1} \leq 0\\
\end{array}
\right.
\right\} \\
D & =  & \{ 
(x_t, f, \theta, x_{t-1}) : A_t x_t + B_t x_{t-1}= b_t
\},
\end{array}
$$
we have 
\begin{equation}\label{valuefunctionqtk}
{\underline{\mathcal{Q}}}_t^{k-1}(x_{t-1}^k )=
\left\{
\begin{array}{l}
\inf \;f + \theta +\mathbb{I}_{\mathcal{S}_k}(x_t,f,\theta , x_{t-1}^k)\\
x_t \in \mathbb{R}^n, f, \theta \in \mathbb{R}.
\end{array}
\right.
\end{equation}
Using Theorem 24(a) in Rockafellar \cite{rock74}, we have
\begin{equation}\label{subdiffvaluef1}
\begin{array}{lll}
\beta_t^k \in \partial {\underline{\mathcal{Q}}}_t^{k-1}(x_{t-1}^k ) &\Leftrightarrow &(0,0,0,\beta_t^k) \in \partial 
( f + \theta  +\mathbb{I}_{\mathcal{S}_k}  )(x_t^k,f_{t k}, \theta_{t k}, x_{t-1}^k)\\
&\Leftrightarrow &(0,0,0,\beta_t^k)   \in [0;1;1;0] +  \mathcal{N}_{\mathcal{S}_k}(x_t^k,f_{t k}, \theta_{t k}, x_{t-1}^k), \;\;(a)
\end{array}
\end{equation}
where $f_{t k}$ and $\theta_{t k}$ are the optimal values of respectively
$f$ and $\theta$ in \eqref{forreg100b}
written for $x_{t-1}=x_{t-1}^k$.
For equivalence \eqref{subdiffvaluef1}-(a), we have used the fact that 
$(x_t,f,\theta,x_{t-1}) \rightarrow f+\theta$ and 
$\mathbb{I}_{\mathcal{S}_k}$ are proper,
finite at $(x_t^k,f_{t k}, \theta_{t k}, x_{t-1}^k)$, and
the intersection of the relative interior of the domain of these functions,
i.e., set $\mbox{ri}(\mathcal{S}_k)$, is nonempty. Next,
\begin{equation}\label{formulanormalsk}
\mathcal{N}_{ \mathcal{S}_k}(x_t^k,f_{t k}, \theta_{t k}, x_{t-1}^k)=\mathcal{N}_{C_k}(x_t^k,f_{t k}, \theta_{t k}, x_{t-1}^k)+ \mathcal{N}_{D}(x_t^k,f_{t k}, \theta_{t k}, x_{t-1}^k)+\mathcal{N}_{\mathcal{X}_t \small{\times} \mathbb{R} \small{\times} \mathbb{R} \small{\times} \mathbb{R}^n}(x_t^k,f_{t k}, \theta_{t k}, x_{t-1}^k),
\end{equation}
and standard calculus on normal cones gives
\begin{equation}\label{normalconeDk}
\begin{array}{lll}
\mathcal{N}_{\mathcal{X}_t \small{\times} \mathbb{R} \small{\times} \mathbb{R} \small{\times} \mathbb{R}^n}(x_t^k,f_{t k}, \theta_{t k}, x_{t-1}^k) &=& \mathcal{N}_{\mathcal{X}_t}(x_t^k) \small{\times}  \{0\} \small{\times}  \{0\}\small{\times}  \{0\},\\
\mathcal{N}_{D}(x_t^k,f_{t k}, \theta_{t k}, x_{t-1}^k)&=&
\Big\{[A_t^{\top}; 0;0;B_t^{\top} ] \lambda \;:\;\lambda \in \mathbb{R}^q\Big\},
\end{array}
\end{equation}
and $\mathcal{N}_{C_k}(x_t^k,f_{t k}, \theta_{t k}, x_{t-1}^k)$ is the set of points of form 
\begin{equation}\label{optcondnormal}
\left( 
\begin{array}{c}
(A_t^{k-1})^{\top} \alpha + (\beta_{t+1}^{0:k-1})^{\top} \delta + (D_t^{k-1})^{\top} \mu \\
-{\textbf{e}}^{\top} \alpha \\
-{\textbf{e}}^{\top} \delta \\
(B_t^{k-1})^{\top} \alpha + (E_t^{k-1})^{\top} \mu
\end{array}
\right) 
\end{equation}
where $\alpha,\delta,\mu$ satisfy 
\begin{equation}\label{satisfyalpha}
\begin{array}{c}
\alpha, \delta, \mu \geq 0,\\
\left(   
\begin{array}{l}
\alpha \\
\delta \\
\mu
\end{array}
\right)^{\top}
\left( 
\begin{array}{l}
A_t^{k-1} x_t^k + B_t^{k-1} x_{t-1}^k + C_t^{k-1} - f_{t k} {\textbf{e}}\\
\theta_{t+1}^{0:k-1} + \beta_{t+1}^{0:k-1} x_t^k - \theta_{t k} {\textbf{e}}\\
D_t^{k-1} x_t^k + E_t^{k-1} x_{t-1}^k + H_t^{k-1} \\
\end{array}
\right)=0.
\end{array}
\end{equation}
Combining \eqref{subdiffvaluef1}, \eqref{formulanormalsk}, \eqref{normalconeDk}, \eqref{optcondnormal}, we see that 
$\beta_t^k \in \partial {\underline{\mathcal{Q}}}_t^{k}(x_{t-1}^k )$ if and only if
$\beta_t^k$ is of form 
\eqref{formulasubgradientqtbar}
where $\alpha, \lambda, \mu$ satisfies \eqref{satisfyalpha} and 
\begin{equation}\label{satisfyalphabis}
\begin{array}{lll}
0 & \in & \mathcal{N}_{\mathcal{X}_t}(x_t^k) + A_t^{\top} \lambda +
(A_t^{k-1})^{\top} \alpha + (\beta_{t+1}^{0:k-1}) \delta  + (D_{t}^{k-1})^{\top} \mu, \\
0 & = & 1 - {\textbf{e}}^{\top} \alpha, \\
0 & = & 1 - {\textbf{e}}^{\top} \delta.
\end{array}
\end{equation}
Finally, it suffices to observe that $\alpha, \lambda, \mu$ satisfies \eqref{satisfyalpha}
and \eqref{satisfyalphabis}  if and only if $\alpha, \lambda, \mu, \delta$ is an optimal solution of dual
problem \eqref{dualqbartk}. 
\end{proof}
\vspace*{0.2cm}
\par Using the previous lemma and denoting by
$(\alpha_t^k, \lambda_t^k, \mu_t^k, \delta_t^k)$ an optimal solution of \eqref{dualqbartk} written for 
$x_{t-1}=x_{t-1}^k$, we have that
\begin{equation}\label{formulabetatk}
\beta_t^k = 
(B_t^{k-1})^{\top} \alpha_t^k +
B_t^{\top} \lambda_t^k
+ (E_{t}^{k-1})^{\top} \mu_t^k
\in \partial  {\underline{\mathcal{Q}}}_t^{k-1} ( x_{t-1}^k ).
\end{equation}

\begin{remark}
When $\mathcal{X}_t$ is polyhedral, formula
 \eqref{formulabetatk} follows from Duality for linear programming.
 For a more general convex set $\mathcal{X}_t$,
 formula  \eqref{formulabetatk} directly follows
 from applying
to value function 
${\underline{\mathcal{Q}}}_t^{k-1}$
Lemma 2.1 in \cite{guiguessiopt2016} or Proposition 3.2 in \cite{guiguesinexactsmd} 
which respectively provide 
a characterization of the subdifferential and subgradients
for value functions of general convex optimization problems
(whose argument is in the objective function and in linear and nonlinear coupling constraints of the corresponding
optimization problem). The proof of Lemma
\ref{lemmsubdiffqtk}
is a  proof of relation \eqref{formulabetatk} 
specializing to the particular
case of value function
${\underline{\mathcal{Q}}}_t^{k-1}$ the proof of Lemma 2.1 in \cite{guiguessiopt2016}.
\end{remark}

\section{The StoDCuP (Stochastic Dynamic Cutting Plane) algorithm}\label{stodcup}

\subsection{Problem formulation and assumptions}

We consider multistage stochastic nonlinear optimization problems of the form
\begin{equation}\label{pbtosolve}
\begin{array}{l}
\displaystyle 
\min_{x_1 \in X_1( x_0 , \xi_1)} f_1(x_1,x_0,\xi_1 ) +
\mathbb{E}\left[ \min_{x_2 \in X_2( x_1 , \xi_2)} f_2(x_2,x_1,\xi_2) +
\mathbb{E}\left[ \ldots + \mathbb{E}\left[ \min_{x_T \in X_T( x_{T-1} , \xi_T)} f_T(x_{T},x_{T-1},\xi_T) \right] \right] \right],
\end{array}
\end{equation}
where $x_0$ is given,  $(\xi_t)_{t=2}^T$ is a stochastic process, $\xi_1$ is deterministic, and
$$ 
X_t( x_{t-1} , \xi_t)= \{x_t \in \mathbb{R}^n : \displaystyle A_{t} x_{t} + B_{t} x_{t-1} = b_t,
g_t(x_t, x_{t-1}, \xi_t) \leq 0, x_t \in \mathcal{X}_t \}.
$$
In the constraint set above, $\mathcal{X}_t$ is polyhedral and $\xi_t$ contains in particular the random elements in matrices $A_t, B_t$, and vector $b_t$.

We make the following assumption on $(\xi_t)$:\\
\par (H0) $(\xi_t)$ 
is interstage independent and
for $t=2,\ldots,T$, $\xi_t$ is a random vector taking values in $\mathbb{R}^K$ with a discrete distribution and
a finite support $\Theta_t=\{\xi_{t 1}, \ldots, \xi_{t M_t}\}$ with $p_{t i}=\mathbb{P}(\xi_t = \xi_{t i}),i=1,\ldots,M_t$, while $\xi_1$ is deterministic.\\

For this problem, we can write Dynamic Programming equations: the first stage problem is 
\begin{equation}\label{firststodp}
\mathcal{Q}_1( x_0 ) = \left\{
\begin{array}{l}
\min_{x_1 \in \mathbb{R}^n} f_1(x_1, x_0, \xi_1)  + \mathcal{Q}_2 ( x_1 )\\
x_1 \in X_1( x_{0}, \xi_1 )\\
\end{array}
\right.
\end{equation}
for $x_0$ given and for $t=2,\ldots,T$, $\mathcal{Q}_t( x_{t-1} )= \mathbb{E}_{\xi_t}[ \mathfrak{Q}_t ( x_{t-1},  \xi_{t}  )  ]$ with
\begin{equation}\label{secondstodp} 
\mathfrak{Q}_t ( x_{t-1}, \xi_{t}  ) = 
\left\{ 
\begin{array}{l}
\min_{x_t \in \mathbb{R}^n}  f_t ( x_t , x_{t-1}, \xi_t ) + \mathcal{Q}_{t+1} ( x_t )\\
x_t \in X_t ( x_{t-1}, \xi_t ),
\end{array}
\right.
\end{equation}
with the convention that $\mathcal{Q}_{T+1}$ is identically zero.

We set $\mathcal{X}_0=\{x_0\}$ and make the following assumptions (H1)-Sto on the problem data:
\par (H1)-Sto: for $t=1,\ldots,T$,  
\begin{itemize}
\item[1)] $\mathcal{X}_{t}$ is a nonempty, compact, and polyhedral set.
\item[2)] For every $j=1,\ldots,M_t$, the function
$f_t(\cdot, \cdot , \xi_{t j})$ is convex, proper, lower semicontinuous on $\mathcal{X}_t \small{\times} \mathcal{X}_{t-1}$ 
and $\mathcal{X}_{t} \small{\times} \mathcal{X}_{t-1}\subset \inte (\dom(f_t(\cdot, \cdot , \xi_{t j})))$.
\item[3)] For every $j=1,\ldots,M_t$, each component $g_{t i}(\cdot, \cdot, \xi_{t  j}), i=1,\ldots,p$, of function $g_{t}(\cdot, \cdot, \xi_{t  j})$ is
convex, proper, lower semicontinuous such that  $\mathcal{X}_{t} \small{\times} \mathcal{X}_{t-1} \subset \inte (\dom (g_{t i}(\cdot, \cdot , \xi_{t j})))$.
\item[4)] $X_1(x_0,\xi_1) \neq \emptyset$ and for every $t=2,\ldots,T$, 
for every $j=1,\ldots,M_t$, $\mathcal{X}_{t-1} \subset \inte (\dom( X_t(\cdot, \xi_{t j})))$.
\end{itemize}

\begin{rem} 
Nonlinear constraints of  form $h_{t i}(x_t,\xi_t)\leq 0$ or $h_{t i}(x_t)\leq 0$ at stage $t$
can be handled, adding the corresponding component functions $h_{t i}$  in $g_t$, as long 
as (H1)-Sto is satisfied. In particular, convexity of $h_{t i}(\cdot,\xi_{t j})$ is required for $j=1,\ldots,M_t$.
\end{rem}

It is easy to show that under Assumption (H1)-Sto, functions $\mathcal{Q}_t$ are convex
and Lipschitz continuous on $\mathcal{X}_{t-1}$:
\begin{lemma} \label{lm:2.1sto}
Let Assumption (H1)-Sto hold. 
Then $\mathcal{Q}_{t}$ is convex Lipschitz continuous on $\mathcal{X}_{t-1}$
for $t=2,\ldots,T+1$.
\end{lemma}
\begin{proof} The proof is analogue to the proof of Lemma \ref{lm:2.1}.$\hfill$ 
\end{proof}

\subsection{Forward StoDCuP}

The algorithm to be presented in this section for solving \eqref{pbtosolve} is an extension of the DCuP algorithm
to the stochastic case. All inequalities and equalities between random variables in the rest of the paper hold almost surely with respect to the sampling of the
algorithm.

Due to Assumption (H0), the $\displaystyle \prod_{t=2}^T M_t$ realizations of $(\xi_t)_{t=1}^T$ form a scenario tree of depth $T+1$
where the root node $n_0$ associated to a stage $0$ (with decision $x_0$ taken at that
node) has one child node $n_1$
associated to the first stage (with $\xi_1$ deterministic).

We denote by $\mathcal{N}$ the set of nodes, by
{\tt{Nodes}}$(t)$ the set of nodes for stage $t$ and
for a node $n$ of the tree, we define: 
\begin{itemize}
\item $C(n)$: the set of children nodes (the empty set for the leaves);
\item $x_n$: a decision taken at that node;
\item $p_n$: the transition probability from the parent node of $n$ to $n$;
\item $\xi_n$: the realization of process $(\xi_t)$ at node $n$\footnote{The same notation $\xi_{\tt{Index}}$ is used to denote
the realization of the process at node {\tt{Index}} of the scenario tree and the value of the process $(\xi_t)$
for stage {\tt{Index}}. The context will allow us to know which concept is being referred to.
In particular, letters $n$ and $m$ will only be used to refer to nodes while $t$ will be used to refer to stages.}:
for a node $n$ of stage $t$, this realization $\xi_n$ contains in particular the realizations
$b_n$ of $b_t$, $A_{n}$ of $A_{t}$, and $B_{n}$ of $B_{t}$.
\item $\xi_{[n]}$: the history of the realizations of process $(\xi_t)$ from the first stage node $n_1$ to node $n$:
 for a node $n$ of stage $t$, the $i$-th component of $\xi_{[n]}$ is $\xi_{\mathcal{P}^{t-i}(n)}$ for $i=1,\ldots, t$,
 where $\mathcal{P}:\mathcal{N} \rightarrow \mathcal{N}$ is the function 
 associating to a node its parent node (the empty set for the root node).
\end{itemize}

At each iteration of the algorithm, trial points
are computed on a sampled scenario and 
lower bounding affine functions, called cuts in the sequel, are built
for convex functions $\mathcal{Q}_{t},t=2,\ldots,T+1$, at these trial points.
More precisely, at iteration $k$ 
denoting by $x_{t-1}^k$ the trial point for stage $t-1$,
the cut
\begin{equation}\label{formulacut}
\mathcal{C}_t^k ( x_{t-1} ) = \theta_t^k +  \langle \beta_t^k , x_{t-1}  \rangle
\end{equation}
is built for $\mathcal{Q}_t$ with the convention that $\mathcal{C}_{T+1}^k$ is the null function (see below for the computation of 
$\theta_t^k$, $\beta_t^k$).
As in SDDP, we end up in iteration $k$ with an approximation
$\mathcal{Q}_t^k$
of $\mathcal{Q}_t$ which is a maximum of $k+1$ affine functions:
$\mathcal{Q}_t^k( x_{t-1} ) =\max_{0 \leq j \leq k} \;\mathcal{C}_t^j ( x_{t-1} )$.

Additionally, the variant we propose builds cutting plane approximations of convex functions 
$f_t(\cdot,\cdot, \xi_{t j})$ and $g_{t i}(\cdot,\cdot, \xi_{t j})$, $t=1,\ldots,T, i=1,\ldots,p, j=1,\ldots,M_t$,
computing linearizations of these functions.
At the end of iteration $k$, these approximations 
will be denoted by $f_{t j}^{k}$ and $g_{t i j}^{k}$
for $f_t(\cdot,\cdot, \xi_{t j})$ and $g_{t i}(\cdot,\cdot, \xi_{t j})$ respectively,
and take the form of a maximum of
$k+1$ affine functions. We use the notation
$$
\begin{array}{l}
f_{t j}^{k}(x_t, x_{t-1}) = \displaystyle \max_{\ell=0,\ldots,k} \; a_{t j}^{\ell}  x_t  +  b_{t j}^{\ell}  x_{t-1}  +  c_{t j}^{\ell},\\
g_{t i j}^{k}(x_t, x_{t-1}) = \displaystyle \max_{\ell=0,\ldots,k} \; d_{t i j}^{\ell}  x_t  +  e_{t i j}^{\ell}  x_{t-1}  +  h_{t i j}^{\ell},
\end{array}
$$
where $a_{t j}^{\ell}, b_{t j}^{\ell}, d_{t i j}^{\ell}$, and $e_{t i j}^{\ell}$ are $n$-dimensional  row vectors. 
The trial points of iteration $k$ are computed before updating these functions, therefore 
using approximations $f_{t j}^{k-1}, g_{t i j}^{k-1},$ and
$\mathcal{Q}_{t+1}^{k-1}$ of $f_t(\cdot,\cdot, \xi_{t j})$, $g_{t i}(\cdot,\cdot, \xi_{t j})$,
and $\mathcal{Q}_{t+1}$ available at the end of iteration $k-1$.
These trial points are decisions 
computed
at nodes $(n_1^k, n_2^k, \ldots, n_T^k)$ using these approximations, knowing that
$n_1^k=n_1$, and for $t \geq 2$, $n_t^k$ is a node of stage $t$, child of node $n_{t-1}^k$, i.e.,
these nodes correspond to a sample $({\tilde \xi}_1^k, {\tilde \xi}_2^k,\ldots, {\tilde \xi}_T^k)$
of $(\xi_1,\xi_2,\ldots,\xi_T)$.
At iteration $k$, the linearizations for $f_t(\cdot,\cdot, \xi_{t j})$, $g_{t i}(\cdot,\cdot, \xi_{t j})$
(resp. $\mathcal{Q}_{t}$) are computed at $(x_m^k, x_n^k)$ (resp. $x_n^k$)  where $n=n_{t-1}^k$, and $m$
is the child node of node $n$ such that $\xi_m = \xi_{t j}$. For convenience, for any node 
$m$ of stage $t$, we will denote by $j_t(m)$ the unique index
$j_t(m)$ such that 
$\xi_{m}=\xi_{t j_t(m)}$. Before detailing the steps of StoDCuP, we need more notation: 
for all $k \geq 1, t=1,\ldots,T, j=1,\ldots, M_t$, let 
$X_{t j}^k: \mathcal{X}_{t-1}  \rightrightarrows \mathcal{X}_t$ be the multifunction given by
\begin{equation} \label{Xtk}
X_{t j}^{k} ( x_{t-1} ) =\{ x_t \in \mathcal{X}_t: \;g_{t i j}^{k}(x_t , x_{t-1} ) \leq 0,i=1,\ldots,p,\;\displaystyle A_{t j} x_{t} + B_{t j} x_{t-1} = b_{t j} \},
\end{equation}
where $A_{t j}, B_{t j}, b_{t j}$ are respectively the realizations of $A_t, B_t$, and $b_t$ in $\xi_{t j}$
and let ${\underline{\mathfrak{Q}}}_{t j}^{k}: \mathcal{X}_{t-1}\rightarrow \mathbb{R}$
be the function
\begin{equation} \label{defxtkj0}
{\underline{\mathfrak{Q}}}_{t j}^{k}(x_{t-1}    )  =  \left\{
\begin{array}{l}
\displaystyle \min_{x_t} \; f_{t j}^{k}( x_t , x_{t-1} ) + \mathcal{Q}_{t+1}^{k}( x_t ) \\
x_t \in X_{t j}^{k}( x_{t-1} ).
\end{array}
\right.
\end{equation}

\begin{figure}
\centering
\begin{tabular}{c}
\includegraphics[scale=0.65]{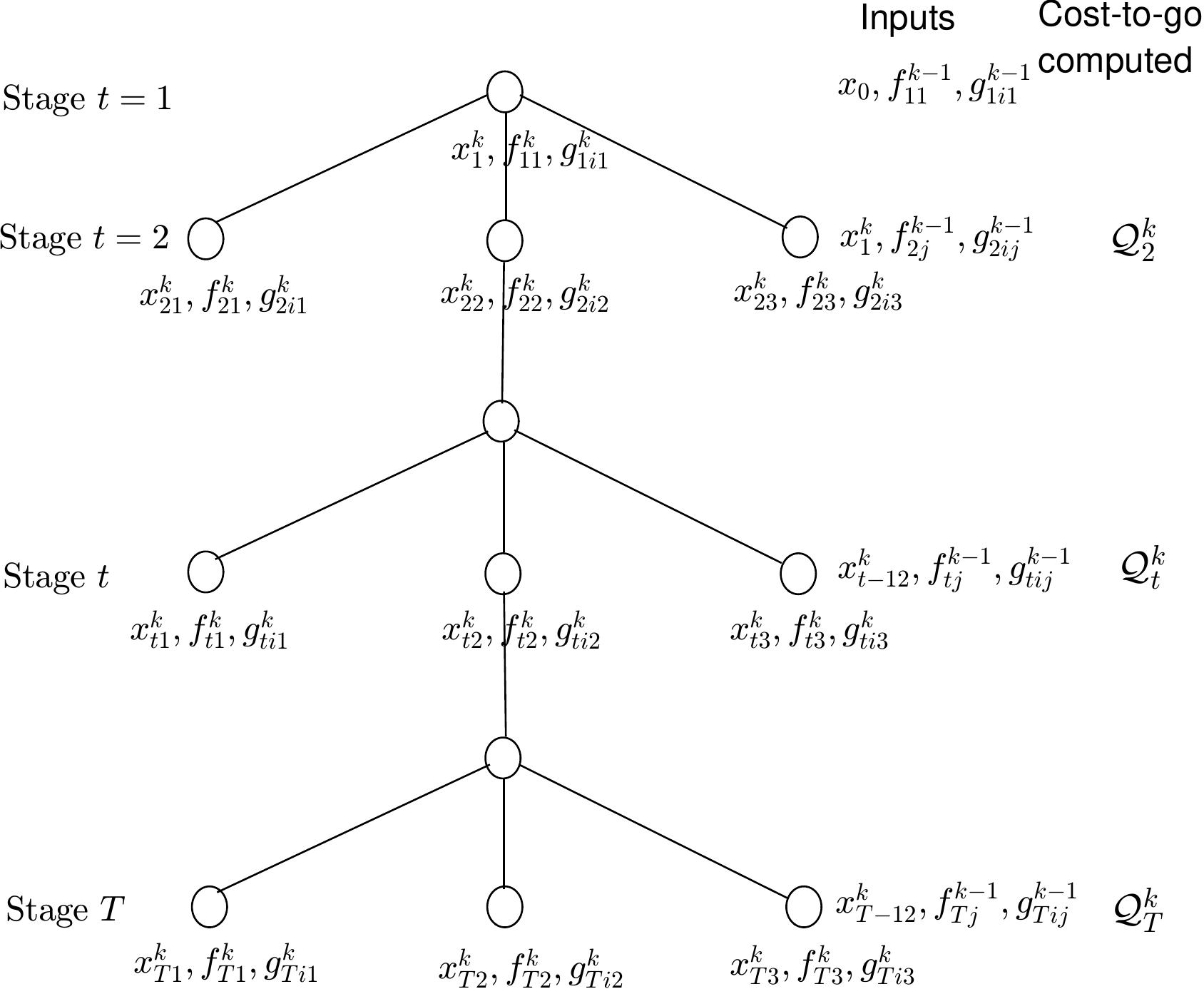}
\end{tabular}
\caption{\label{figureb} Variables updated in
iteration $k$ of StoDCuP.
In this representation, for simplicity,
every node had 3 child nodes and
the sampled scenario 
is $(\xi_1,\xi_{2 2},\ldots,\xi_{t2},\ldots,\xi_{T2})$
with corresponding decisions
$(x_1^k,x_{2 2}^k,\ldots,x_{t 2}^k,\ldots,x_{T 2}^k)$. The decisions computed for the nodes
of stage $t$ on this scenario are
denoted on this figure by
$x_{t 1}^k, x_{t 2}^k, x_{t 3}^k$
for nodes with realization of 
$\xi_t$ given by respectively
$\xi_{t 1}, \xi_{t 2}, \xi_{t 3}$.
For a given stage $t$, the 
inputs are $x_{t-1 2}^k$ (trial point),
$f_{t j}^{k-1}, g_{t i j}^{k-1}$ (for
all $i,j$) while the outputs
are $\mathcal{Q}_t^k$ and
for node with realization} $\xi_{t j}$ of $\xi_t$
decision $x_{t j}^k$
and functions $f_{t j}^k$, $g_{t i j}^k$.
\label{figvarstodcup}
\end{figure}

The detailed steps of the algorithm are described below (see the correspondence
with DCuP). We refer to Figure 
\ref{figvarstodcup} for the representation
of the variables updated in iteration
$k$ of StoDCuP.\\
\rule{\linewidth}{1pt}
\par {\textbf{Forward StoDCuP (Stochastic Dynamic Cutting Plane) with linearizations computed in a forward pass.}}\\
\rule{\linewidth}{1pt}
\begin{itemize}
\item[Step 0)] {\textbf{Initialization.}} For $t=1,\ldots,T$,
$j=1,\ldots,M_t$, $i=1,\ldots,p$,  take  
$f_{t j}^0, g_{t i j}^0: \mathcal{X}_t \small{\times} \mathcal{X}_{t-1} \to \mathbb{R}$
affine functions satisfying $f_{t j}^0 \leq f_t(\cdot,\cdot,\xi_{t j})$, 
$g_{t i j}^0 \leq g_{t i}(\cdot,\cdot,\xi_{t j})$,
and for $t=2,\ldots,T$, $\mathcal{Q}_t^0: \mathcal{X}_{t-1} \to \mathbb{R}$ 
is an affine function satisfying $\mathcal{Q}_t^0 \leq \mathcal{Q}_t$.
Set $x_{n_0} = x_0$, set the iteration count $k$ to 1, and $\mathcal{Q}_{T+1}^0 \equiv 0$. 
\item[Step 1)]{\textbf{Forward pass.}} Set
$\mathcal{C}_{T+1}^k=\mathcal{Q}_{T+1}^k \equiv 0$ and $x_0^k=x_0$.
\item[] Generate a sample $({\tilde \xi}_1^k, {\tilde \xi}_2^k,\ldots, {\tilde \xi}_T^k)$
of $(\xi_1, \xi_2,\ldots, \xi_T)$ corresponding to a set of nodes  $(n_1^k, n_2^k, \ldots, n_T^k)$ 
where $n_1^k=n_1$, and for $t \geq 2$, $n_t^k$ is a node of stage $t$, child of node $n_{t-1}^k$. Set $n_0^k=n_0$.
\item[] {\textbf{For}} $t=1,\ldots,T$, do:
\item[]\hspace*{0.7cm}Let $n=n_{t-1}^k$.
\item[]\hspace*{0.7cm}{\textbf{For }}every $m \in C(n)$,
\begin{itemize}
\item[]\hspace*{0.7cm}a) compute an optimal solution $x_m^k$ of
\begin{equation} \label{defxtkj}
{\underline{\mathfrak{Q}}}_{t j_t(m)}^{k-1}(x_{n}^k )  =  \left\{
\begin{array}{l}
\displaystyle \min_{x_m} \; f_{t j_t(m)}^{k-1}( x_m , x_{n}^{k} ) + \mathcal{Q}_{t+1}^{k-1}( x_m ) \\
x_m \in X_{t j_t(m)}^{k-1}( x_{n}^k  )
\end{array}
\right.
\end{equation}
\item[]\hspace*{0.7cm}where we recall that 
$$
\begin{array}{lcl}
f_{t j_t(m)}^{k-1}( x_m , x_{n}^{k} )&=&
\displaystyle \max_{\ell=0,\ldots,k-1} \; a_{t j_t(m)}^{\ell}  x_m  +  b_{t j_t(m)}^{\ell}  x_{n}^k  +  c_{t j_t(m)}^{\ell},\\
\mathcal{Q}_{t+1}^{k-1}( x_m )&=&\max_{0 \leq \ell \leq k-1} \;\mathcal{C}_{t+1}^{\ell} ( x_{m} ).
\end{array}
$$
\item[]\hspace*{0.7cm}b) Compute
function values and  
subgradients of convex functions $f_t(\cdot,\cdot,\xi_m)$ and $g_{t i}(\cdot,\cdot,\xi_m)$
\item[]\hspace*{0.8cm}at $(x_m^k , x_n^k )$ and let 
$\ell_{f_{t}(\cdot,\cdot,\xi_m)}((\cdot,\cdot); (x_m^{k}, x_{n}^{k} ) ) $ and 
$\ell_{g_{t i}(\cdot,\cdot,\xi_m)}((\cdot,\cdot); (x_m^{k}, x_{n}^{k} ))$
denote the 
\item[]\hspace*{0.7cm}corresponding linearizations.
\item[]\hspace*{0.7cm}c) Set
$$
\begin{array}{lcl}
f_{t j_t(m)}^{k} &=& \max\Big( f_{t j_t(m)}^{k-1} \,,\,  \ell_{f_{t}(\cdot,\cdot,\xi_m)}((\cdot,\cdot); (x_m^{k}, x_{n}^{k} ) )   \Big),\\
g_{ti}^{k} &= &\max\Big( g_{t i}^{k-1} \,,\,
\ell_{g_{t i}(\cdot,\cdot,\xi_m)}((\cdot,\cdot); (x_m^{k}, x_{n}^{k} )) \Big), \quad \forall i=1,\ldots,p.
\end{array}
$$
\item[]\hspace*{0.7cm}d) If $t \ge 2$ then compute 
$\beta_{t m}^{k} \in \partial {\underline{\mathfrak{Q}}}_{t j_t(m)}^{k-1}(x_{n}^k )$.
\end{itemize}
\hspace*{1cm}{\textbf{End For}}\\
\hspace*{1cm}{\textbf{If }}$t \geq 2$ compute: 
\begin{equation}\label{formulasbetathetastodcips}
\begin{array}{lcl}
\beta_{t}^{k} & =& \displaystyle \sum_{m \in C(n)} p_m \beta_{t m}^{k},\\
\theta_t^k &= &\displaystyle \sum_{m \in C(n)}  p_{m} 
\Big[ {\underline{\mathfrak{Q}}}_{t j_t(m)}^{k-1}(x_{n}^k ) 
- \langle \beta_{t m}^k  , x_n^k  \rangle 
   \Big],
\end{array}
\end{equation}
\begin{itemize}
\item[]\hspace*{0.5cm}yielding the new cut
$\mathcal{C}_t^k(x_{t-1})=\theta_t^k +
\langle \beta_t^k , x_{t-1} \rangle $
and $\mathcal{Q}_{t}^{k} = \max\{ \mathcal{Q}^{k-1}_{t}, \mathcal{C}_t^{k}  \}$.
\end{itemize}
\hspace*{1cm}{\textbf{End If}}\\
{\textbf{End For}}
\item[Step 2)] Do $k \leftarrow k+1$ and go to Step 2).
\end{itemize}
\rule{\linewidth}{1pt}

The following assumption will be made on the
sampling process in StoDCuP:\\
\par (H2) The samples of $(\xi_t)$ generated in StoDCuP are independent: $(\tilde \xi_2^k, \ldots, \tilde \xi_T^k)$ is a realization of
$\xi^k=(\xi_2^k, \ldots, \xi_T^k) \sim (\xi_2, \ldots,\xi_T)$ 
and $\xi^k, k\geq 1$, are independent.\\

 Recall that there are
$\displaystyle \prod_{t=2}^T M_t$ possible
scenarios (realizations) for
$(\xi_2,\ldots,\xi_T)$. Moreoever, by (H2),
for every such scenario $s_j$, $j=1,\ldots,\displaystyle \prod_{t=2}^T M_t$,
the events $E_n=\{\xi^n = s_j\}, n \geq 1$, are independent
and have a positive probability that only depends
on $j$. This gives $\sum_{n \geq 1} \mathbb{P}(E_n)=\infty$
and by the Borel-Cantelli lemma, this implies
that $\mathbb{P}(\displaystyle \varlimsup_{n \rightarrow \infty} E_n)=1$.
In what follows, several relations hold almost
surely. In this case, the corresponding
event of probability 1 is 
$\varlimsup_{n \rightarrow \infty} E_n$
corresponding to those realizations
of StoDCuP where every scenario $s_j$ is
sampled an infinite number of times.

\begin{rem}\label{remarkinfnode} As a consequence of the previous
observation, for every realization of StoDCuP,
and every node $n$ of the scenario tree,
an infinite number of scenarios sampled in StoDCuP
pass through that node $n$.\\
\end{rem}

We have for StoDCuP the following analogue of Lemma \ref{lipcontQt} for DCuP (the proof is similar
to the proof of Lemma \ref{lipcontQt}):
\begin{lemma}\label{lipcontQtstob}
Let Assumptions (H0) and (H1)-Sto hold. Then, the following statements hold for StoDCuP:
\begin{itemize}
\item[(a)] For $t=2,\ldots,T$, the sequence 
$\{\beta_t^k\}_{k=1}^\infty$ is almost surely  bounded.
\item[(b)] There exists   $L \ge 0$ such that
for each $t=2,\ldots,T$, $\mathcal{Q}^k_{t}$
 is $L$-Lipschitz continuous on $\mathcal{X}_{t-1}$ for every $k \ge 1$.
 \item[(c)] There exists $\hat L \ge 0$ such that
 for each $t=1,\ldots,T$, $j=1,\ldots,M_t$, functions $f_{t j}^k$ and $g_{t i j}^k$ are 
 $\hat L$-Lipschitz continuous on $\mathcal{X}_t \times \mathcal{X}_{t-1}$ for every $k \ge 1$ and $i=1,\ldots,p$.
 \end{itemize}
\end{lemma}

\begin{rem}[On  the cuts and linearizations computed] 
Assumption (H0) is fundamental for StoDCuP, due to the following claim:
\begin{itemize}
\item[(C)] StoDCuP builds a cut for $\mathcal{Q}_t, t=2,\ldots,T$, on any sampled scenario and
a single cut for each of the functions 
$f_t(\cdot,\cdot,\xi_{t j}), g_{t i}(\cdot,\cdot,\xi_{t j}), t=1,\ldots,T,j=1,\ldots,M_t$, $i=1,\ldots,p$, at each iteration.
\end{itemize}
The validity of the formulas of the cuts for $\mathcal{Q}_t$ will be checked in Lemma \ref{lemmacuts}.
The fact that a single cut is built for functions $f_t(\cdot,\cdot,\xi_{t j}), g_{t i}(\cdot,\cdot,\xi_{t j})$,
$i=1,\ldots,p$, $t=1,\ldots,T,j=1,\ldots,M_t$,
comes from the fact that at iteration $k$ and stage $t$ a cut is built for each of 
functions $f_t(\cdot,\cdot,\xi_m), g_{t i}(\cdot,\cdot,\xi_m)$, $i=1,\ldots,p$, $m \in C(n)$, where $n=n_{t-1}^k$,
and due to Assumption (H0), to each $m \in C(n)$, corresponds one and only one index $j=j_t(m)$
such that $\xi_m=\xi_{t j}=\xi_{t j_{t}(m)}$.
\end{rem}

\begin{rem} 
The algorithm can be extended to solve risk-averse problems.
It was shown in \cite{guiguesrom10} that dynamic programming equations can be written and that
SDDP can be applied for multistage stochastic linear optimization problems which
minimize some extended polyhedral risk measure of the cost.
As a special case, spectral risk measures are considered in  \cite{guiguesrom12} where
analytic formulas for some cut coefficients computed by SDDP are available. 
Similarly, StoDCuP can be extended to solve multistage nonlinear 
optimization problems with objective and constraint functions as in \eqref{pbtosolve}
if instead of minimizing the expected cost we minimize an extended polyhedral risk measure of the cost,
as long as Assumptions (H0) and (H1)-Sto are satisfied.
It is also possible to apply StoDCuP to solve 
risk-averse dynamic programming equations with nested conditional risk measures (see \cite{ruszshap2}, \cite{ruszshap1} for details on conditional risk mappings)
and objective and constraint functions as in \eqref{pbtosolve},
again,  as long as Assumptions (H0) and (H1)-Sto are satisfied.
Using SDDP in this risk-averse setting was proposed in \cite{shapsddp}.
\end{rem}

We can simulate the policy obtained after $k-1$ iterations of StoDCuP and define 
decisions $x_n^k$ at each node $n$ of the scenario tree as follows:\\
\rule{\linewidth}{1pt}
\par {\textbf{Simulation of StoDCuP after $k-1$ iterations.}}\\
\rule{\linewidth}{1pt}
Set $x_{n_0}^k=x_0$.\\
{\textbf{For }}$t=1,\ldots,T$,\\
\hspace*{0.7cm}{\textbf{For}} every node $n \in {\tt{Nodes}}(t-1)$,\\
\hspace*{1.4cm}{\textbf{For }}every $m \in C(n)$,\\
\hspace*{2.1cm}compute an optimal solution $x_m^k$ of
\begin{equation} \label{defxtkjsim}
{\underline{\mathfrak{Q}}}_{t j_t(m)}^{k-1}(x_{n}^k )  =  \left\{
\begin{array}{l}
\displaystyle \min_{x_m} \; f_{t j_t(m)}^{k-1}( x_m , x_{n}^k ) + \mathcal{Q}_{t+1}^{k-1}( x_m ) \\
x_m \in X_{t j_t(m)}^{k-1}( x_{n}^k  ).
\end{array}
\right.
\end{equation}
\hspace*{1.4cm}{\textbf{End For }}\\
\hspace*{0.7cm}{\textbf{End For}}\\
{\textbf{End For}}\\
\rule{\linewidth}{1pt}

We close this section providing 
in Lemma \ref{lemmaft} below
simple relations involving the linearizations of the objective and constraint functions
that will be used for the convergence analysis of StoDCuP.

\begin{lemma}\label{lemmaft}
Let Assumption (H1)-Sto hold. For every $t=1,\ldots,T$, $j=1,\ldots,M_t$, $i=1,\ldots,p$, we have almost surely
\begin{equation}\label{linearineq}
f_t(x_t, x_{t-1}, \xi_{t j}) \geq f_{t j}^k ( x_t, x_{t-1} ), g_{t i}(x_t, x_{t-1}, \xi_{t j}) \geq g_{t i j}^k ( x_t, x_{t-1} ),\;\forall k \geq 0, \forall x_t \in \mathcal{X}_t, \forall x_{t-1} \in \mathcal{X}_{t-1},
\end{equation}
and for every $k \geq 1$, 
\begin{equation}\label{approxfeasset}
X_{t}( x_{t-1} , \xi_{t j} ) \subset X_{t j}^k ( x_{t-1} ),\;\forall \;  x_{t-1} \in \mathcal{X}_{t-1}.
\end{equation}
For all $t=1,\ldots,T$, $i=1,\ldots,p$, for all $n \in {\tt{Nodes}}(t-1)$, for all $k \in \mathcal{S}_n$, we have for all $m \in C(n)$:
\begin{equation}\label{linearineq1}
f_t(x_m^k , x_{n}^k , \xi_{m}) = f_{t j_t(m)}^k ( x_m^k, x_{n}^k ) \mbox{ and }g_{t i}(x_m^k , x_{n}^k , \xi_{m}) = g_{t i j_t(m)}^k ( x_m^k, x_{n}^k ),\;\mbox{almost surely}
\end{equation}
For all $t=1,\ldots,T$, $i=1,\ldots,p$, for all $n \in {\tt{Nodes}}(t-1)$, for all $k \geq 1$, for all $m \in C(n)$,
we have
\begin{equation}\label{linearineq3}
g_{t i j_t(m)}^{k-1}(x_m^k, x_n^k) \leq 0,\;\mbox{almost surely},
\end{equation}
\begin{equation}\label{relinitsm}
0 \leq \max( g_{t i}(x_m^{k}, x_{n}^{k},\xi_m ) , 0 ) \leq g_{t i}( x_m^{k}, x_{n}^{k} , \xi_m ) - g_{t i j_t(m)}^{k-1}( x_m^{k}, x_{n}^{k}),\;\mbox{almost surely}
\end{equation}
\end{lemma}
\begin{proof} Let us show \eqref{linearineq}. The relation holds for $k=0$. 
Now let us fix $t \in \{1,\ldots,T\}$, $j \in \{1,\ldots,M_t\}$, $k \geq 1$ and 
$\ell \in \{1,\ldots,k\}$.
At iteration $\ell$, setting $n=n_{t-1}^{\ell}$, there exists one and only one node $m$ in the set  $C( n )$
such that $\xi_{m}=\xi_{t j}$ with $j=j_t(m)$ and by the subgradient inequality
for every $x_t \in \mathcal{X}_t$, for every $x_{t-1} \in \mathcal{X}_{t-1}$, we have 
\begin{equation}\label{firstflin}
\begin{array}{lll}
f_t( x_t , x_{t-1} , \xi_{t j} ) = f_t(x_t,x_{t-1},\xi_{m} )& \geq &  \ell_{f_t(\cdot,\cdot,\xi_m)}(x_t,x_{t-1};(x_m^{\ell}, x_n^{\ell})) \\ 
g_{t i}( x_t , x_{t-1} , \xi_{t j} ) = g_{t i}(x_t,x_{t-1},\xi_{m} )& \geq &  \ell_{g_{t i}(\cdot,\cdot,\xi_m)}(x_t,x_{t-1};(x_m^{\ell}, x_n^{\ell})),
\end{array}
\end{equation}
which, by Step c) of StoDCuP, immediately implies 
\eqref{linearineq}
and clearly inclusion \eqref{approxfeasset} is a consequence of \eqref{linearineq}.\hfill

Take $t \in \{1,\ldots,T\}$, $i \in \{1,\ldots,p\}$, take a node $n \in {\tt{Nodes}}(t-1)$ and
$k \in \mathcal{S}_n$. Then for any $m \in C(n)$, a linearization is built
for $f_t(\cdot,\cdot,\xi_m)$ and $g_{t i}(\cdot,\cdot,\xi_m)$ at $(x_m^k, x_n^k)$.
Therefore,
$$
\begin{array}{lcl}
f_t(x_m^k,x_n^k,\xi_m) &  \stackrel{\eqref{linearineq}}{\geq} &  f_{t j_t(m)}^k (x_m^k , x_n^k ) \\
& \geq  & \ell_{f_t(\cdot,\cdot,\xi_m)}(x_m^{k},x_{n}^{k};(x_m^{k}, x_n^{k})) =  f_t(x_m^k,x_n^k,\xi_m) \mbox{ since }n_{t-1}^k=n,\\
g_{t i}(x_m^k , x_n^k , \xi_m  )  & \stackrel{\eqref{linearineq}}{\geq}& g_{t i j_t(m)}^k (x_m^k , x_n^k ) \\
&  \geq     &   \ell_{g_{t i}(\cdot , \cdot,\xi_{m})}(x_m^k , x_n^k; ( x_m^{k}, x_{n}^{k} ) )=g_{t i}(x_m^k , x_n^k , \xi_m  ), \mbox{ since }n_{t-1}^k=n,
\end{array}
$$
and \eqref{linearineq1} follows. 

Relation \eqref{linearineq3} comes from the fact that $x_m^k \in X_{t j_t(m)}^{k-1}( x_{n}^k  )$
by definition of $x_m^k$ (see the simulation of StoDCuP).\hfill

Finally take a realization $\omega$ of StoDCuP. We show that 
\begin{equation}\label{relinitsmb}
0 \leq \max( g_{t  i}(x_m^{k}(\omega), x_{n}^{k}(\omega),\xi_m ) , 0 ) \leq g_{t i}( x_m^{k}(\omega), x_{n}^{k}(\omega) , \xi_m ) - g_{t i j_t(m)}^{k-1}(\omega)( x_m^{k}(\omega), x_{n}^{k}(\omega) ).
\end{equation}
If $g_{t i}(x_m^{k}(\omega), x_{n}^{k}(\omega), \xi_m ) \leq 0$ then \eqref{relinitsmb} holds because 
$g_{t i}(\cdot,\cdot,\xi_m) \geq  g_{t i j_t(m)}^{k-1}(\omega)$ and  if $g_{t  i}(x_m^{k}(\omega), x_{n}^{k}(\omega), \xi_m ) > 0$ then 
\eqref{relinitsmb} holds too because of inequality \eqref{linearineq3}.
Therefore, \eqref{relinitsmb} holds. \hfill
\end{proof}

\subsection{Implementation details for Steps b) and  d) of StoDCuP}\label{stodcupdeta}

In this section, we explain how to compute 
variables 
$a_{t j}^{\ell}$, $b_{t j}^{\ell}$, $c_{t j}^{\ell}$,
$d_{t i j}^{\ell}$, $e_{t i j}^{\ell}$, 
$h_{t i j}^{\ell}$, as well as cut coefficients
$\beta_{t m}^k$ in StoDCuP.

In Step b) of StoDCuP, we
compute  an arbitrary subgradient $[s_1 ; s_2]$ of 
convex function $f_t(\cdot,\cdot,\xi_m)$
at $(x_m^k , x_n^k )$ where 
$s_1, s_2 \in \mathbb{R}^n$
and set $a_{t j_t(m)}^k=s_1^{\top}$ and $b_{t j_t(m)}^k=s_2^{\top}$. For $i=1,\ldots,p$, we also 
compute an arbitrary subgradient 
$[s_{1 i} ; s_{2 i}]$ of convex function 
$g_{t i}(\cdot,\cdot,\xi_m)$
at $(x_m^k , x_n^k )$ where $s_{1 i}, s_{2 i} \in \mathbb{R}^n$; we set
$d_{t i j_t(m)}^k=s_{1 i}^{\top}$, 
$e_{t i j_t(m)}^k =s_{2 i}^{\top}$, and 
compute
$$
\begin{array}{l}
c_{t j_t(m)}^{k} =  f_t(x_m^k ,x_n^k,\xi_{m}) - a_{t j_t(m)}^k x_m^k  - b_{t j_t(m)}^k x_n^k,\\
h_{t i j_t(m)}^{k} =  g_{t i}(x_m^k ,x_n^k,\xi_{m}) - d_{t i j_t(m)}^k x_m^k - e_{t i j_t(m)}^k x_n^k.
\end{array}
$$

For the computation of $\beta_{t m}^k$,
it is convenient to
introduce $k \times n$  matrices
\begin{equation}\label{mat1s}
A_{t j}^{k} = \left[     
\begin{array}{c}
a_{t j}^0\\
a_{t j}^1 \\
\vdots \\
a_{t j}^{k}
\end{array}
\right],\;
B_{t j}^{k} = \left[    
\begin{array}{c}
b_{t j}^0\\
b_{t j}^1 \\
\vdots \\
b_{t j}^{k}
\end{array}
\right],\;
D_{t i j}^{k} = \left[    
\begin{array}{c}
d_{t i j}^0\\
d_{t i j}^1 \\
\vdots \\
d_{t i j}^{k}
\end{array}
\right],
E_{t i j}^{k} = \left[    
\begin{array}{c}
e_{t i j}^0\\
e_{t i j}^1 \\
\vdots \\
e_{t i j}^{k}
\end{array}
\right],\;
\beta_t^{0:k}=
\left[    
\begin{array}{c}
(\beta_{t}^0)^{\transp}\\
(\beta_t^1)^{\transp} \\
\vdots \\
(\beta_{t}^{k})^{\transp}
\end{array}
\right],
\end{equation}
$k$ dimensional vectors,
\begin{equation}\label{mat2s}
C_{t j}^{k} = \left[    
\begin{array}{c}
c_{t j}^0\\
c_{t j}^1 \\
\vdots \\
c_{t j}^{k}
\end{array}
\right],\;
H_{t i j}^{k} = \left[    
\begin{array}{c}
h_{t i j}^0\\
h_{t i j}^1 \\
\vdots \\
h_{t i j}^{k}
\end{array}
\right],\mbox{ and }
\theta_{t}^{0:k} = \left[    
\begin{array}{c}
\theta_{t}^0\\
\theta_{t}^1 \\
\vdots \\
\theta_{t}^{k}
\end{array}
\right],
\end{equation}
and matrices and vectors
\begin{equation}\label{mat3s}
D_{t j}^{k} = \left[    
\begin{array}{c}
D_{t 1 j}^k\\
D_{t 2 j}^k \\
\vdots \\
D_{t p j}^{k}
\end{array}
\right],\;
E_{t j}^{k} = \left[    
\begin{array}{c}
E_{t 1 j}^k\\
E_{t 2 j}^k \\
\vdots \\
E_{t p j}^{k}
\end{array}
\right],\;
H_{t j}^{k} = \left[    
\begin{array}{c}
H_{t 1 j}^k\\
H_{t 2 j}^k \\
\vdots \\
H_{t p j}^{k}
\end{array}
\right].
\end{equation}

If $\mathcal{X}_t=\{x_t : \mathbb{X}_t  x_t \geq {\bar x}_t\}$, we can write problem \eqref{defxtkj0} as
\begin{equation}\label{modelelinstodcup}
{\underline{\mathfrak{Q}}}_{t j}^{k}(x_{t-1} )  =  \left\{
\begin{array}{l}
\displaystyle \min_{x_t, f, \theta} \; f  + \theta \\
f {\textbf{e}} \geq A_{t j}^k x_t + B_{t j}^k x_{t-1} + C_{t j}^k,\\
A_{t j}  x_t + B_{t j} x_{t-1} = b_{t j},\\
D_{t j}^k x_t + E_{t j}^k x_{t-1} + H_{t j}^k \leq 0,\\
\theta {\textbf{e}} \geq \theta_{t+1}^{0:k} + \beta_{t+1}^{0:k} x_t,\;\mathbb{X}_t x_t \geq {\bar x}_t.
\end{array}
\right.
\end{equation}
Due to Assumption (H1)-Sto-4), for every $x_{t-1} \in \mathcal{X}_{t-1}$ and $j=1,\ldots,M_t$,
there exists $x_t \in \mathcal{X}_t$ such that $A_{t j}  x_t + B_{t j} x_{t-1} = b_{t j}$,
and $g_{t i}(x_t,x_{t-1},\xi_{t j}) \leq 0$, $i=1,\ldots,p$, which implies 
$g_{t i j}^k(x_t,x_{t-1}) \leq 0$, $i=1,\ldots,p$, $D_{t j}^k x_t + E_{t j}^k x_{t-1} + H_{t j}^k \leq 0$
and therefore the above problem \eqref{modelelinstodcup} is feasible.
Recalling (H1)-Sto-1), this linear program 
has a finite optimal value. Therefore
this optimal value is the optimal value of the dual problem and can be expressed as:
$$
{\underline{\mathfrak{Q}}}_{t j}^{k}(x_{t-1} )  =  \left\{
\begin{array}{l}
\displaystyle \max_{\alpha, \mu, \delta, \nu, \lambda} \;\alpha^{\top} (  B_{t j}^k x_{t-1}  +  C_{t j}^k ) + \mu^{\top} ( E_{t j}^k  x_{t-1}  +   H_{t j}^k  )   + \delta^{\top} \theta_{t+1}^{0:k} + \lambda^{\top} ( b_{t j}  - B_{t j} x_{t-1} ) + \nu^{\top} {\bar x}_t \\
(A_{t j}^k)^{\top} \alpha + (D_{t j}^k )^{\top} \mu + (\beta_{t+1}^{0:k})^{\top} \delta - \mathbb{X}_t^{\top}  \nu  - (A_{t j})^{\top} \lambda = 0,\\
{\textbf{e}}^{\top} \alpha = 1,\,{\textbf{e}}^{\top} \delta = 1, \alpha, \mu, \delta, \nu \geq 0.
\end{array}
\right.
$$
The above representation of 
${\underline{\mathfrak{Q}}}_{t j}^{k}$
allows us to obtain the formulas for
$\beta_{t m}^k, \beta_t^k, \theta_t^k$.
More precisely, consider iteration $k$ and stage $t \geq 2$
of the forward pass of StoDCuP.
Setting $n=n_{t-1}^k$ and $m \in C(n)$, let
$(\alpha_m^k, \mu_m^k, \delta_m^k, \nu_m^k, \lambda_m^k)$
be an optimal solution of the dual problem
{\small{
\begin{equation}\label{dualscupbetakm}
\begin{array}{l}
\displaystyle \max_{\alpha, \mu, \delta, \nu, \lambda} \;\alpha^{\top} (  B_{t j_t(m)}^{k-1} x_{n}^k  +  C_{t j_t(m)}^{k-1} ) + \mu^{\top} ( E_{t j_t(m)}^{k-1}  x_{n}^k  +   H_{t j_t(m)}^{k-1}     ) + \delta^{\top} \theta_{t+1}^{0:k-1} + \lambda^{\top} ( b_{t j_t(m)}  - B_{t j_t(m)} x_{n}^k ) + \nu^{\top} {\bar x}_t \\
(A_{t j_t(m)}^{k-1})^{\top} \alpha + (D_{t j_t(m)}^{k-1} )^{\top} \mu + (\beta_{t+1}^{0:k-1})^{\top} \delta - \mathbb{X}_t^{\top}  \nu  - (A_{t j_t(m)})^{\top} \lambda = 0,\\
{\textbf{e}}^{\top} \alpha = 1,\,{\textbf{e}}^{\top} \delta = 1, \alpha, \mu, \delta, \nu \geq 0.
\end{array}
\end{equation}
}}
By the discussion above, the optimal value
of \eqref{dualscupbetakm} is 
${\underline{\mathfrak{Q}}}_{t j}^{k}(x_n^k)$. We now show  in Lemma \ref{lemmacuts} below
that we can choose in StoDCuP,
\begin{equation}\label{formulasbetathetastodcip}
\begin{array}{lcl}
\beta_{t m}^{k} & =& \displaystyle \sum_{m \in C(n)} (B_{t j_t(m)}^{k-1})^{\top} \alpha_m^{k}  + (E_{t j_t(m)}^{k-1} )^{\top} \mu_m^k - B_{t j_t(m)}^{\top} \lambda_m^k,\\
\beta_{t}^{k} & =& \displaystyle \sum_{m \in C(n)} p_m \Big[  (B_{t j_t(m)}^{k-1})^{\top} \alpha_m^{k}  + (E_{t j_t(m)}^{k-1} )^{\top} \mu_m^k - B_{t j_t(m)}^{\top} \lambda_m^k \Big],\\
\theta_t^k &= &\displaystyle \sum_{m \in C(n)}  p_{m} \Big[  \langle \alpha_m^k ,  C_{t j_t(m)}^{k-1}  \rangle  + \langle  \mu_m^k , H_{t j_t(m)}^{k-1} \rangle  +  \langle  \delta_m^k ,  \theta_{t+1}^{0:k-1} \rangle  +  \langle \lambda_m^k ,  b_{t j_t(m)} \rangle  + \langle \nu_m^k  ,  {\bar x}_t \rangle  \Big].
\end{array}
\end{equation}
More precisely, we
show in Lemma \ref{lemmacuts}
that computations
\eqref{formulasbetathetastodcip} provide
valid cuts (lower bounding functions $\mathcal{C}_t^k$) 
for $\mathcal{Q}_t$, in particular that 
$\beta_{t m}^k \in \partial {\underline{\mathfrak{Q}}}_{t j_t(m)}^{k-1}(x_{n}^k)$, as required by Step d)
of StoDCuP, and $\beta_t^k \in \partial \mathcal{Q}_t(x_n^k)$:

\begin{lemma}\label{lemmacuts} Let Assumptions (H0) and (H1)-Sto hold. For every $t=2,\ldots,T+1$, for every $k \geq 1$,
we have almost surely
\begin{equation}\label{validcutsQ}
\mathcal{Q}_t( x_{t-1} ) \geq \mathcal{C}_t^k ( x_{t-1} ) \mbox{ and }\mathcal{Q}_t( x_{t-1} ) \geq \mathcal{Q}_t^k ( x_{t-1} ),\;\forall  x_{t-1} \in \mathcal{X}_{t-1}.
\end{equation}
For all $t=1,\ldots,T$, $j=1,\ldots,M_t$, for every $k \geq 1$, we have almost surely
\begin{equation}\label{approxvaluefunc}
{\underline{\mathfrak{Q}}}_{t j}^{k}(x_{t-1}  ) \leq \mathfrak{Q}_t ( x_{t-1} , \xi_{t j} ) \mbox{ for all }x_{t-1} \in \mathcal{X}_{t-1}.
\end{equation}
For all $t=2,\ldots,T$, for every $k \geq 1$, 
defining
$
{\underline{\mathcal{Q}}}_t^{k-1}(x_{n}^k ) = \sum_{j=1}^{M_t} p_{t j} {\underline{\mathfrak{Q}}}_{t j}^{k-1}(x_{n}^k ),  
$
we have for every $n \in {\tt{Nodes}}(t-1)$ and for 
all $k \in \mathcal{S}_n$:
\begin{equation}\label{erroroncutQ}
{\underline{\mathcal{Q}}}_t^{k-1}(x_{n}^k )=\mathcal{C}_t^k ( x_n^k ),\;\mbox{almost surely}
\end{equation}
\end{lemma}
\begin{proof} Let us show \eqref{validcutsQ}-\eqref{approxvaluefunc} by backward induction on $t$. 
Relation \eqref{validcutsQ} clearly holds for $t=T+1$.
Now assume that for some $t \in \{1,\ldots,T\}$, we have
$\mathcal{Q}_{t+1}( x_{t} ) \geq \mathcal{Q}_{t+1}^k ( x_{t} )$
for all $x_{t} \in \mathcal{X}_{t}$ and all $k \geq 1$. 
Using Lemma \ref{lemmaft}, we have for all $k \geq 1$,
for all $j=1,\ldots,M_t$, for all $x_t \in \mathcal{X}_t$, $x_{t-1} \in \mathcal{X}_{t-1}$,
that $f_{t j}^k(x_t,x_{t-1}) \leq f_{t}(x_t,x_{t-1},\xi_{t j})$ and  
$X_{t} (x_{t-1},\xi_{t j}) \subset X_{t j}^k (x_{t-1})$, which, together with the induction hypothesis $\mathcal{Q}_{t+1}^k \leq \mathcal{Q}_{t+1}$,
implies 
\begin{equation}\label{proofapproxvaluefunc}
{\underline{\mathfrak{Q}}}_{t j}^{k}(x_{t-1}  ) \leq \mathfrak{Q}_t ( x_{t-1} , \xi_{t j} ) \mbox{ for all }x_{t-1} \in \mathcal{X}_{t-1},
\end{equation}
i.e., \eqref{approxvaluefunc}.
Now observe that due to Assumption (H1)-Sto, for every $x_{t-1} \in \mathcal{X}_{t-1}$, the optimization problem 
$$
{\underline{\mathfrak{Q}}}_{t j}^{k-1}(x_{t-1}    )  =  \left\{
\begin{array}{l}
\displaystyle \min_{x_t} \; f_{t j}^{k-1}( x_t , x_{t-1} ) + \mathcal{Q}_{t+1}^{k-1}( x_t ) \\
x_t \in X_{t j}^{k-1}( x_{t-1} ),
\end{array}
\right.
$$
is a linear program with feasible set that is
bounded (since $\mathcal{X}_t$ is compact) and nonempty (it contains the nonempty set $X_t(x_{t-1})$). Therefore it has a finite
optimal value which is also the optimal value of the dual problem given by
\begin{equation}\label{approxvaluefuncdual}
{\underline{\mathfrak{Q}}}_{t j}^{k-1}(x_{t-1}    )  =  \left\{
\begin{array}{l}
\displaystyle \max_{\alpha, \mu, \delta, \nu, \lambda} \;\mathcal{D}_{t j}^{k-1}(\alpha, \mu, \delta, \nu, \lambda;x_{t-1})\\
(A_{t j}^{k-1})^{\top} \alpha + (D_{t j}^{k-1} )^{\top} \mu + (\beta_{t+1}^{0:k-1})^{\top} \delta - \mathbb{X}_t^{\top}  \nu  - (A_{t j})^{\top} \lambda = 0,\\
{\textbf{e}}^{\top} \alpha = 1,\,{\textbf{e}}^{\top} \delta = 1, \alpha, \mu, \delta, \nu \geq 0,
\end{array}
\right.
\end{equation}
where 
{\small{
$$
\mathcal{D}_{t j}^{k-1}(\alpha, \mu, \delta, \nu, \lambda;x_{t-1})=
\alpha^{\top} (  B_{t j}^{k-1} x_{t-1}  +  C_{t j}^{k-1} ) + \mu^{\top} ( E_{t j}^{k-1}  x_{t-1}  +   H_{t j}^{k-1}     ) + \delta^{\top} \theta_{t+1}^{0:k-1} + \lambda^{\top} ( b_{t j}  - B_{t j} x_{t-1} ) + \nu^{\top} {\bar x}_t.
$$
}}
Now assume that $t \geq 2$. Let us take $m \in C(n_{t-1}^k )$. Recall that $j_t(m)$ is the unique index $j$ such that 
$\xi_{t j}=\xi_m$.
Clearly $(\alpha_m^k, \mu_m^k, \delta_m^k, \nu_m^k, \lambda_m^k)$ is feasible for dual problem \eqref{approxvaluefuncdual} 
written for $j=j_t(m)$ and therefore
for any $x_{t-1} \in \mathcal{X}_{t-1}$ we have
\begin{equation}\label{bounddual}
{\underline{\mathfrak{Q}}}_{t j_t(m)}^{k-1}(x_{t-1}) \geq \mathcal{D}_{t j_t(m)}^{k-1}(\alpha_m^k , \mu_m^k , \delta_m^k , \nu_m^k , \lambda_m^k ;x_{t-1}),
\end{equation}
which gives 
$$
\begin{array}{lcl}
\mathcal{Q}_t(x_{t-1} ) & = & \displaystyle \sum_{j=1}^{M_t} p_{t j} \mathfrak{Q}_{t}(x_{t-1}, \xi_{t j} )\\
& \stackrel{(H0)}{=}  & \displaystyle \sum_{m \in C( n_{t-1}^k  )} p_{m} \mathfrak{Q}_{t}(x_{t-1}, \xi_{m} )\\
& =  & \displaystyle \sum_{m \in C( n_{t-1}^k  )} p_{m} \mathfrak{Q}_{t}(x_{t-1}, \xi_{t j_t(m)} )\\
& \stackrel{\eqref{proofapproxvaluefunc}}{\geq}  & \displaystyle \sum_{m \in C( n_{t-1}^k  )} p_{m}  {\underline{\mathfrak{Q}}}_{t j_t(m)}^{k-1}(x_{t-1} )\\
& \stackrel{\eqref{bounddual}}{\geq} & \displaystyle \sum_{m \in C( n_{t-1}^k  )} p_{m} \mathcal{D}_{t j_t(m)}^{k-1}(\alpha_m^k , \mu_m^k , \delta_m^k , \nu_m^k , \lambda_m^k ;x_{t-1}) \\
& = & \mathcal{C}_t^k(  x_{t-1} ),
\end{array}
$$
for every $x_{t-1} \in \mathcal{X}_{t-1}$, where for the last equality, we have used \eqref{formulacut} and \eqref{formulasbetathetastodcip}. Therefore we have shown \eqref{validcutsQ}.

Now take $n \in {\tt{Nodes}}(t-1)$ and $k \in \mathcal{S}_n$.
Then by definition of $(\alpha_m^k, \mu_m^k, \delta_m^k, \nu_m^k, \lambda_m^k)$ and of $\mathcal{C}_t^k$, we get 
for any $m \in C(n)$:
\begin{equation}\label{eq1fin}
{\underline{\mathfrak{Q}}}_{t j_t(m)}^{k-1}(x_{n}^k    ) = \mathcal{D}_{t j_t(m)}^{k-1}(\alpha_m^k , \mu_m^k , \delta_m^k , \nu_m^k , \lambda_m^k ;x_{n}^k) 
\end{equation}
and
\begin{equation}\label{eq2fin}
\mathcal{C}_t^k ( x_n^k ) = \displaystyle \sum_{m \in C(n)} p_m  \mathcal{D}_{t j_t(m)}^{k-1}(\alpha_m^k , \mu_m^k , \delta_m^k , \nu_m^k , \lambda_m^k ;x_{n}^k)=
\displaystyle \sum_{m \in C(n)} p_m  {\underline{\mathfrak{Q}}}_{t j_t(m)}^{k-1}(x_{n}^k    ) ={\underline{\mathcal{Q}}}_t^{k-1}(x_{n}^k ).
\end{equation}
 \hfill 
\end{proof}

\subsection{Convergence analysis}

In what follows, if the stage associated to node $n$ is $\tau(n)$, we use the notation
\begin{equation}
\mathcal{S}_n = \{k \in \mathbb{N}^{*} : n_{\tau(n)}^k = n  \}.
\end{equation}
In other words, $\mathcal{S}_n$
the set of iterations $k$ where the sampled scenario passes through node $n$.

\begin{thm}[Convergence of StoDCuP]\label{convproofStoDCuP} 
Let Assumption (H0), (H1)-Sto, and (H2) hold. Then
\par (i) for every $t=1,\ldots,T$,$i=1,\ldots,p$, almost surely
\begin{equation}\label{limgtas}
\lim_{k \rightarrow +\infty} \max( g_{t i}(x_m^{k}, x_{n}^{k}, \xi_m ) , 0 )  = 0,\;\forall m \in 
{\tt{Nodes}}(t), n=\mathcal{P}(m). 
\end{equation}
For all $t=2,\ldots,T+1$, for all node $n \in {\tt{Nodes}}(t-1)$, we have almost surely
\begin{equation}\label{htstodcup}
\mathcal{H}(t):\;\lim_{k \rightarrow + \infty} \mathcal{Q}_t( x_{n}^{k} ) -  \mathcal{Q}_t^{k} ( x_{n}^{k} )=0.
\end{equation}

\par (ii) The limit of the sequence of first stage problems optimal values 
$( f_{1 1}^{k-1}(x_{n_1}^{k} , x_{0} ) + \mathcal{Q}_2^{k-1}( x_{n_1}^k )  )_{k \geq 1}$ is the optimal value
$\mathcal{Q}_1( x_0 ) $ of \eqref{firststodp} and 
any accumulation point of the sequence $(x_{n_1}^{k})$ is an optimal solution to the first stage problem \eqref{firststodp}.
\end{thm}
\begin{proof} 
We first show \eqref{limgtas}. Let us fix $t \in \{1,\ldots,T\}$, 
$i \in \{1,\ldots,p\}$, $m \in {\tt{Nodes}}(t)$, $n=\mathcal{P}(m)$.
Recall from Lemma \ref{lemmaft} that
\begin{equation}\label{relinitsm}
0 \leq \max( g_{t i}(x_m^{k}, x_{n}^{k},\xi_m ) , 0 ) \leq g_{t i}( x_m^{k}, x_{n}^{k} , \xi_m ) - g_{t i j_t(m)}^{k-1}( x_m^{k}, x_{n}^{k} ). 
\end{equation}
We now show that 
\begin{equation}\label{proofg}
\lim_{k \rightarrow +\infty}  g_{t i}( x_m^{k}, x_{n}^{k},\xi_m ) - g_{t i j_t(m)}^{k-1}( x_m^{k}, x_{n}^{k} ) =0,
\end{equation}
which will show \eqref{limgtas} due to relation \eqref{relinitsm}.

Recalling that set $\mathcal{S}_n$ is infinite (see
Remark \ref{remarkinfnode}), we denote by $k(1), k(2),\ldots,$ the iterations in $\mathcal{S}_n$ with $k(i) < k(i+1)$:
$\mathcal{S}_n=\{k(1),k(2),k(3),\ldots\}$. Let us first show that we have
\begin{equation}\label{limgtasSn}
\lim_{k \rightarrow +\infty, k \in \mathcal{S}_n} \max( g_{t i}(x_m^{k}, x_{n}^{k},\xi_m ) , 0 ) =  0. 
\end{equation}
For all $\ell \geq 1$, relation \eqref{linearineq1} gives
\begin{equation}\label{limgtas1}
g_{t i}(x_m^{k(\ell)} , x_{n}^{k(\ell)} , \xi_{m}) = g_{t i j_t(m)}^{k(\ell)} ( x_m^{k(\ell)}, x_{n}^{k(\ell)} ).
\end{equation}
Let us now apply Lemma \ref{techlemmasequence} to $y^{\ell}=(x_m^{k(\ell)} , x_{n}^{k(\ell)})$, sequence $f^{\ell}=g_{t i j_t(m)}^{k(\ell)}$,
and $f=g_{t i}(\cdot,\cdot,\xi_m)$ (observe that the assumptions of the lemma are satisfied with $k_0=1$). Since 
$$
\lim_{\ell \rightarrow +\infty} f( y^{\ell} ) - f^{\ell}( y^{\ell} ) = 0,
$$
we deduce that
\begin{equation}\label{finalgbef}
\lim_{\ell \rightarrow +\infty} f( y^{\ell} ) - f^{\ell-1}( y^{\ell} ) = 
\lim_{\ell \rightarrow +\infty} g_{t i}( x_m^{k(\ell)} , x_{n}^{k(\ell)}, \xi_m  ) -  g_{t i j_t(m)}^{k(\ell - 1) } (x_m^{k(\ell)} , x_{n}^{k(\ell)} )=0.
\end{equation}
Since $k(\ell) \geq 1 + k(\ell - 1)$, we have 
$0 \leq g_{t i}(\cdot,\cdot,\xi_m) - g_{t i j_t(m)}^{k(\ell)-1}(\cdot,\cdot)  \leq g_{t i}(\cdot,\cdot,\xi_m) - g_{t i j_t(m)}^{k(\ell - 1)}(\cdot,\cdot)$ and therefore \eqref{finalgbef} implies
\begin{equation}\label{finalgbef0}
\lim_{\ell \rightarrow +\infty} g_{t i}( x_m^{k(\ell)} , x_{n}^{k(\ell)}, \xi_m  ) -  g_{t i j_t(m)}^{k(\ell)-1} (x_m^{k(\ell)} , x_{n}^{k(\ell)} )
=\lim_{k \rightarrow +\infty, k \in \mathcal{S}_n} g_{t i}( x_m^{k} , x_{n}^{k}, \xi_m  ) -  g_{t i j_t(m)}^{k-1} (x_m^{k} , x_{n}^{k} )=0.
\end{equation}
Finally, we show in the Appendix that
\begin{equation}\label{proofg0}
\lim_{k \rightarrow +\infty, k \notin \mathcal{S}_n}  g_{t i}( x_m^{k}, x_{n}^{k} ) - g_{t i j_t(m)}^{k-1}( x_m^{k}, x_{n}^{k} ) =0,
\end{equation}
which achieves the proof of \eqref{proofg} and therefore of \eqref{limgtas}.

Let us now show $\mathcal{H}(t)$ by backward induction on $t$. $\mathcal{H}(T+1)$ holds
since $\mathcal{Q}_{T+1}=\mathcal{Q}_{T+1}^k$.
Assume now that $\mathcal{H}(t+1)$ holds for some $t \in \{2,\ldots,T\}$
and let us show that $\mathcal{H}(t)$ holds.
Take a node $n \in {\tt{Nodes}}(t-1)$ and let us denote again by
$k(1), k(2),\ldots,$ the iterations in $\mathcal{S}_n$ with $k(i) < k(i+1)$:
$\mathcal{S}_n=\{k(1),k(2),k(3),\ldots\}$. Let us first show that 
\begin{equation}\label{limitinSn}
\lim_{k \rightarrow + \infty, k \in \mathcal{S}_n} \mathcal{Q}_t( x_{n}^{k} ) -  \mathcal{Q}_t^{k} ( x_{n}^{k} )=
\lim_{\ell \rightarrow + \infty} \mathcal{Q}_t( x_{n}^{k(\ell)} ) -  \mathcal{Q}_t^{k(\ell)} ( x_{n}^{k(\ell)} )=0. 
\end{equation}
By definition of $\mathcal{Q}_t^{k(\ell)}$, we have $\mathcal{Q}_t^{k(\ell)} ( x_n^{k(\ell)} ) \geq \mathcal{C}_t^{k(\ell)} ( x_n^{k(\ell)} ) $
and therefore for all $\ell \geq 1$ we get:
\begin{equation}\label{inductfirst}
\begin{array}{lcl}
0 \leq \mathcal{Q}_t ( x_n^{k(\ell)} ) - \mathcal{Q}_t^{k(\ell)} ( x_n^{k(\ell)} ) & \leq &  \mathcal{Q}_t ( x_n^{k(\ell)} ) - \mathcal{C}_t^{k(\ell)} ( x_n^{k(\ell)} ) \\
& = &  \mathcal{Q}_t ( x_n^{k(\ell)} ) - {\underline{\mathcal{Q}}}_t^{k(\ell)-1}(x_{n}^{k(\ell)} ),\\
& = &  \displaystyle  \sum_{m \in C( n_{t-1}^{k(\ell)}  ) }  p_{m} \Big[ \mathfrak{Q}_t( x_{n}^{k(\ell)} , \xi_m    )    -  {\underline{\mathfrak{Q}}}_{t j_t(m)}^{k(\ell)-1}(x_{n}^{k(\ell)} )\Big].
\end{array}
\end{equation} 
By definiton of $x_m^k$, we have
\begin{equation}\label{defepssolxmk}
{\underline{\mathfrak{Q}}}_{t j_t(m)}^{k(\ell)-1}(x_{n}^{k(\ell)} )
= f_{t j_t(m)}^{k(\ell)-1}(x_m^{k(\ell)}, x_n^{k(\ell)} ) +\mathcal{Q}_{t+1}^{k(\ell)-1}( x_m^{k(\ell)} ),
\end{equation}
which, plugged into \eqref{inductfirst}, gives
\begin{equation}\label{second}
0 \leq \mathcal{Q}_t ( x_n^{k(\ell)} ) - \mathcal{Q}_t^{k(\ell)} ( x_n^{k(\ell)} )  
\leq \displaystyle  \sum_{m \in C( n_{t-1}^{k(\ell)}  ) }  p_{m} \Big[ 
\mathfrak{Q}_t( x_{n}^{k(\ell)} , \xi_m    )    -  f_{t j_t(m)}^{k(\ell)-1}(x_{m}^{k(\ell)}, x_n^{k(\ell)} )  -  \mathcal{Q}_{t+1}^{k(\ell)-1}( x_m^{k(\ell)}   )\Big].
\end{equation} 
Let us apply Lemma \ref{techlemmasequence} to $y^{\ell}=(x_m^{k(\ell)} , x_{n}^{k(\ell)})$, sequence $f^{\ell}=f_{t j_t(m)}^{k(\ell)}$,
and $f=f_{t}(\cdot,\cdot,\xi_m)$ (observe that the assumptions of the lemma are satisfied).
Due to \eqref{linearineq1}, we have
$$
\lim_{\ell \rightarrow +\infty} f(y^{\ell}) - f^{\ell}( y^{\ell} )  = 0
$$
and therefore 
\begin{equation}\label{limitffirst}
\lim_{\ell \rightarrow +\infty} f(y^{\ell}) - f^{\ell-1}( y^{\ell} )= 
\lim_{\ell \rightarrow +\infty} f_{t}(x_m^{k(\ell)} , x_{n}^{k(\ell)},\xi_m)-  f_{t j_t  (m)}^{k(\ell-1)}(x_m^{k(\ell)} , x_{n}^{k(\ell)})=0.
\end{equation}
Since $k(\ell) \geq k( \ell - 1 ) +1$, we have $0 \leq  f_{t}(x_m^{k(\ell)} , x_{n}^{k(\ell)},\xi_m)-  f_{t j_t(m)}^{k(\ell)-1}(x_m^{k(\ell)} , x_{n}^{k(\ell)})  \leq f_{t}(x_m^{k(\ell)} , x_{n}^{k(\ell)},\xi_m)-  f_{t j_t(m)}^{k(\ell-1)}(x_m^{k(\ell)} , x_{n}^{k(\ell)})$
which combined with \eqref{limitffirst} gives
\begin{equation}\label{limitsecond}
\lim_{\ell \rightarrow +\infty} f_{t}(x_m^{k(\ell)} , x_{n}^{k(\ell)},\xi_m)-  f_{t j_t(m)}^{k(\ell)-1}(x_m^{k(\ell)} , x_{n}^{k(\ell)})=0.
\end{equation}
Using \eqref{defepssolxmk} and \eqref{approxvaluefunc}, we get
$$
f_{t j_t(m)}^{k(\ell)-1}(x_m^{k(\ell)}, x_n^{k(\ell)} ) +\mathcal{Q}_{t+1}^{k(\ell)-1}(x_m^{k(\ell)}) ={\underline{\mathfrak{Q}}}_{t j_t(m)}^{k(\ell)-1}(x_{n}^{k(\ell)}  ) \leq \mathfrak{Q}_t ( x_{n}^{k(\ell)} , \xi_{m} ).
$$
Therefore the sequence 
$( f_{t j_t(m)}^{k(\ell)-1}(x_m^{k(\ell)}, x_n^{k(\ell)} ) +\mathcal{Q}_{t+1}^{k(\ell)-1}(x_m^{k(\ell)})- \mathfrak{Q}_t ( x_{n}^{k(\ell)} , \xi_{m} ) )_{\ell \geq 1}$
is bounded and has a finite limit superior which satisfies
\begin{equation}\label{limsup1}
\varlimsup_{\ell \rightarrow +\infty}  f_{t j_t(m)}^{k(\ell)-1}(x_m^{k(\ell)}, x_n^{k(\ell)} ) +\mathcal{Q}_{t+1}^{k(\ell)-1}(x_m^{k(\ell)})  - \mathfrak{Q}_t ( x_{n}^{k(\ell)} , \xi_{m} )    \leq 0.
\end{equation}
Applying Lemma \ref{techlemmasequence} to $y^{\ell}=x_{m}^{k(\ell)}$, sequence $f^{\ell}=\mathcal{Q}_{t+1}^{k(\ell)}$,
and $f=\mathcal{Q}_{t+1}$ (observe that the assumptions of the lemma are satisfied), since from the induction hypothesis we know that
$$
\lim_{\ell \rightarrow +\infty} f(y^{\ell}) - f^{\ell}( y^{\ell} )  = 0
$$
we deduce that
\begin{equation}\label{limitffirst1}
\lim_{\ell \rightarrow +\infty} f(y^{\ell}) - f^{\ell-1}( y^{\ell} )= 
\lim_{\ell \rightarrow +\infty} \mathcal{Q}_{t+1}(  x_m^{k(\ell)} )  - \mathcal{Q}_{t+1}^{k(\ell-1)}(  x_m^{k(\ell)}  )  =0.
\end{equation}
Since $k(\ell) \geq k( \ell - 1 ) +1$, we have $0 \leq \mathcal{Q}_{t+1}(  x_m^{k(\ell)} )  - \mathcal{Q}_{t+1}^{k(\ell)-1}(  x_m^{k(\ell)}  ) \leq \mathcal{Q}_{t+1}(  x_m^{k(\ell)} )  - \mathcal{Q}_{t+1}^{k(\ell-1)}(  x_m^{k(\ell)}  )$, which combines with
\eqref{limitffirst1} to give
\begin{equation}\label{limitsecondf}
 \lim_{\ell \rightarrow +\infty} \mathcal{Q}_{t+1}(  x_m^{k(\ell)} )  - \mathcal{Q}_{t+1}^{k(\ell)-1}(  x_m^{k(\ell)}  )  =0.
\end{equation}
Combining \eqref{limitsecond}, \eqref{limsup1}, and \eqref{limitsecondf}, we obtain
\begin{equation}\label{limsup1b}
\varlimsup_{\ell \rightarrow +\infty}  f_{t}(x_m^{k(\ell)}, x_n^{k(\ell)}, \xi_m ) +\mathcal{Q}_{t+1}(x_m^{k(\ell)}) - \mathfrak{Q}_t ( x_{n}^{k(\ell)} , \xi_{m} )    \leq 0.
\end{equation}
Let us now show by contradiction that 
\begin{equation}\label{liminfT}
\varliminf_{k \rightarrow +\infty} f_{t}(x_m^{k(\ell)}, x_n^{k(\ell)}, \xi_m ) +\mathcal{Q}_{t+1}(x_n^{k(\ell)}) - \mathfrak{Q}_t ( x_{n}^{k(\ell)} , \xi_{m} )    \geq 0.
\end{equation}
Assume that \eqref{liminfT} does not hold.
Using the fact that sequence $( x_m^{k}, x_{n}^{k}  )_{k \in \mathcal{S}_n}$
belongs to the compact set $\mathcal{X}_t \small{\times} \mathcal{X}_{t-1}$, 
and the lower semicontinuity of $f_t(\cdot,\cdot,\xi_m)$, $g_t(\cdot,\cdot,\xi_m)$,
$\mathcal{Q}_{t}$, $\mathfrak{Q}_{t}(\cdot,\xi_m)$, there is a subsequence 
$( x_m^{k}, x_{n}^{k}  )_{k \in K}$ with $K \subset \mathcal{S}_n$
converging to some $(\bar x_m, \bar x_n) \in \mathcal{X}_t \small{\times} \mathcal{X}_{t-1}$
such that 
$$
\begin{array}{l}
f_{t}({\bar x}_m , {\bar x}_n, \xi_m ) +\mathcal{Q}_{t+1}({\bar x}_n ) - \mathfrak{Q}_t ( \bar x_n, \xi_{m} )< 0
\end{array}
$$
and $\bar x_m \in X_t( \bar x_n , \xi_m)$. This is in contradiction with the definition of $\mathfrak{Q}_t$. Therefore we must have
$$
\begin{array}{lll}
0 & = & \lim_{\ell \rightarrow +\infty}  f_{t}(x_m^{k(\ell)}, x_n^{k(\ell)}, \xi_m ) +\mathcal{Q}_{t+1}(x_m^{k(\ell)}) - \mathfrak{Q}_t ( x_{n}^{k(\ell)} , \xi_{m} )\\
& = &  \lim_{\ell \rightarrow +\infty}  f_{t j_t(m)}^{k(\ell)-1}(x_m^{k(\ell)}, x_n^{k(\ell)}) +\mathcal{Q}_{t+1}^{k(\ell)-1}(x_m^{k(\ell)}) - \mathfrak{Q}_t ( x_{n}^{k(\ell)} , \xi_{m} )\\
\end{array}
$$
which, plugged into \eqref{second} gives 
\begin{equation}\label{limitQinSn}
\lim_{k \rightarrow +\infty, k \in \mathcal{S}_n} \mathcal{Q}_t( x_n^k ) - \mathcal{Q}_t^k ( x_n^k ) = 0. 
\end{equation}
Finally, we show in the Appendix that
\begin{equation}\label{limitQNotinSn}
\lim_{k \rightarrow +\infty, k \notin \mathcal{S}_n} \mathcal{Q}_t( x_n^k ) - \mathcal{Q}_t^k ( x_n^k ) = 0,
\end{equation}
which achieves the proof of $\mathcal{H}(t)$.

\par (ii) The proof of (ii) can easily be obtained from (i), see Theorem 4.1-(ii) in \cite{guiguessiopt2016} for details.
\hfill
\end{proof}

\begin{rem}[Stopping criterion]
The stopping criterion is similar to SDDP.
We can stop the algorithm when the gap $\frac{Ub-Lb}{Ub}$ is less than a threshold, for instance
$5\%$, where $Ub$ and $Lb$ are upper and lower bounds, respectively, defined as follows. 
Due to Lemma \ref{lemmacuts}, we can take as a lower
bound on the optimal value of problem \eqref{pbtosolve} the value
$Lb={\underline{\mathfrak{Q}}}_{1 1}^{k-1}(x_{0} )$.
The upper bound $Ub$ corresponds to the upper end of a 100(1-$\alpha$)\%-one-sided confidence interval (with for instance $\alpha=0.05$) on the optimal value for 
$N$ policy realizations (using the costs of decisions taken on $N$ independent sampled scenarios).
\end{rem}

\subsection{Other variants}\label{sec:othervariants}

It is easy to adapt several 
recent enhancements of SDDP to
the forward StoDCuP method we have
just presented.
More precisely, we can extend
forward StoDCuP to forward-backward StoDCuP
which builds the trial points
and cuts for the objective and
constraint functions corresponding
to the sampled scenario in the forward pass
and to build cuts for the cost-to-go functions
$\mathcal{Q}_t$ in a backward pass.
In this case, the backward pass also
builds cuts for all functions
$f_t(\cdot,\cdot,\xi_{t j}),$
$g_{t i}(\cdot,\cdot,\xi_{t j})$
$t=1,\ldots,T$, $j=1,\ldots,M_t$,
$i=1,\ldots,p$.
It is also easy to
incorporate in StoDCuP regularization as in
\cite{guilejtekregsddp}, to apply
multicut variants as in \cite{guiguesbandarra18}, \cite{birgemulti},
and cut selection strategies for the bundles of cuts of $\mathcal{Q}_t$,
for instance along the lines of  \cite{dpcuts0}, \cite{guiguesejor17}, \cite{guiguesbandarra18}. 
Observe, however, that all linearizations for 
$f_t(\cdot,\cdot,\xi_{t j})$ and $g_{t i}(\cdot,\cdot,\xi_{t j})$
are tight and therefore no cut selection is needed for these 
linearizations.

\section{Inexact cuts in StoDCuP}\label{variantsstodcup}

\if{ 
\subsection{Forward-backward StoDCuP}

Similarly to DCuP, we can extend forward StoDCuP presented in the previous section 
to forward-backward StoDCuP. In this variant, at iteration $k$, we still compute
in a forward pass trial points $x_n^k$ for nodes $n \in \{n_1,n_2^k,\ldots,n_T^k\}$
with $n_t^k$ a child node of node $n_{t-1}^k$. For each $t=1,\ldots,T$, and 
$j=1,\ldots,M_t$, $i=1,\ldots,p$, a linearization is also computed in the forward pass for 
$f_t(\cdot,\cdot, \xi_{t j})$ and $g_{t i}(\cdot,\cdot, \xi_{t j})$ at these trial points.
However, cuts for $\mathcal{Q}_t$ are computed in a backward pass and in this backward pass
an additional linearization is also built for  $f_t(\cdot,\cdot, \xi_{t j})$ and $g_{t i}(\cdot,\cdot, \xi_{t j})$
using points computed in both the backward and forward pass. 
For the cut computed at iteration $k$ for $\mathcal{Q}_t$, we will still use the notation:
$$
\mathcal{C}_t^k ( x_{t-1} ) = \theta_t^k + \langle \beta_t^k , x_{t-1}  \rangle
$$
with the convention that $\mathcal{C}_{T+1}^k$ is the null function (see below for the computation of 
$\theta_t^k$, $\beta_t^k$). 
We end up iteration $k$ with approximation
$\mathcal{Q}_t^k (  x_{t-1}  ) =\max_{0 \leq j \leq k} \;\mathcal{C}_t^j ( x_{t-1} )$ of $\mathcal{Q}_t$.

Therefore, at iteration $k$, two approximations of
functions $f_t(\cdot,\cdot, \xi_{t j})$ and $g_{t i}(\cdot,\cdot, \xi_{t j})$ are computed
which will be denoted by $f_{t j}^{2k-1}$ and $g_{t i j}^{2k-1}$, respectively,
in the end of the forward pass, and by $f_{t j}^{2k}$ and $g_{t j}^{2k}$, respectively,
in the end of the backward pass.

The detailed steps of forward-backward StoDCuP are described below.\\
\rule{\linewidth}{1pt}
\par {\textbf{Forward-Backward StoDCuP (Stochastic Dynamic Cutting Plane) with linearizations computed in forward
and backward passes.}}\\
\rule{\linewidth}{1pt}
\begin{itemize}
\item[Step 1)] {\textbf{Initialization.}} 
For $t=1,\ldots,T$,  take 
$f_{t j}^0, g_{t i j}^0: \mathcal{X}_t \small{\times} \mathcal{X}_{t-1} \to \mathbb{R}$
affine functions satisfying $f_{t j}^0 \leq f_t(\cdot,\cdot,\xi_{t j})$, 
$g_{t i j}^0 \leq g_{t i}(\cdot,\cdot,\xi_{t j})$,
and for $t=2,\ldots,T$, $\mathcal{Q}_t^0: \mathcal{X}_{t-1} \to \mathbb{R}$ 
is an affine function satisfying $\mathcal{Q}_t^0 \leq \mathcal{Q}_t$.
Set the iteration count $k$ to 1 and $\mathcal{Q}_{T+1}^0 \equiv 0$.
\item[Step 2)] {\textbf{Forward pass.}} \\
Generate a sample $({\tilde \xi}_1^k, {\tilde \xi}_2^k,\ldots, {\tilde \xi}_T^k)$
of $(\xi_1, \xi_2,\ldots, \xi_T)$. \\
{\textbf{For }}$t=1,\ldots,T$,\\
\hspace*{0.7cm}{\textbf{For }}$j=1,\ldots,M_t$,\\
\hspace*{1.4cm}{\textbf{If }}$\xi_{t j}={\tilde \xi}_t^k$ {\textbf{ then }}compute an optimal solution $x_t^k$ of
\begin{equation} \label{defxtkjbis}
{\underline{\mathfrak{Q}}}_{t j}^{2k-2}(x_{t-1}^k )  =  \left\{
\begin{array}{l}
\displaystyle \inf_{x_t} \; f_{t j}^{2k-2}( x_t , x_{t-1}^{k} ) + \mathcal{Q}_{t+1}^{k-1}( x_t ) \\
x_t \in X_{t j}^{2k-2}( x_{t-1}^{k}),
\end{array}
\right.
\end{equation}
\hspace*{1.8cm}where $x_{0}^{k} = x_0$ and where for all $k \geq 1$,
\begin{equation} \label{Xtkbis}
X_{t j}^{2k-2} ( x_{t-1}^k) =\{ x_t \in \mathcal{X}_t: \;g_{t i j}^{2k-2}(x_t, x_{t-1}^k) \leq 0,i=1,\ldots,p,\;\displaystyle A_{t j} x_{t} + B_{t j} x_{t-1}^k = b_{t j} \}.
\end{equation}
\hspace*{1.8cm}Compute $f_t( x_t^{k}, x_{t-1}^{k} , \xi_{t j} )$, 
$g_{t i}( x_t^{k}, x_{t-1}^{k}, \xi_{t j} )$, and subgradients
of $f_t(\cdot,\cdot,\xi_{t j})$,\\
\hspace*{1.8cm} $g_{t i}(\cdot,\cdot,\xi_{t j})$ at $( x_t^{k}, x_{t-1}^{k} )$ 
with corresponding linearizations $\ell_{f_t(\cdot,\cdot,\xi_{t j})}(\cdot, \cdot; ( x_t^{k}, x_{t-1}^{k} ) ) $\\
\hspace*{1.8cm} and $\ell_{g_{t i}(\cdot,\cdot,\xi_{t j})}(\cdot, \cdot; ( x_t^{k}, x_{t-1}^{k} ) )$. Compute
$$
\begin{array}{l}
f_{t j}^{2k-1}(\cdot,\cdot)   \leftarrow \max\Big(f_{t j}^{2k-2}(\cdot,\cdot) , \ell_{f_t(\cdot,\cdot,\xi_{t j})}(\cdot, \cdot; ( x_t^{k}, x_{t-1}^{k} ) )  \Big),\\
g_{t i j}^{2k-1}(\cdot,\cdot)   \leftarrow \max\Big(g_{t i j}^{2k-2}(\cdot,\cdot) , \ell_{g_{t i}(\cdot,\cdot,\xi_{t j})}(\cdot, \cdot; ( x_t^{k}, x_{t-1}^{k} ) )  \Big).\\
\end{array}
$$
\hspace*{1.4cm}{\textbf{Else}}
$$
f_{t j}^{2k-1} = f_{t j}^{2k-2}, \;g_{t i j}^{2k-1} = g_{t i j}^{2k-2}.
$$
\hspace*{1.4cm}{\textbf{End If}}\\
\hspace*{0.7cm}{\textbf{End For}}\\
{\textbf{End For}}
\item[Step 3)] {\textbf{Backward pass.}}\\
Set $\theta_{T+1}^k=0$ and $\beta_{T+1}^k=0$.\\
{\textbf{For }}$t=T,\ldots,2$,\\
\hspace*{0.8cm}{\textbf{For }}$j=1,\ldots,M_t$,\\
\hspace*{1.5cm}Compute an optimal solution $x_{t j}^{B k}$ of \\
\begin{equation}\label{primalpbisddpbis}
{\underline{\mathfrak{Q}}}_{t j}^{2k-1}(x_{t-1}^k ) = \left\{
\begin{array}{l}
\displaystyle \inf_{x_t} \;f_{t j}^{2k-1}( x_t , x_{t-1}^k ) + \mathcal{Q}_{t+1}^k ( x_t )\\ 
x_t \in X_{t j}^{2k-1}(x_{t-1}^k ).\\
\end{array}
\right.
\end{equation}
\hspace*{1.4cm} Compute $f_t( x_{t j}^{B k}, x_{t-1}^{k} , \xi_{t j})$, 
$g_{t i}( x_{t j}^{B k}, x_{t-1}^{k}, \xi_{t j} )$ and subgradients
of $f_t(\cdot,\cdot,\xi_{t j})$ and\\
\hspace*{1.5cm}$g_{t i}(\cdot,\cdot,\xi_{t j})$ at $( x_{t j}^{B k}, x_{t-1}^{k} )$ with corresponding
linearizations \\
\hspace*{1.5cm}$\ell_{f_t(\cdot,\cdot,\xi_{t j})}(\cdot, \cdot; ( x_{t j}^{B k}, x_{t-1}^{k} ) ) $
and $\ell_{g_{t i}(\cdot,\cdot,\xi_{t j})}(\cdot, \cdot; ( x_{t j}^{B k}, x_{t-1}^{k} ) )$.\\
\hspace*{1.5cm}Compute
$$
\begin{array}{l}
f_{t j}^{2k}(\cdot,\cdot)   \leftarrow \max\Big(f_{t j}^{2k-1}(\cdot,\cdot) , \ell_{f_t(\cdot,\cdot,\xi_{t j})}(\cdot, \cdot; ( x_{t j}^{B k}, x_{t-1}^{k} ) )  \Big),\\
g_{t i j}^{2k}(\cdot,\cdot)   \leftarrow \max\Big(g_{t i j}^{2k-1}(\cdot,\cdot) , \ell_{g_{t i}(\cdot,\cdot,\xi_{t j})}(\cdot, \cdot; ( x_{t j}^{B k}, x_{t-1}^{k} ) )  \Big).\\
\end{array}
$$
\hspace*{1.5cm}Compute (for instance using Lemma 2.1 in \cite{guiguessiopt2016}) a subgradient $\beta_t^{k j}$ 
of ${\underline{\mathfrak{Q}}}_{t j}^{2k-1}$ at  $x_{t-1}^k$ and\\
\hspace*{1.5cm}the cut coefficients: 
$$
\begin{array}{lcl}
\theta_{t}^{k}& =& \sum_{j=1}^{M_t}  p_{t j} (
{\underline{\mathfrak{Q}}}_{t j}^{2k-1}(x_{t-1}^k ) - \langle \beta_t^{k j} , x_{t-1}^k \rangle)
\mbox{ and }\beta_t^k = \sum_{j=1}^{M_t} p_{t j} \beta_t^{k j}.
\end{array}
$$
\hspace*{0.8cm}{\textbf{End For}}\\
{\textbf{End For}}\\
Compute an optimal solution $x_{1}^{B k}$ of \\
\begin{equation}\label{primalpbisddpbis}
\left\{
\begin{array}{l}
\displaystyle \inf_{x_1} \;f_{1 1}^{2k-1}( x_1 , x_{0}) + \mathcal{Q}_{2}^k ( x_1 )\\ 
x_1 \in X_{1 1}^{2k-1}(x_{0}).\\
\end{array}
\right.
\end{equation}
Compute $f_1( x_{1}^{B k}, x_{0} , \xi_{1})$, 
$g_{1 i}( x_{1}^{B k}, x_{0}, \xi_{1} )$, and subgradients
of $f_1(\cdot,\cdot,\xi_1)$, $g_{1 i}(\cdot,\cdot,\xi_1)$ at $( x_{1}^{B k}, x_{0} )$ with \\
corresponding
linearizations $\ell_{f_1(\cdot,\cdot,\xi_{1})}(\cdot, \cdot; ( x_{1}^{B k}, x_{0} ) ) $
and $\ell_{g_{1 i}(\cdot,\cdot,\xi_{1})}(\cdot, \cdot; ( x_{1}^{B k}, x_{0} ) )$.\\
Compute
$$
\begin{array}{l}
f_{1 1}^{2k}(\cdot,\cdot)   \leftarrow \max\Big(f_{1 1}^{2k-1}(\cdot,\cdot) , \ell_{f_1(\cdot,\cdot,\xi_{1})}(\cdot, \cdot; ( x_{1}^{B k}, x_{0} ) )  \Big),\\
g_{1 i 1}^{2k}(\cdot,\cdot)   \leftarrow \max\Big(g_{1 i 1}^{2k-1}(\cdot,\cdot) , \ell_{g_{1 i}(\cdot,\cdot,\xi_{1})}(\cdot, \cdot; ( x_{1}^{B k}, x_{0} ) )  \Big).\\
\end{array}
$$
\item[Step 4)] Do $k \leftarrow k+1$ and go to Step 2).\\
\end{itemize}

\subsection{Inexact cuts in StoDCuP}

}\fi

In this section, we present an extension of StoDCuP to solve problem \eqref{pbtosolve}.
Since all subproblems of forward StoDCuP presented in Section \ref{stodcup} are linear programs, it is
easy to derive an inexact variant of StoDCuP that computes $\varepsilon_t^k$-optimal solutions
(instead of optimal solutions in StoDCuP) of the subproblems solved for iteration $k$ and stage $t$. We show in Lemma \ref{lemmacutsbis} below that the cuts computed by this variant are still valid and that the distance between 
the cuts and ${\underline{\mathcal{Q}}}_t^{k-1}(\cdot ) = \sum_{j=1}^{M_t} p_{t j} {\underline{\mathfrak{Q}}}_{t j}^{k-1}(\cdot)$ at the trial point $x_n^k$ for stage $t$ and iteration $k$
is at most $\varepsilon_t^k$. This variant of StoDCuP, called inexact StoDCuP, is given below
and the convergence of the method is proved in Theorem \ref{convproofStoDCuPinex}:\\
\rule{\linewidth}{1pt}
\par {\bf Inexact StoDCuP.}\\
\rule{\linewidth}{1pt}
\begin{itemize}
\item[Step 1)] {\textbf{Initialization.}} For $t=1,\ldots,T$,  take 
$f_{t j}^0, g_{t i j}^0: \mathcal{X}_t \small{\times} \mathcal{X}_{t-1} \to \mathbb{R}$
affine functions satisfying $f_{t j}^0 \leq f_t(\cdot,\cdot,\xi_{t j})$, 
$g_{t i j}^0 \leq g_{t i}(\cdot,\cdot,\xi_{t j})$,
and for $t=2,\ldots,T$, $\mathcal{Q}_t^0: \mathcal{X}_{t-1} \to \mathbb{R}$ 
is an affine function satisfying $\mathcal{Q}_t^0 \leq \mathcal{Q}_t$.
Set $x_{n_0} = x_0$, set the iteration count $k$ to 1, and $\mathcal{Q}_{T+1}^0 \equiv 0$. 
\item[Step 2)] Generate a sample $({\tilde \xi}_1^k, {\tilde \xi}_2^k,\ldots, {\tilde \xi}_T^k)$
of $(\xi_1, \xi_2,\ldots, \xi_T)$ corresponding to a set of nodes  $(n_1^k, n_2^k, \ldots, n_T^k)$ 
where $n_1^k=n_1$, and for $t \geq 2$, $n_t^k$ is a node of stage $t$, child of node $n_{t-1}^k$. Set $n_0^k=n_0$.\\
Do $\theta_{T+1}^k=0$ and $\beta_{T+1}^k=0$.\\
{\textbf{For }}$t=1,\ldots,T$,\\
\hspace*{0.7cm}Let $n=n_{t-1}^k$.\\
\hspace*{0.7cm}{\textbf{For }}every $m \in C(n)$,
\begin{itemize}
\item[]\hspace*{0.7cm}a) compute an $\varepsilon_t^k$-optimal feasible solution $x_m^k$ of
\begin{equation} \label{defxtkj}
{\underline{\mathfrak{Q}}}_{t j_t(m)}^{k-1}(x_{n}^k  )  =  \left\{
\begin{array}{l}
\displaystyle \min_{x_m} \; f_{t j_t(m)}^{k-1}( x_m , x_{n}^{k} ) + \mathcal{Q}_{t+1}^{k-1}( x_m ) \\
x_m \in X_{t j_t(m)}^{k-1}( x_{n}^k  ).
\end{array}
\right.
\end{equation}
\item[] \item[]\hspace*{0.7cm}b) Compute
function values and  
subgradients of convex functions $f_t(\cdot,\cdot,\xi_m)$ and $g_{t i}(\cdot,\cdot,\xi_m)$
\item[]\hspace*{0.8cm}at $(x_m^k , x_n^k )$ and let 
$\ell_{f_{t}(\cdot,\cdot,\xi_m)}((\cdot,\cdot); (x_m^{k}, x_{n}^{k} ) ) $ and 
$\ell_{g_{t i}(\cdot,\cdot,\xi_m)}((\cdot,\cdot); (x_m^{k}, x_{n}^{k} ))$
denote the 
\item[]\hspace*{0.7cm}corresponding linearizations. 
\item[]\hspace*{0.7cm}c) Set
$$
\begin{array}{lcl}
f_{t j_t(m)}^{k} &=& \max\Big( f_{t j_t(m)}^{k-1} \,,\,  \ell_{f_{t}(\cdot,\cdot,\xi_m)}((\cdot,\cdot); (x_m^{k}, x_{n}^{k} ) )   \Big),\\
g_{ti}^{k} &= &\max\Big( g_{t i}^{k-1} \,,\,
\ell_{g_{t i}(\cdot,\cdot,\xi_m)}((\cdot,\cdot); (x_m^{k}, x_{n}^{k} )) \Big), \quad \forall i=1,\ldots,p.
\end{array}
$$
\item[]\hspace*{0.7cm}d) 
Using the notation of Section
\ref{stodcupdeta}, in particular 
\eqref{mat1s}, \eqref{mat2s}, and \eqref{mat2s}, if $t\geq 2$ compute
\item[]\hspace*{1.2cm}an 
$\varepsilon_t^k$-optimal feasible solution $(\alpha_m^k, \mu_m^k, \delta_m^k, \nu_m^k, \lambda_m^k)$ of the dual problem
{\small{
$$
\begin{array}{l}
\displaystyle \max_{\alpha, \mu, \delta, \nu, \lambda} \;\alpha^{\top} (  B_{t j_t(m)}^{k-1} x_{n}^k  +  C_{t j_t(m)}^{k-1} ) + 
\mu^{\top} ( E_{t j_t(m)}^{k-1}  x_{n}^k  +   H_{t j_t(m)}^{k-1}     ) + \delta^{\top} \theta_{t+1}^{0:k-1}\\
\hspace*{1.1cm} + \lambda^{\top} ( b_{t j_t(m)}  - B_{t j_t(m)} x_{n}^k ) + \nu^{\top} {\bar x}_t \\
(A_{t j_t(m)}^{k-1})^{\top} \alpha + (D_{t j_t(m)}^{k-1} )^{\top} \mu + (\beta_{t+1}^{0:k-1})^{\top} \delta - \mathbb{X}_t^{\top}  \nu  - (A_{t j_t(m)})^{\top} \lambda = 0,\\
{\textbf{e}}^{\top} \alpha = 1,\,{\textbf{e}}^{\top} \delta = 1, \alpha, \mu, \delta, \nu \geq 0.
\end{array}
$$
}}
\end{itemize}
\hspace*{0.7cm}{\textbf{End For}}\\
\hspace*{0.7cm}{\textbf{If }}$t \geq 2$ compute: 
\begin{equation}\label{formulabetatheta}
\begin{array}{lcl}
\beta_{t}^{k} & =& \displaystyle \sum_{m \in C(n)} p_m \Big[  (B_{t j_t(m)}^k )^{\top} \alpha_m^{k-1}  + (E_{t j_t(m)}^{k-1} )^{\top} \mu_m^k - B_{t j_t(m)}^{\top} \lambda_m^k \Big],\\
\theta_t^k &= &\displaystyle \sum_{m \in C(n)}  p_{m} \Big[  \langle \alpha_m^k ,  C_{t j_t(m)}^{k-1}  \rangle  + \langle  \mu_m^k , H_{t j_t(m)}^{k-1} \rangle  +  \langle  \delta_m^k ,  \theta_{t+1}^{0:k-1} \rangle  +  \langle \lambda_m^k ,  b_{t j_t(m)} \rangle  + \langle \nu_m^k  ,  {\bar x}_t \rangle  \Big].
\end{array}
\end{equation}
\hspace*{0.7cm}{\textbf{End If}}\\
{\textbf{End For}}
\item[Step 4)] Do $k \leftarrow k+1$ and go to Step 2).
\end{itemize}
\rule{\linewidth}{1pt}
Clearly Lemma \ref{lemmaft} still holds for Inexact StoDCuP. The quality of the cuts computed for $\mathcal{Q}_t$ by Inexact StoDCuP is given in Lemma \ref{lemmacutsbis}:
\begin{lemma}[Validity and quality of cuts computed by Inexact StoDCuP] \label{lemmacutsbis}
Let Assumptions (H0) and (H1)-Sto hold. For every $t=2,\ldots,T+1$, for every $k \geq 1$,
we have 
\begin{equation}\label{validcutsQb}
\mathcal{Q}_t( x_{t-1} ) \geq \mathcal{C}_t^k ( x_{t-1} ) \mbox{ and }\mathcal{Q}_t( x_{t-1} ) \geq \mathcal{Q}_t^k ( x_{t-1} ),\;\forall  x_{t-1} \in \mathcal{X}_{t-1}.
\end{equation}
For all $t=1,\ldots,T$, $j=1,\ldots,M_t$, for every $k \geq 1$, we have
\begin{equation}\label{approxvaluefuncb}
{\underline{\mathfrak{Q}}}_{t j}^{k}(x_{t-1}  ) \leq \mathfrak{Q}_t ( x_{t-1} , \xi_{t j} ) \mbox{ for all }x_{t-1} \in \mathcal{X}_{t-1}.
\end{equation}
For all $t=2,\ldots,T$, for every $k \geq 1$, 
defining
$
{\underline{\mathcal{Q}}}_t^{k-1}(x_{n}^k ) = \sum_{j=1}^{M_t} p_{t j} {\underline{\mathfrak{Q}}}_{t j}^{k-1}(x_{n}^k )$,  
we have for every $n \in {\tt{Nodes}}(t-1)$ and for 
all $k \in \mathcal{S}_n$:
\begin{equation}\label{erroroncutQbis}
0 \leq {\underline{\mathcal{Q}}}_t^{k-1}(x_{n}^k ) - \mathcal{C}_t^k ( x_n^k ) \leq \varepsilon_t^k.
\end{equation}
\end{lemma}
\begin{proof} The proofs of  \eqref{validcutsQ} and \eqref{approxvaluefunc}
in Lemma \ref{lemmacuts}
can be used to prove \eqref{validcutsQb} and \eqref{approxvaluefuncb} for Inexact StoDCuP, observing that only feasibility and not optimality of the
primal and dual solutions computed as well as Lemma \ref{lemmaft} (which, as we have already observed, holds) are needed in these proofs.

Now take $n \in {\tt{Nodes}}(t-1)$ and $k \in \mathcal{S}_n$.
Then recalling that 
{\small{
$$
\mathcal{D}_{t j}^{k-1}(\alpha, \mu, \delta, \nu, \lambda;x_{t-1})=
\alpha^{\top} (  B_{t j}^{k-1} x_{t-1}  +  C_{t j}^{k-1} ) + \mu^{\top} ( E_{t j}^{k-1}  x_{t-1}  +   H_{t j}^{k-1}     ) + \delta^{\top} \theta_{t+1}^{0:k-1} + \lambda^{\top} ( b_{t j}  - B_{t j} x_{t-1} ) + \nu^{\top} {\bar x}_t,
$$
}}
by definition of $(\alpha_m^k, \mu_m^k, \delta_m^k, \nu_m^k, \lambda_m^k)$ and of $\mathcal{C}_t^k$, we get 
\begin{equation}\label{eq1finb}
{\underline{\mathfrak{Q}}}_{t j_t(m)}^{k-1}(x_{n}^k    ) -  \varepsilon_t^k  \leq \mathcal{D}_{t j_t(m)}^{k-1}(\alpha_m^k , \mu_m^k , \delta_m^k , \nu_m^k , \lambda_m^k ;x_{n}^k) \leq 
{\underline{\mathfrak{Q}}}_{t j_t(m)}^{k-1}(x_{n}^k    )
\end{equation}
and
\begin{equation}\label{eq2finb}
\mathcal{C}_t^k ( x_n^k ) = \displaystyle \sum_{m \in C(n)} p_m  \mathcal{D}_{t j_t(m)}^{k-1}(\alpha_m^k , \mu_m^k , \delta_m^k , \nu_m^k , \lambda_m^k ;x_{n}^k).
\end{equation}
Since ${\underline{\mathcal{Q}}}_t^{k-1}(x_{n}^k ) = \sum_{m \in C( n_{t-1}^k )} p_{m} {\underline{\mathfrak{Q}}}_{t j_t(m)}^{k-1}(x_{n}^k ) $, 
$p_m \geq 0,$ and $\sum_{m \in C(n)} p_m=1$, relations \eqref{eq1finb} and \eqref{eq2finb} imply \eqref{erroroncutQbis}. 
\hfill 
\end{proof}
Lemma \ref{lipcontQtstoinex} below  is  the analogue of Lemma \ref{lipcontQtstob}:
\begin{lemma}\label{lipcontQtstoinex}
Let Assumptions (H0) and (H1)-Sto hold and assume that sequences $\varepsilon_t^k$ are bounded: 
$|\varepsilon_t^k| \leq {\hat \varepsilon}$ for all $t,k$, for some $0 \leq {\hat \varepsilon}<+\infty$.
Then, the following statements hold for Inexact StoDCuP:
\begin{itemize}
\item[(a)] For $t=2,\ldots,T$, the sequences $\{\theta_t^k\}_{k=1}^\infty$ and 
$\{\beta_t^k\}_{k=1}^\infty$ are almost surely bounded.
\item[(b)] There exists   $L \ge 0$ such that
for each $t=2,\ldots,T$, $\mathcal{Q}^k_{t}$
 is $L$-Lipschitz continuous on $\mathcal{X}_{t-1}$ for every $k \ge 1$.
 \item[(c)] There exists $\hat L \ge 0$ such that
 for each $t=1,\ldots,T$, $j=1,\ldots,M_t$, functions $f_{t j}^k$ and $g_{t i j}^k$ are 
 $\hat L$-Lipschitz continuous on $\mathcal{X}_t \times \mathcal{X}_{t-1}$ for every $k \ge 1$ and $i=1,\ldots,p$.
 \end{itemize}
\end{lemma}
\begin{proof}
(a) Using (H1)-Sto, there is $\varepsilon>0$ such that for every $t \in \{2,\ldots,T\}$,
every $x_{t-1} \in \mathcal{X}_{t-1} + {\bar B}(0;\varepsilon)$,
and every $j=1,\ldots,M_t$, the set $X_{t j}^{0}(x_{t-1})$ is nonempty and $f_{t j}^0(\cdot,x_{t-1})+\mathcal{Q}_{t+1}^0(\cdot)$
is continuous on this set. Therefore ${\underline{\mathfrak{Q}}}_{t j}^{0}$ is convex and finite 
on $\mathcal{X}_{t-1} + {\bar B}(0;\varepsilon)$, implying that
${\underline{\mathfrak{Q}}}_{t j}^{0}$ is Lipschitz continuous on
$\mathcal{X}_{t-1}$. It follows that ${\underline{\mathcal{Q}}}_t^{0}$ is also Lipschitz 
continuous on $\mathcal{X}_{t-1}$ and we can define 
$\displaystyle \min_{x_{t-1} \in \mathcal{X}_{t-1} } {\underline{\mathcal{Q}}}_t^{0}( x_{t-1} ) \in \mathbb{R}$.
Similarly to DCuP, due to (H1)-Sto, we can also choose $\varepsilon>0$ in such
a way that
$\mathcal{Q}_t$ is Lipschitz continuous on $\mathcal{X}_{t-1} + {\bar B}(0;\varepsilon)$,
implying that we can define 
$\max_{x_{t-1} \in \mathcal{X}_{t-1}+ \bar B(0;{{\varepsilon}})}\mathcal{Q}_t(x_{t-1} )<+\infty$. 
We can now easily extend the proof of Lemma \ref{lipcontQtstob}: for every  
$x_{t-1} \in \mathcal{X}_{t-1} + \bar B(0;\varepsilon)$, denoting $n=n_{t-1}^k$,
we have for $k \geq 2$:
$$
\begin{array}{rcl}
\displaystyle \max_{x_{t-1} \in \mathcal{X}_{t-1}+ \bar B(0;{{\varepsilon}})}\mathcal{Q}_t(x_{t-1} )     
\geq \mathcal{Q}_t( x_{t-1}) 
 &  \stackrel{\eqref{validcutsQb}}{\geq} & \mathcal{C}_t^{k}( x_{t-1})   \\
&  =  & \mathcal{C}_t^{k}( x_{n}^k ) + \langle \beta_t^k  , x_{t-1} - x_{n}^k \rangle \;\;\;\;[\mathcal{C}_t^k\mbox{ is affine],}\\
&  \stackrel{\eqref{erroroncutQbis}}{\geq}  &  {\underline{\mathcal{Q}}}_t^{k-1}(x_n^k) -  \varepsilon_t^k  + 
\langle \beta_t^k  , x_{t-1} - x_{n}^k \rangle,\\
& \geq & \displaystyle \min_{x_{t-1} \in \mathcal{X}_{t-1} } {\underline{\mathcal{Q}}}_t^{0}( x_{t-1} ) - 
{\hat \varepsilon}  + \langle \beta_t^k  , x_{t-1} - x_{n}^k \rangle.
\end{array}
$$
For $\beta_t^k \neq 0$, take $x_{t-1}=x_{n}^k + \frac{{{\varepsilon}}}{2} \frac{\beta_t^k}{\|\beta_t^k\|}$ to obtain 
$$
\|\beta_t^k\| \leq L:= \frac{2}{{{\varepsilon}}}\left( {\hat{\varepsilon}} +   \displaystyle \max_{x_{t-1} \in \mathcal{X}_{t-1}+ \bar B(0;{{\varepsilon}})} \mathcal{Q}_t(x_{t-1} )
- \displaystyle \min_{x_{t-1} \in \mathcal{X}_{t-1}} {\underline{\mathcal{Q}}}_t^{0}( x_{t-1} )\right).
$$
Using \eqref{erroroncutQbis}, we also have  for $n=n_{t-1}^k$:
$$
-{\hat{\varepsilon}} + \min_{x_{t-1} \in \mathcal{X}_{t-1}}
{\underline{\mathcal{Q}}}_t^0(x_{t-1} ) \leq \theta_t^k = \mathcal{C}_t^k( x_n^k  ) \leq 
\max_{x_{t-1} \in \mathcal{X}_{t-1}} \mathcal{Q}_t( x_{t-1}).
$$
\par (b) immediately follows from (a) and (c) from (H1)-Sto.$\hfill$
\end{proof}

\begin{thm}[Convergence of Inexact StoDCuP]\label{convproofStoDCuPinex} 
Let Assumptions (H0), (H1)-Sto, and (H2) hold and assume that 
$\lim_{k \rightarrow +\infty} \varepsilon_t^k = 0$ for $t=1,\ldots,T$.
Then the conclusions of Theorem \ref{convproofStoDCuP} hold: 
for every $t=1,\ldots,T$,$i=1,\ldots,p$, almost surely
\eqref{limgtas} and \eqref{htstodcup} hold and the limit of the sequence of first stage problems optimal values 
$( f_{1 1}^{k-1}(x_{n_1}^{k} , x_{0} ) + \mathcal{Q}_2^{k-1}( x_{n_1}^k )  )_{k \geq 1}$ is the optimal value
$\mathcal{Q}_1( x_0 ) $ of \eqref{firststodp} and 
any accumulation point of the sequence $(x_{n_1}^{k})$ is an optimal solution to the first stage problem \eqref{firststodp}.
\end{thm}
\begin{proof}
The proof is an adaptation of the proof of Theorem \ref{convproofStoDCuP} and uses
Lemmas \ref{lemmaft}, \ref{lemmacutsbis}, and \ref{lipcontQtstoinex}. We highlight these
adaptations below.

Using Lemma \ref{lemmacutsbis}, for Inexact StoDCuP relation \eqref{inductfirst} becomes
\begin{equation}\label{inductfirstb}
\begin{array}{lcl}
0 \leq \mathcal{Q}_t ( x_n^{k(\ell)} ) - \mathcal{Q}_t^{k(\ell)} ( x_n^{k(\ell)} ) & \leq &  \mathcal{Q}_t ( x_n^{k(\ell)} ) - \mathcal{C}_t^{k(\ell)} ( x_n^{k(\ell)} ) \\
& \stackrel{\eqref{erroroncutQ}}{\leq}& \varepsilon_t^{k(\ell)} + \mathcal{Q}_t ( x_n^{k(\ell)} ) - {\underline{\mathcal{Q}}}_t^{k(\ell)-1}(x_{n}^{k(\ell)} ),\\
& = & \varepsilon_t^{k(\ell)} + \displaystyle  \sum_{m \in C( n_{t-1}^{k(\ell)}  ) }  p_{m} \Big[ \mathfrak{Q}_t( x_{n}^{k(\ell)} , \xi_m    )    -  {\underline{\mathfrak{Q}}}_{t j_t(m)}^{k(\ell)-1}(x_{n}^{k(\ell)} )\Big].
\end{array}
\end{equation}
Also, by definiton of $x_m^k$, we now have
\begin{equation}\label{defepssolxmkb}
{\underline{\mathfrak{Q}}}_{t j_t(m)}^{k(\ell)-1}(x_{n}^{k(\ell)} )
\leq f_{t j_t(m)}^{k(\ell)-1}(x_m^{k(\ell)}, x_n^{k(\ell)} ) +\mathcal{Q}_{t+1}^{k(\ell)-1}( x_m^{k(\ell)}  ) \leq {\underline{\mathfrak{Q}}}_{t j_t(m)}^{k(\ell)-1}(x_{n}^{k(\ell)} ) + \varepsilon_t^{k(\ell)},
\end{equation}
which, plugged into \eqref{inductfirstb} gives
\begin{equation}\label{secondb}
0 \leq \mathcal{Q}_t ( x_n^{k(\ell)} ) - \mathcal{Q}_t^{k(\ell)} ( x_n^{k(\ell)} )  
\leq 2 \varepsilon_t^{k(\ell)} + \displaystyle  \sum_{m \in C( n_{t-1}^{k(\ell)}  ) }  p_{m} \Big[ 
\mathfrak{Q}_t( x_{n}^{k(\ell)} , \xi_m    )    -  f_{t j_t(m)}^{k(\ell)-1}(x_{m}^{k(\ell)}, x_n^{k(\ell)} )  -  \mathcal{Q}_{t+1}^{k(\ell)-1}( x_m^{k(\ell)}   )\Big].
\end{equation} 
The remaining relations and arguments used in the convergence proof of StoDCuP apply to prove the theorem.
\hfill
\end{proof}

\if{
\vspace*{0.4cm}
{\color{red{
\begin{remark} The extensions
of StoDCuP given in
Section \ref{sec:othervariants}
also apply to Inexact StoDCuP.
\end{remark}
}
}\fi

\section{Numerical experiments}\label{sec:numexp}

We consider the multistage nondifferentiable nonlinear  stochastic 
program given by the following DP equations:
the Bellman function for stage 
$t=1,\ldots,T$, is
$\mathcal{Q}_t(x_{t-1})=\mathbb{E}_{\xi_t,\Psi_t,U_t}
[\mathfrak{Q}_t(x_{t-1},\xi_t,\Psi_t,U_t)]$ and for
$t=1,\ldots,T$, $\mathfrak{Q}_t(x_{t-1},\xi_t,\Psi_t,U_t)$ is given by
\begin{equation}\label{pbsim}
\begin{array}{l}
\min \;f_t(x_t,x_{t-1},\xi_t,U_t) + \mathcal{Q}_{t+1}(x_t)\\
-100\,{\textbf{e}} \leq x_t \leq 100\,{\textbf{e}},\\
\max(  4(x_t - {\textbf{e}})^T (x_t-{\textbf{e}}), 
x_t^T \xi_t \xi_t^T x_t + x_t^T \xi_t + 1 ) \leq \Psi_t,
\end{array}
\end{equation}
where $x_t \in \mathbb{R}^n$,
$
f_t(x_t,x_{t-1},\xi_t,U_t)=
\max( (x_t-x_{t-1})^T \xi_t \xi_t^T (x_t-x_{t-1})
+ x_t^T \xi_t + 1 ,
x_t^T \xi_t \xi_t^T x_t + x_t^T {\textbf{e}} + U_t ),
$ {\textbf{e}} is a vector of size $n$ of
ones, and $\mathcal{Q}_{T+1}$ is the null function. In these equations, $\xi_t$ is a discretization of a Gaussian random vector
with mean vector $m_t$ having entries $1$ or $-1$ and covariance matrix $\Sigma_t=A_t A_t^T + 0.5 I$
where $A_t$ has entries in $[-0.5,0.5]$; $U_t$ is a discrete random variable taking values $+10$, $-10$,
and $\Psi_t$ has discrete distribution with support contained
in $[10^4,10^5]$. The number of realizations $M_t$ for $(\xi_t,\Psi_t,U_t)$
is fixed to $M_t=M$ for each stage. We assume
that 
$(\xi_1,\Psi_1,U_1)$ is known
and $(\xi_2,\Psi_2,U_2),\ldots,(\xi_T,\Psi_T,U_T)$ are independent.

We generate 6 instances of this problem with parameters
$T,n,M$ given by 
$(T,n,M)=(3,10,2),$ $(3,10,10)$, $(5,10,10)$, $(5,10,20)$, $(10,200,10)$, and $(10,200,20)$.
The instances are chosen taking realizations $\Psi_{t j}$
of $\Psi_{t}$ sufficiently large, in such a way that  
Assumption (H1)-Sto-4) holds.\footnote{We checked
that the instances generate nontrivial
nondifferentiable problems in the sense that
no function in the max dominates the other on the
set $\mathcal{X}_t:=\{x_t \in \mathbb{R}^n: -100\,{\textbf{e}} \leq x_t \leq 100\,{\textbf{e}}\}$.} It is easy to check that the remaining assumptions
(H1)-Sto and (H0) are satisfied and therefore 
StoDCuP and Inexact StoDCuP (IStoDCuP) can both be applied
to solve the problem. 
Since the problem is nondifferentiable,
SDDP and Inexact SDDP from \cite{guiguesinexact2018}
cannot be applied directly. However, it is
possible to reformulate the problem 
as a differentiable problem replacing in
\eqref{pbsim} each max with 2 quadratic constraints.
The number of variables and of linear and quadratic constraints
of the deterministic equivalent corresponding to this reformulation
is given in Table \ref{tablevar} for all instances.

\begin{table}
\centering
{\small{
\begin{tabular}{|c|c|c|c|}
\hline
{\tt{Instance}}   & Variables  & Linear constraints    & Quadratic constraints \\
\hline
$T,n,M$& $(n+2)(1+M^{T-1})$&$(2n+1)(1+M^{T-1})$&$4(1+M^{T-1})$\\
\hline
3,10,2&60&105&20\\
\hline
3,10,10&1212&2121&404\\
\hline
5,10,10&120 012&210 021&40 004\\
\hline
5,10,20&1.92e6&3.36e6&6.4e5\\
\hline
10,200,10&2.02e11&4.01e11&4e9\\
\hline
10,200,20&1.0342e14&2.0531e14&2.0480e12\\
\hline
\end{tabular}
}}
\caption{Number of variables and constraints
of the deterministic equivalents of the 6 instances.}\label{tablevar}
\end{table}

Using this reformulation, we implemented 
ISDDP given in \cite{guiguesinexact2018} and
SDDP, using Mosek \cite{mosek} to solve the subproblems.
Unfortunately,  none of the 6 instances could be solved by these
implementations because essentially all suproblems
to be solved within SDDP and ISDDP 
cannot be solved by Mosek due to the fact that all
the matrices of the quadratic forms are ill-conditioned,
yielding an error in the convexity check performed by
Mosek (even if of course in theory all subproblems
are convex) which is
done using Cholesky factorizations of those matrices.
Rather than a flaw of Mosek which is an efficient solver
for conic problems, the problem comes from
the subproblems under
consideration which are difficult to solve because of
the degeneracy of the quadratic forms.\footnote{We also implemented ISDDP using the inexact
cuts from Section 2 of \cite{guisvmont20} and such variant could not solve
our instances neither, again because Mosek failed
to solve all quadratic subproblems of the corresponding ISDDP.}
In this condition, StoDCup and IStoDCuP (considering
the variants which linearize all nonlinear functions
at all iterations for all subproblems) which only have to solve
linear subproblems are possible solution methods to
solve the original problem. The corresponding Matlab implementation
can be found at {\tt{https://github.com/vguigues/StoDCuP}}.\footnote{The tests were run in file {\tt{TestStoDCuP.m}} and
the functions implementing StoDCup and IStoDCuP
are {\tt{inexact\_stodcup\_quadratic.m}} and 
{\tt{inexact\_stodcup\_quadratic\_cut\_selection.m}}, this latter being 
a variant with cut selection, denoted IStoDCuP CS in this section.} Both StoDCuP and IStoDCuP were 
warm-started constructing 20 linearizations
of each function $f_{t}(\cdot,\cdot,\xi_{t j})$
and $g_{t i}(\cdot,\cdot,\xi_{t j})$
at points randomly selected in the set
$\mathcal{X}_t:=\{x_t \in \mathbb{R}^n: -100\,{\textbf{e}} \leq x_t \leq 100\,{\textbf{e}}\}$.

For IStoDCuP to be well defined, we also need
to set the level of accuracy of the computed
solutions along the iterations of the method.
It makes sense to increase the accuracy (or equivalently to decrease the relative error)  of the
solutions  as the algorithm progresses and eventually
for a given iteration 
to increase the accuracy with the stage.
In our experiments the
relative error of the subproblem solutions (Mosek
parameter MSK\_DPAR\_INTPNT\_TOL\_REL\_GAP
 whose range is any value $\geq 10^{-14}$ and default value is
$10^{-8}$) is given in Table \ref{tableacc}; see also
Remark 2 in \cite{guiguesinexact2018} for other choices of sequences
of noises $\varepsilon_t^k$. For StoDCuP, this parameter
was set to $10^{-10}$ for all iterations.

\begin{table}
\centering
{\small{
\begin{tabular}{|c|c|c|c|c|c|c|c|}
\hline
 Iteration   & 1--10  & 11--20    & 21--40   & 41--140
&  141--240  & 241--350 & $>350$ \\
\hline
MSK\_DPAR\_INTPNT\_TOL\_REL\_GAP   & 10  & 5    & 3   &1 
&   0.5 & 0.1& e-6\\
\hline
 \end{tabular}
}}
\caption{Relative error of the subproblem solutions in
IStoDCuP along iterations (Mosek parameter MSK\_DPAR\_INTPNT\_TOL\_REL\_GAP).}\label{tableacc}
\end{table}

Same as SDDP, methods StoDCuP and IStoDCuP 
compute at each iteration
a lower bound on the optimal value which
is the optimal value of the first stage
problem solved in the forward pass
and upper bounds computed as SDDP by Monte-Carlo
simulations, from iterations 200 on, using 
the last 200 forward scenarios. We
also run the methods with the smoothed upper
bounds used in
\cite{pythonlibrarysddp,shapguichen} which consists in using all
previous forward passes to compute the upper
bound but this implementation needed many more iterations to satisfy the stopping criterion
for the large instances and the corresponding results will not be
reported. We should also recall (see \cite{guiguesinexact2018})
that for both  IStoDCuP  and StoDCuP the first stage problems are solved
with high accuracy to get valid lower bounds
from the optimal values of the first stage forward
subproblems. 
The algorithms stopped when a relative
gap of at most 0.1 was achieved for the first
four instances while for the last two instances, 
the algorithms 
were run for 900 and
600 iterations, respectively.

As mentioned in Section 
\ref{sec:othervariants}, the cut
selection methods proposed
in \cite{guiguesejor17,guiguesbandarra18,dpcuts0}
for SDDP
can be directly applied to StoDCuP.
The convergence of DDP, single cut SDDP, and multicut SDDP
combined with these cut selection
methods 
was proved in 
\cite{guiguesejor17, guiguesbandarra18}.
For the three largest instances,
we tested another cut selection strategy
for the inexact variant IStoDCuP
of StoDCuP, denoted by IStoDCuP CS, which consists, in the backward passes, from a given iteration 
$I$ and for the next $L-1$
iterations, to simultaneously add
a new cut  
(computed at the trial points
computed in the forward pass)
for each cost-to-go function
and to eliminate the oldest cut.
As long as $L$ is not too large,
we only eliminate, progressively, the cuts computed
with loose accuracy (the cuts
computed for the first $L$ iterations).
Therefore, with this method, in the end of iterations
$I,I+1,\ldots,I+L-1$, the number
of cuts for each cost-to-go function
is constant, equal to $I$,
and then from iteration $I+L$ on,
one cut is added for each cost-to-go function
at each iteration as in IStoDCuP if we choose
one sampled scenario per forward pass.
In our experiments, this cut selection strategy was run
taking $I=L=350$.

The evolution of the upper and lower bounds
along the iterations of StoDCuP, IStoDCup, and IStoDCuP CS to solve the 6 instances is given in Figure
\ref{figureb} while the cumulated CPU time
is given in Figure \ref{figurt}.  All methods were implemented in Matlab and run on an Intel Core i7, 1.8GHz, processor
with 12,0 Go of RAM. More precisely, the number
of iterations and CPU time required
to solve all instances is given in Table \ref{tablesumm}
and the bounds and cumulated CPU time
for some iterations are given in
Table \ref{tablebounds}.

We observe that the sequences of upper bounds
tend to  decrease, the sequences of lower bounds are
increasing, and all these sequences converge to the
same values for a given instance; which 
illustrates the validity of StoDCuP and IStoDCuP
to solve a multistage stochastic nondifferentiable
convex problem and 
is a good
indication that both methods have been well implemented.

In all instances, at least one of the inexact variants of StoDCuP
was quicker than StoDCuP and provided policies
of similar quality. 
A general behavior we expect for IStoDCuP is to
have quicker iterations but to need more iterations, as for instance $T=5$, $n=10$, $M=10$, or a similar
number of iterations, as for instances $T=3$, $n=10$, $M=2$ and 
$T=5, n=10, M=20$ (in this latter
the number of iterations
before gettting a gap smaller than 0.1
is 1376, 1387, and 1642 for
respectively IStoDCuP CS, StoDCuP,
and IStoDCuP (see Table \ref{tablebounds})).
However, it may happen that StoDCuP requires more iterations
as for instance $T=3$, $n=10$, $M=10$. The inexact variant with cut selection tested
on the three largest instances allowed us to decrease 
the gap with respect to IStoDCuP while still being quicker
than StoDCuP. It is also interesting to see that on the largest
instance this inexact variant also yielded a much smaller gap
than StoDCuP after completing the 600 iterations (see Table \ref{tablebounds} and Figure \ref{figureb}).

\begin{table}
\centering
{\small{
\begin{tabular}{|c|c|c|c|c|}
\hline
$(T,n,M)$   & Iterations StoDCuP & 
\begin{tabular}{c}Iterations Inexact\\StoDCuP
\end{tabular}  & CPU time StoDCuP    & \begin{tabular}{c}CPU Time Inexact\\StoDCuP\end{tabular}\\
\hline
$3,10,2$     &   216    &   216       &    10.13   &  7.63     \\
\hline
$3,10,10$     &    586   &   451      &    597.2   &  148.4     \\
\hline
$5,10,10$     &  1061      &    1221     &   4345    &2825       \\
\hline
$5,10,20$     &  1387  &   1376     &   4493    &       3784\\
\hline
$10,200,10$     &   900     &  900       &    62 536  &   55 061    \\
\hline
$10,200,20$     &    600     &  600      &    26 414   &  25 276  \\
\hline
 \end{tabular}
}}
\caption{Number of iterations and CPU time (in seconds)
for each instance and method. For Inexact
StoDCuP, we report the quickest,
among IStoDCuP and IStoDCuP CS.}\label{tablesumm}
\end{table}

\begin{figure}
\centering
\begin{tabular}{cc}
\includegraphics[scale=0.6]{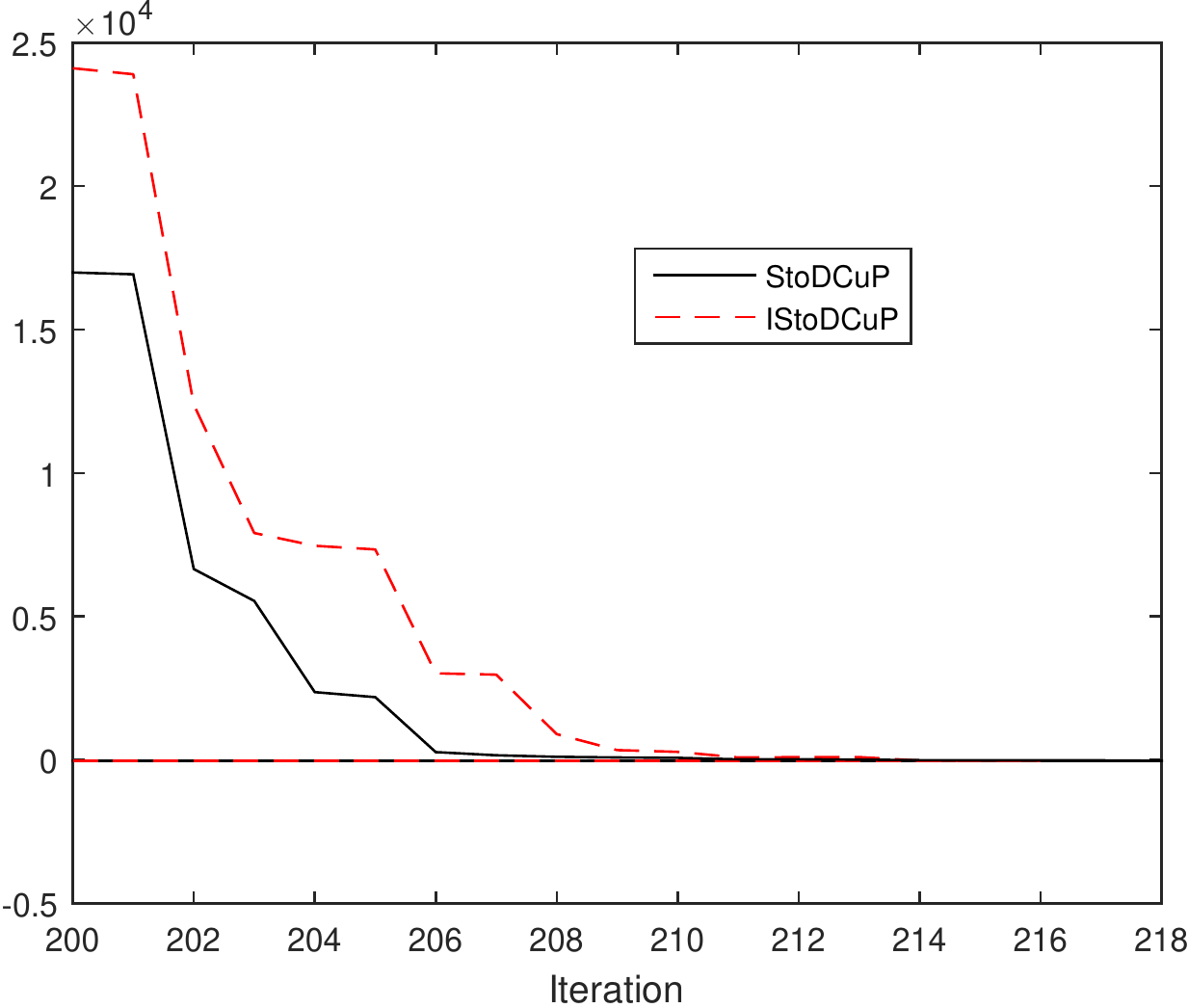}&
\includegraphics[scale=0.6]{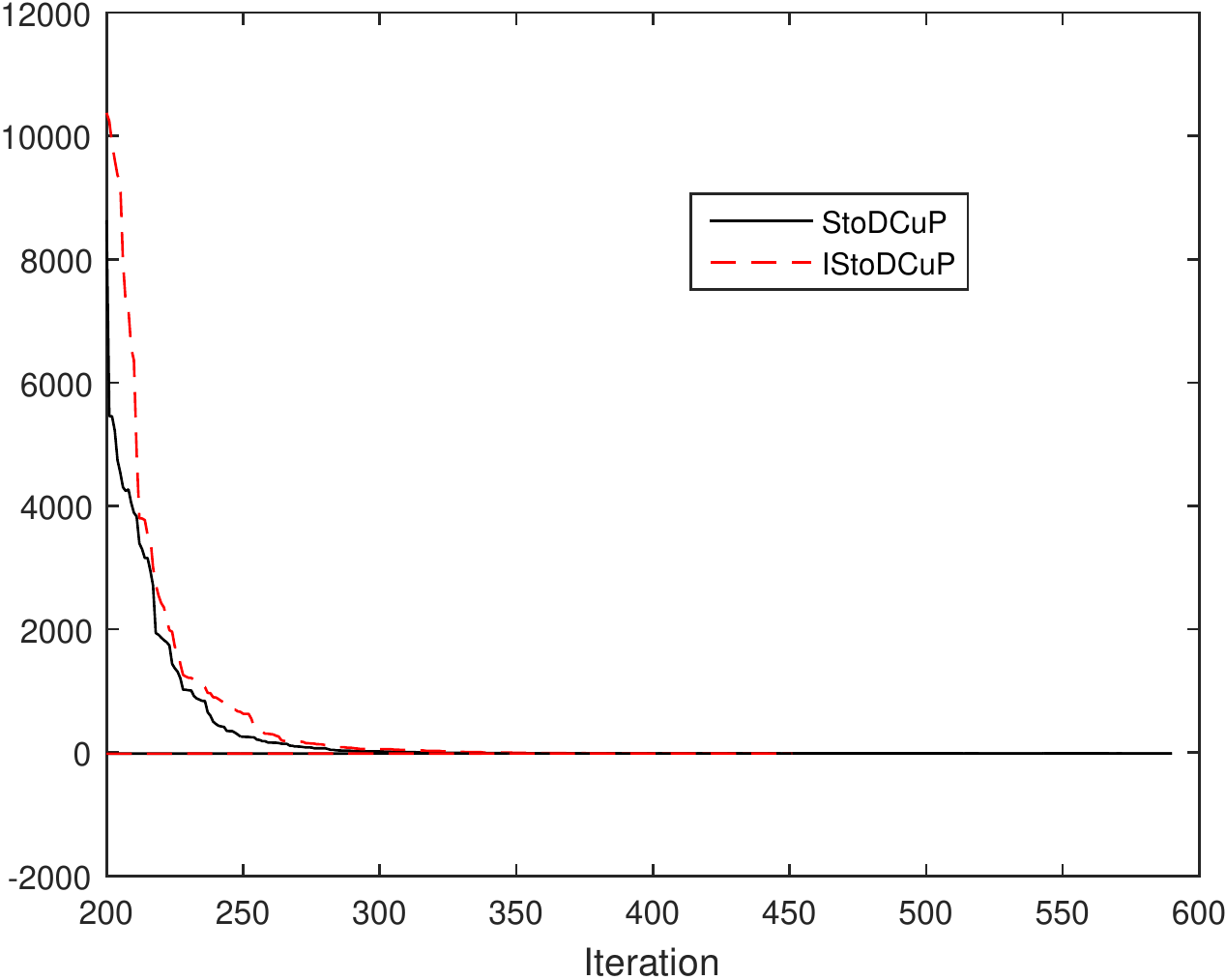}\\
$T=3, n=10, M=2$& $T=3, n=10, M=10$\\
\includegraphics[scale=0.6]{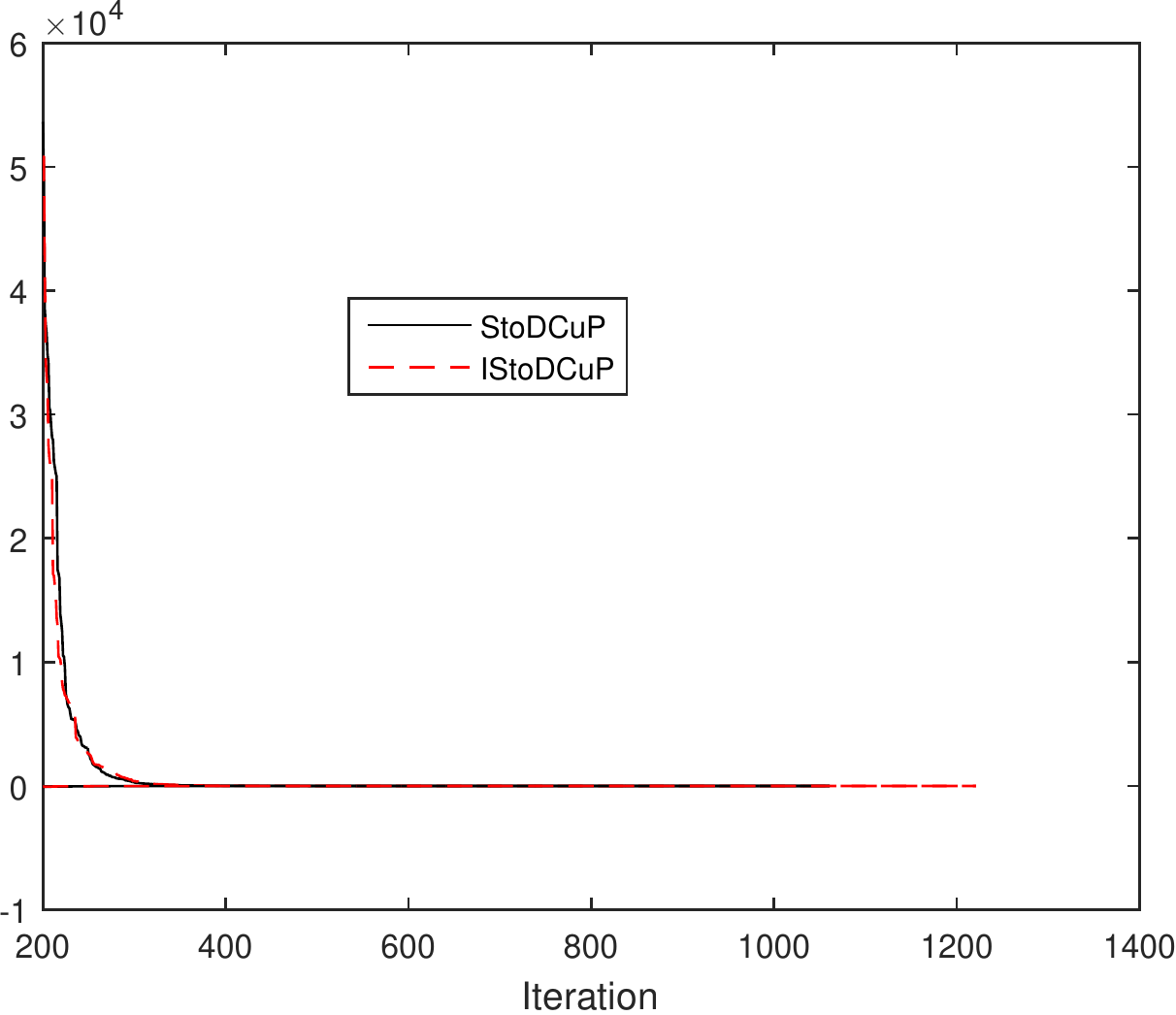}&
\includegraphics[scale=0.6]{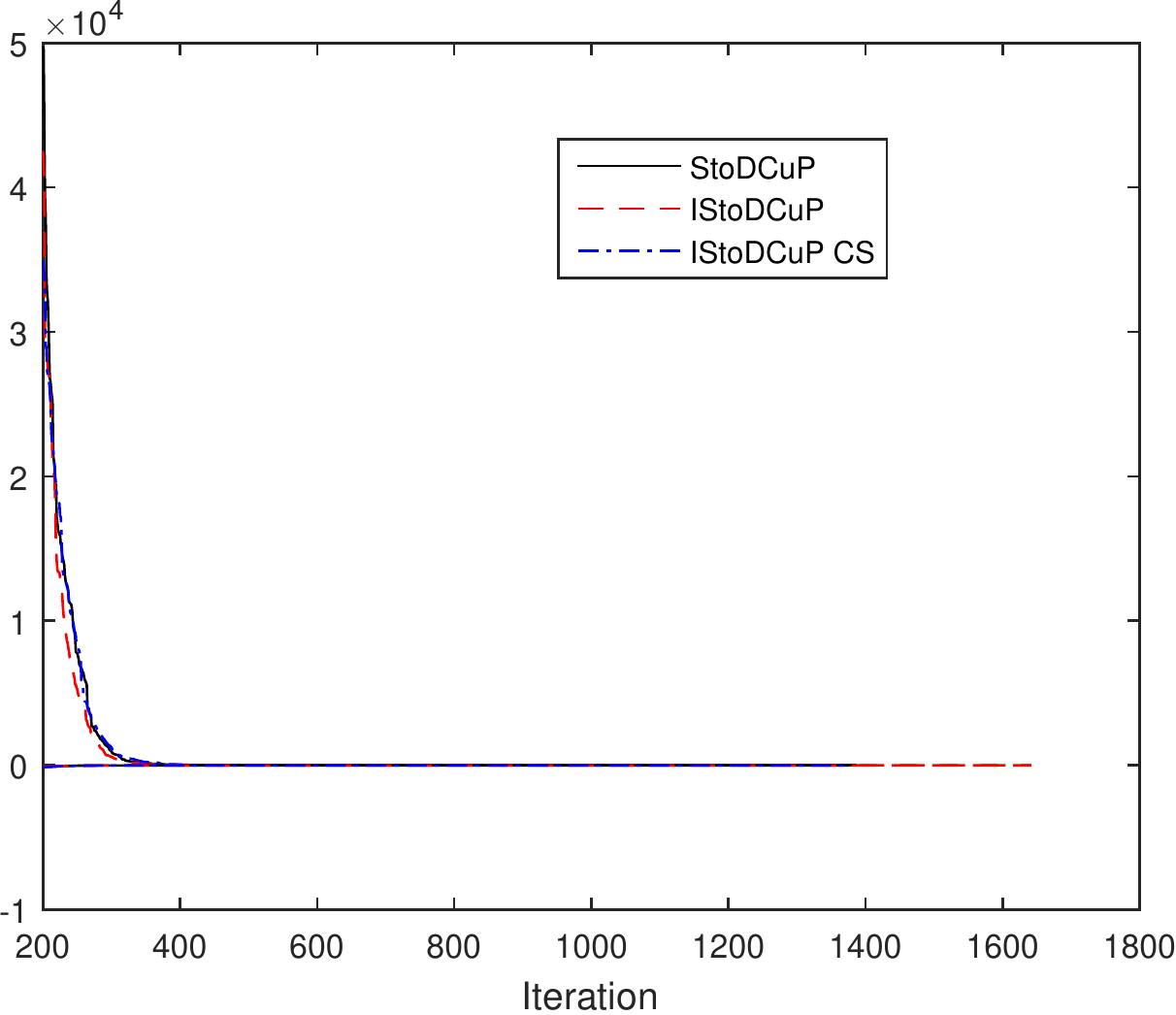}\\
$T=5, n=10, M=10$&$T=5, n=10, M=20$\\
\includegraphics[scale=0.6]{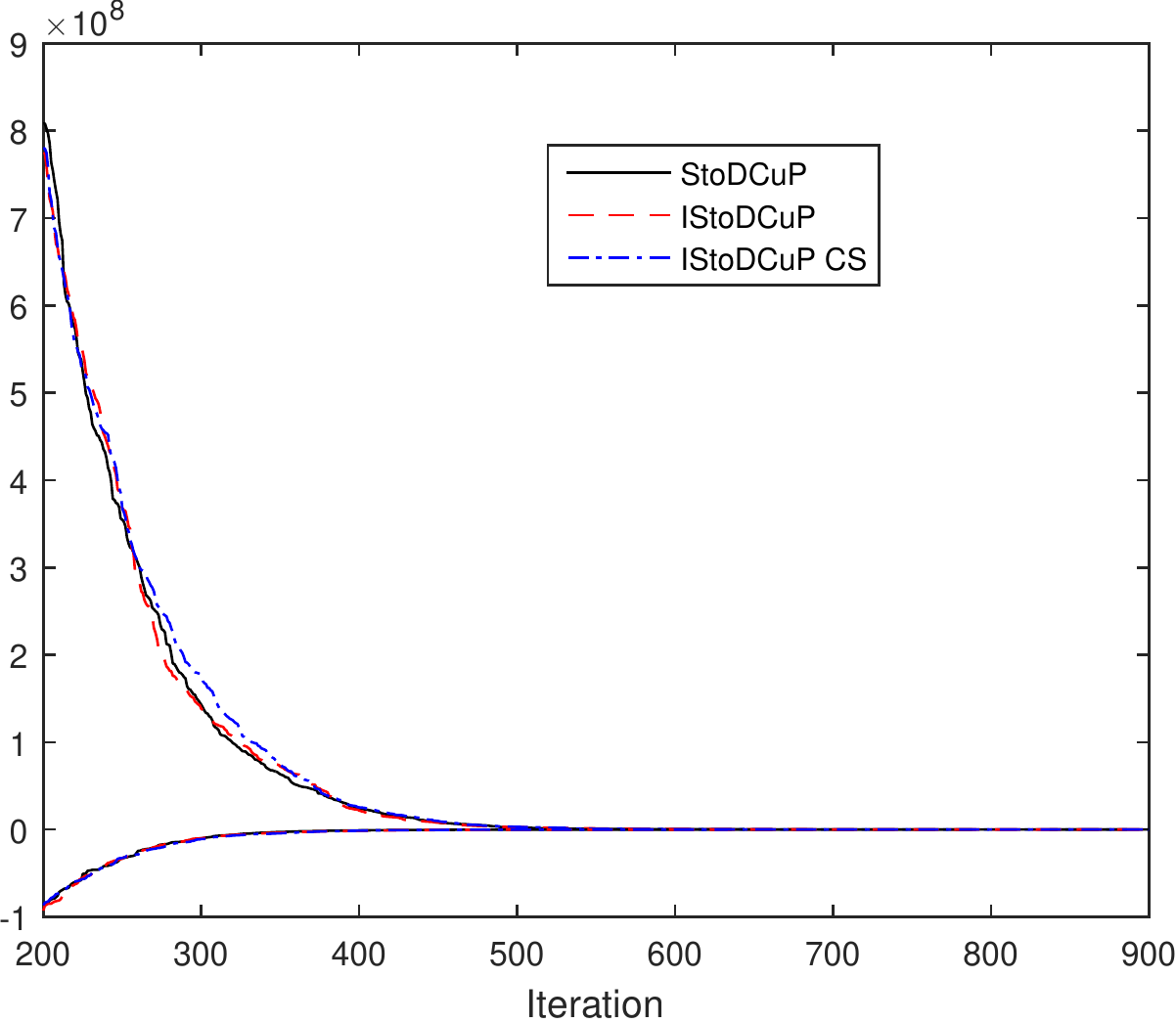}&
\includegraphics[scale=0.6]{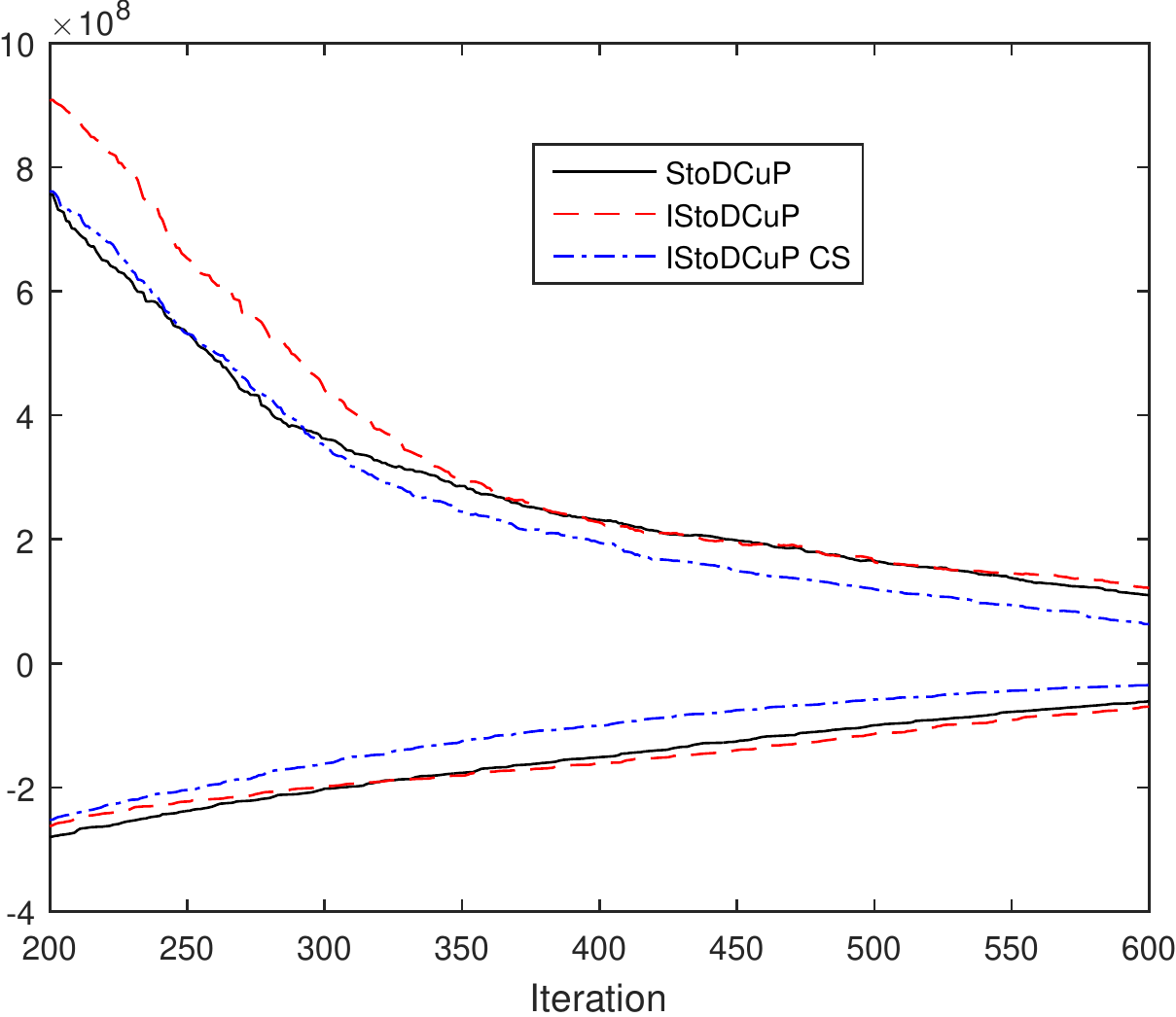}\\
$T=10, n=200, M=10$&$T=10, n=200, M=20$
\end{tabular}
\caption{\label{figureb} Upper and lower bounds computed by StoCuP, IStoDCuP, and IStoDCuP CS along the iterations to solve
the instances.}
\end{figure}

\begin{figure}
\centering
\begin{tabular}{cc}
\includegraphics[scale=0.6]{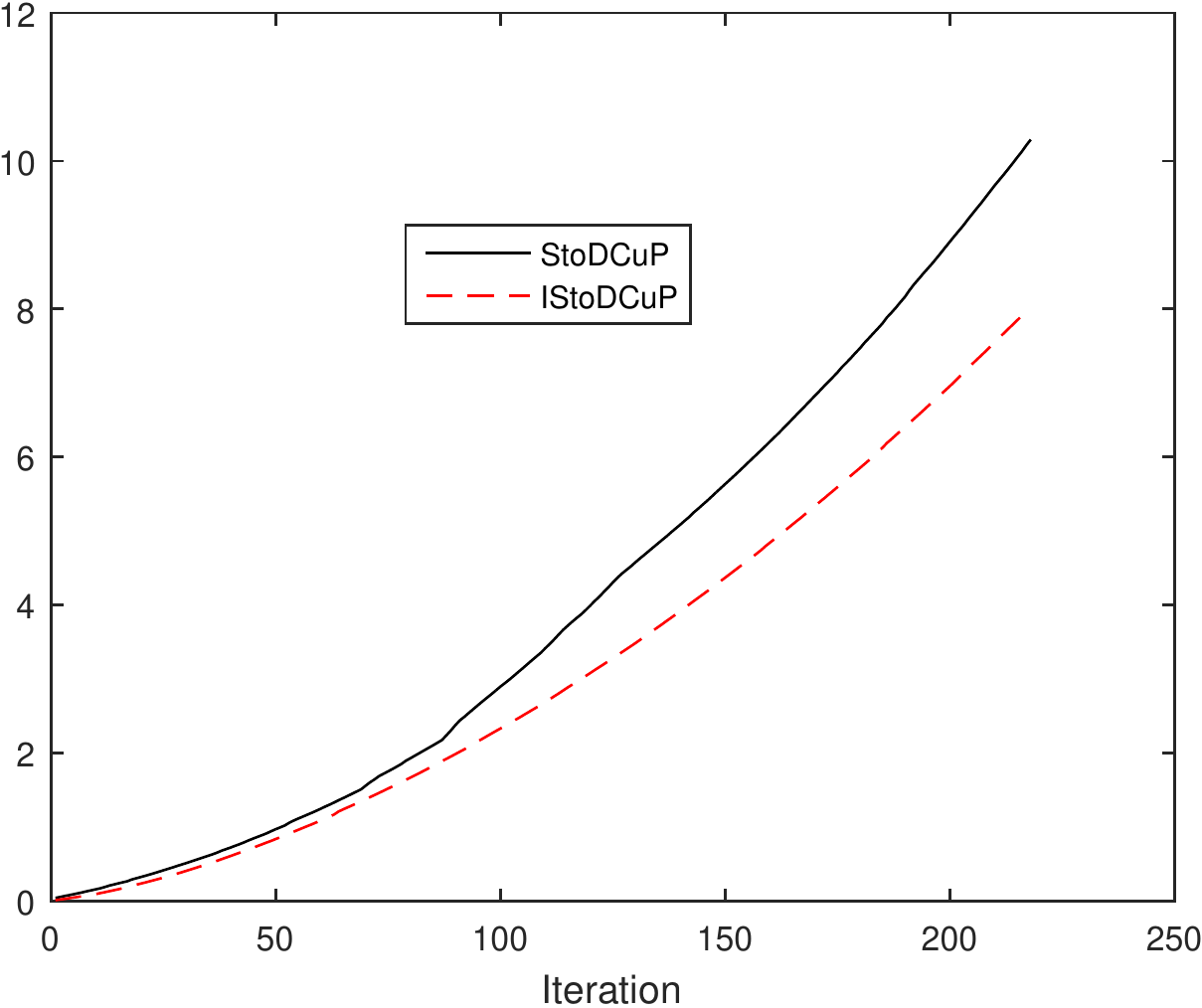}&
\includegraphics[scale=0.6]{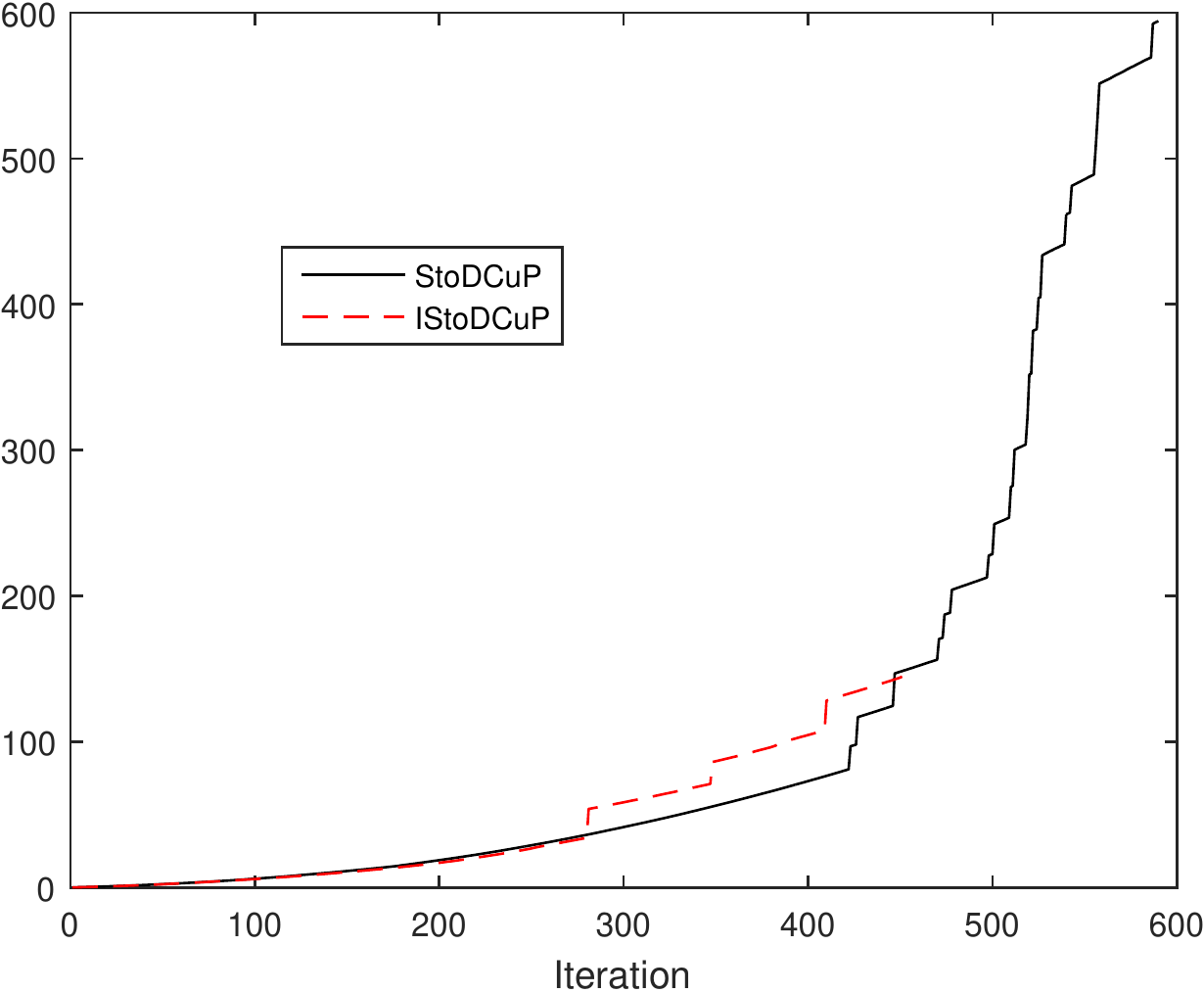}\\
$T=3, n=10, M=2$& $T=3, n=10, M=10$ \\
\includegraphics[scale=0.6]{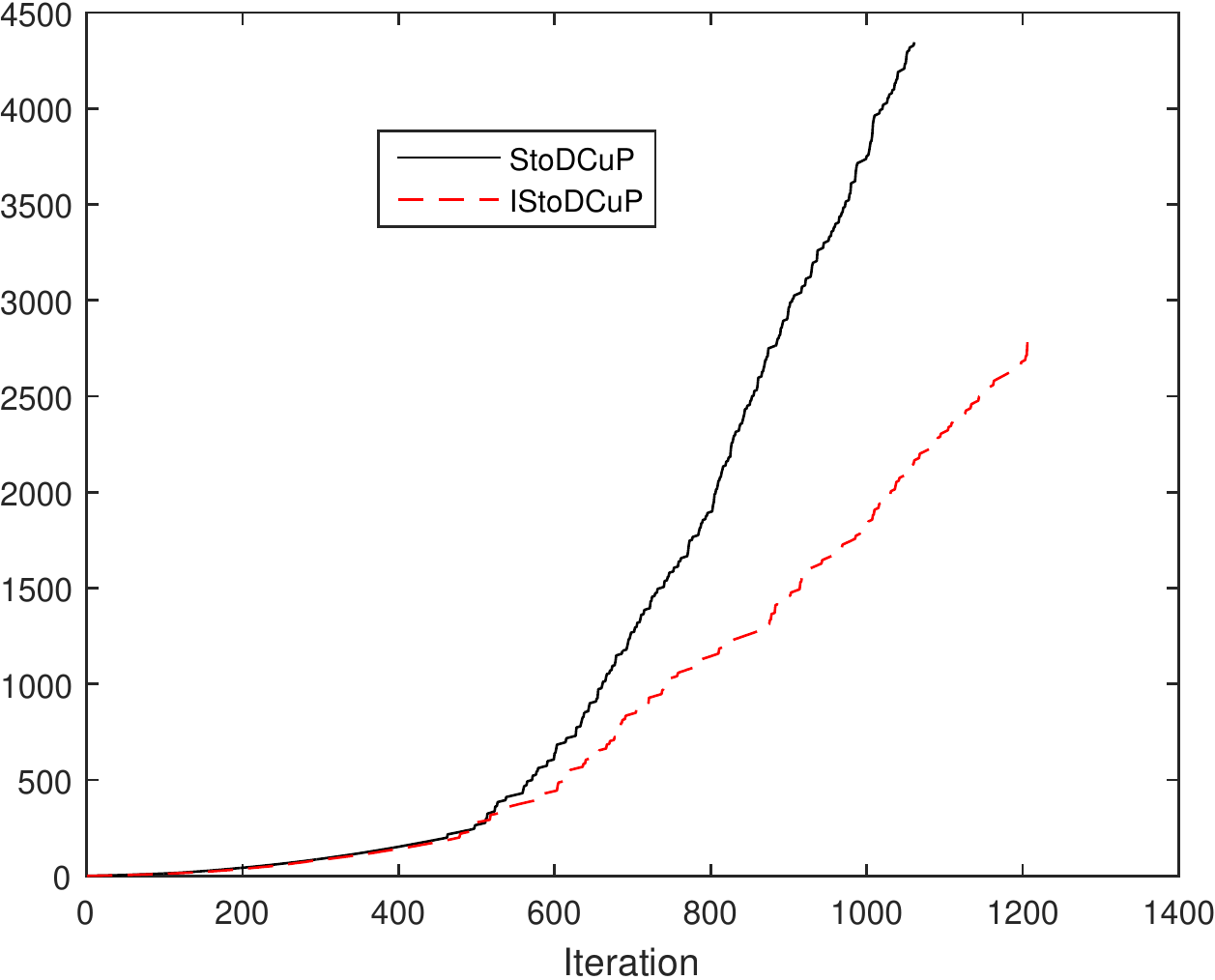}&
\includegraphics[scale=0.6]{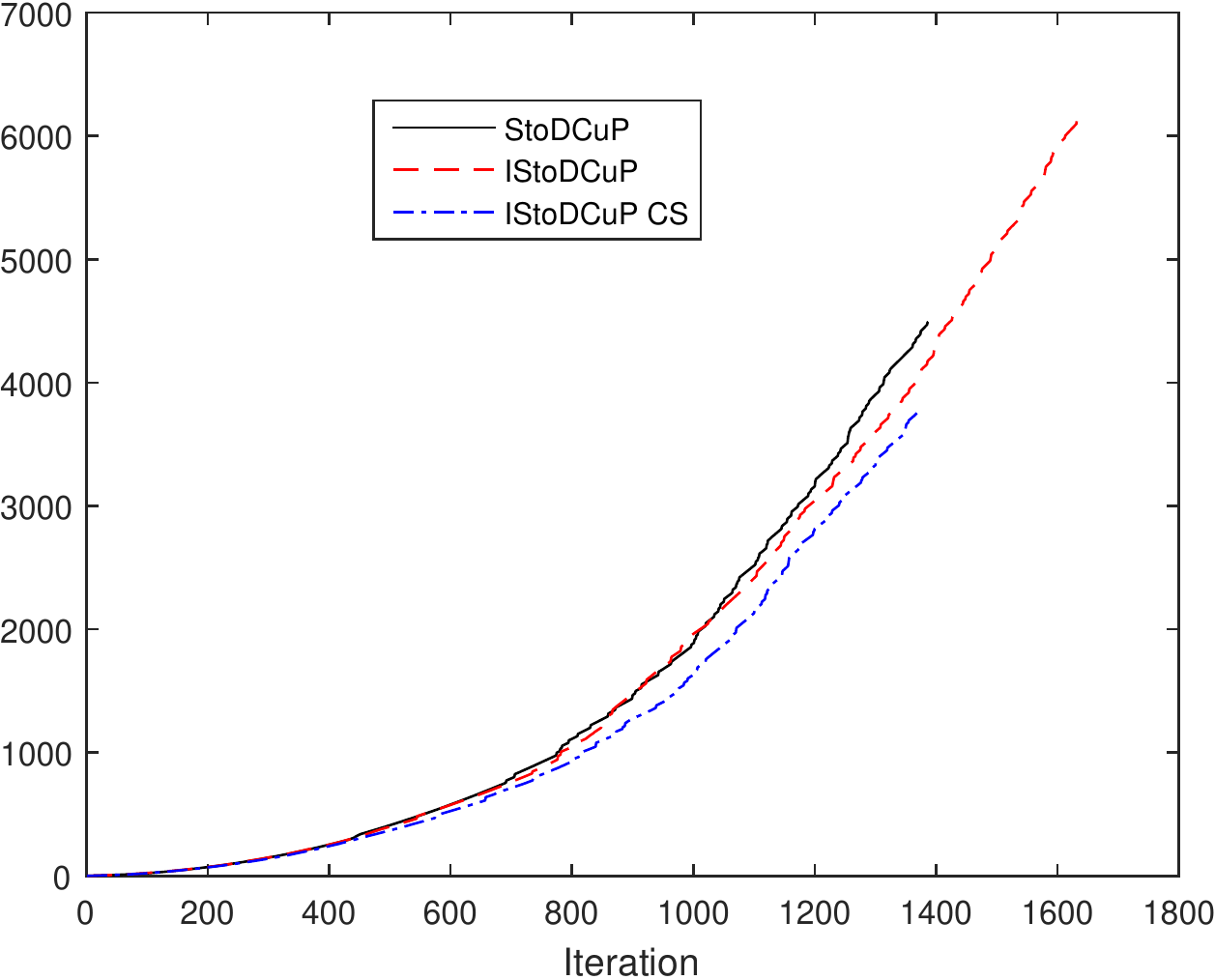}\\
$T=5, n=10, M=10$& $T=5, n=10, M=20$\\
\includegraphics[scale=0.6]{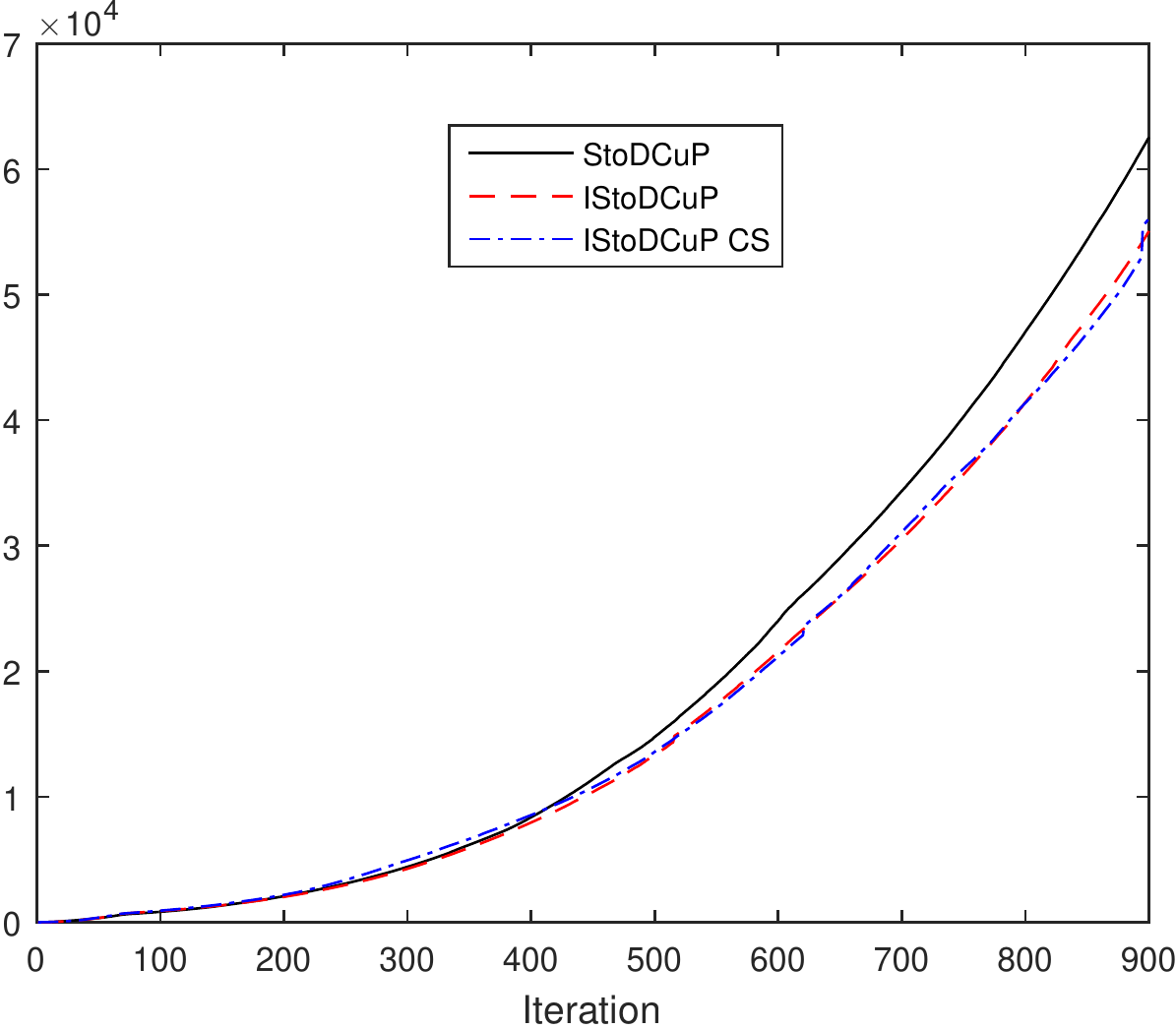}&
\includegraphics[scale=0.6]{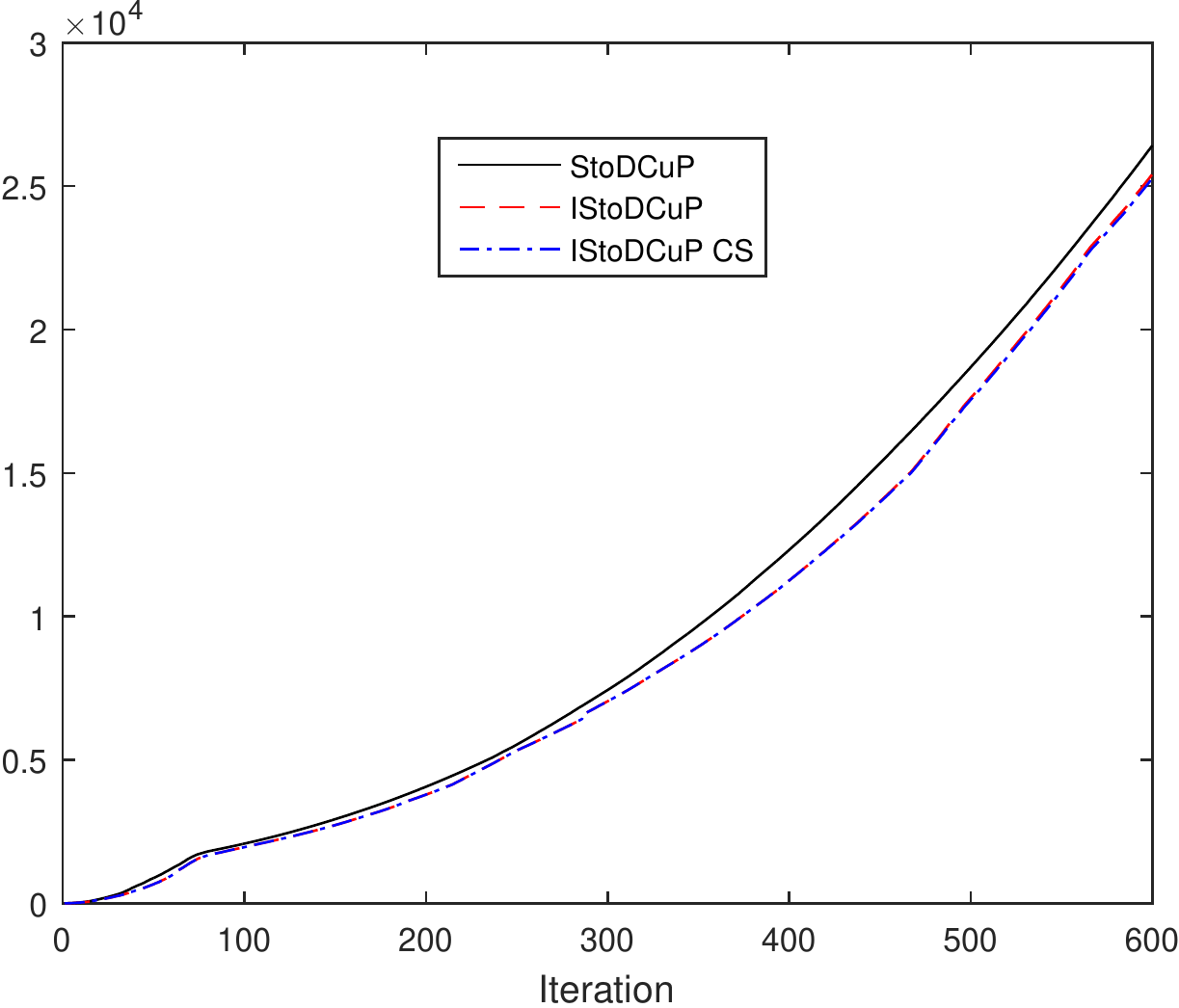}\\
$T=10, n=200, M=10$& $T=10, n=200, M=20$
\end{tabular}
\caption{\label{figurt} Cumulated CPU time in seconds
along the iterations of StoCuP, IStoDCuP, and IStoDCuP CS  to solve
the instances.}
\end{figure}

\begin{table}
\centering
{\small{
\begin{tabular}{|c|c|c|c|c|c|c|}
\hline
 Iteration   & UB IStoDCuP  & UB StoDCuP    &  LB IStoDCuP  & LB StoDCuP
&   Time IStoDCuP & Time StoDCuP \\
\hline
10 &  - & -  &   -36 424  & -15 037& 0.09 & 0.16   \\
\hline
200 & 24 109  & 16 990  &  -29.3120  &-29.3123 & 6.98  & 8.91   \\
\hline
210 & 282.4  & 73.9  & -29.3120 & -29.3123 & 7.57 & 9.67 \\
\hline
216 & -26.99  & -27.34  &  -29.3120  & -29.3123 & 7.63   & 10.13   \\
\hline
 \end{tabular}
}}
\begin{center}
$T=3, n=10, M=2$
\end{center}

{\small{
\begin{tabular}{|c|c|c|c|c|c|c|}
\hline
 Iteration   & UB IStoDCuP  & UB StoDCuP    &  LB IStoDCuP  & LB StoDCuP
&   Time IStoDCuP & Time StoDCuP \\
\hline
10 & -  & -  & -98 584    &-82 872 & 0.29 & 0.35  \\
\hline
210 & 6361.8  & 3889.3   &  -19.023 & -17.44 & 20.6  & 18.5\\
\hline
451 & -14.97 & -13.58  & -16.27  & -16.25&  148.4  & 144.6  \\
\hline
 \end{tabular}
}}
\begin{center}
$T=3, n=10, M=10$
\end{center}

{\small{
\begin{tabular}{|c|c|c|c|c|c|c|}
\hline
 Iteration   & UB IStoDCuP  & UB StoDCuP    &  LB IStoDCuP  & LB StoDCuP
&   Time IStoDCuP & Time StoDCuP \\
\hline
10 & -  & -  & -440 000    & -352 310 & 0.65 & 0.72   \\
\hline
400 &  9.68 & 13.84  & -12.09 & -12.58 & 140.9 & 151.7 \\
\hline
800 & -6.01 & -8.69  & -10.83  & -10.85& 1144.6   & 1897.2   \\
\hline
1061 & -7.84 & -9.78  &-10.72   &-10.72 &  2166  &  4345  \\
\hline
1221 & -9.78 &  - & -10.69  & - & 2825   & -  \\
\hline
 \end{tabular}
}}
\begin{center}
$T=5, n=10, M=10$
\end{center}

{\small{
\begin{tabular}{|c|c|c|c|c|}
\hline
Iteration &  600 & 1376  & 1387  & 1642\\ 
\hline
UB IStoDCuP CS   & -0.586  & -4.5343   & -&-  \\
\hline
UB IStoDCuP   & -1.6317 & -3.7448 &-3.7993  & -4.5153\\
\hline
UB StoDCuP    & -0.0327 & -3.9773 &-4.5648 &  -\\
\hline
LB IStoDCuP CS  & -5.5886 & -4.9595 & -&-  \\
\hline
LB IStoDCuP    & -5.6431 &  -4.9584  & -4.9552 & -4.9078\\
\hline
LB StoDCuP    & -5.7420 & -4.9623 & -4.9591 & -\\
\hline
Time IStoDCuP CS    & 525 &  3784 &- &-\\
\hline
Time IStoDCuP     &  575 & 4106  & 4180& 6178\\
\hline
Time StoDCuP    & 579& 4424 & 4493& -\\
\hline
\end{tabular}
}}
\begin{center}
$T=5, n=10, M=20$
\end{center}

{\small{
\begin{tabular}{|c|c|c|c|}
\hline
Iteration &  400 & 600  & 900  \\ 
\hline
UB IStoDCuP CS   & 2.5343e7  & 3.6465e5   & 143.2  \\
\hline
UB IStoDCuP   & 2.1689e7   & 4.6095e5   & 338.7 \\
\hline
UB StoDCuP    & 2.3785e7  & 3.0292e5   & -50.4  \\
\hline
LB IStoDCuP CS  & -1.3343e6  & -4.0214e4    & -444.4  \\
\hline
LB IStoDCuP    & -1.4643e6  & -5.0292e4   &  -436.4 \\
\hline
LB StoDCuP    & -0.9529e6  &  -2.0954e4  &  -428.9 \\
\hline
Time IStoDCuP CS    & 8 534.6  &  21 166   &  56 082\\
\hline
Time IStoDCuP     & 7 946.7  &  21 557  &  55 061 \\
\hline
Time StoDCuP    &  8 364.2 & 24 015   & 62 536 \\
\hline
\end{tabular}
}}
\begin{center}
$T=10, n=200, M=10$
\end{center}

{\small{
\begin{tabular}{|c|c|c|c|}
\hline
Iteration &  400 & 500  & 600  \\ 
\hline
UB IStoDCuP CS   & 1.943e8  & 1.1955e8   & 0.6321e8  \\
\hline
UB IStoDCuP   &  2.2722e8 &  1.6320e8  &  1.2129e8 \\
\hline
UB StoDCuP    & 2.3129e8  & 1.6563e8   & 1.0990e8  \\
\hline
LB IStoDCuP CS  & -1.0060e8  & -0.5826e8   &  -0.3522e8 \\
\hline
LB IStoDCuP    &  -1.6151e8 & -1.1376e8   &  -0.6979e8 \\
\hline
LB StoDCuP    & -1.5124e8  & -0.9974e8   & -0.6059e8  \\
\hline
Time IStoDCuP CS    & 1.1254e4  & 1.7554e4   & 2.5276e4  \\
\hline
Time IStoDCuP     & 1.1254e4  &  1.7618e4  & 2.5418e4  \\
\hline
Time StoDCuP    & 1.2320e4  & 1.8689e4   &  2.6414e4 \\
\hline
\end{tabular}
}}
\begin{center}
$T=10, n=200, M=20$
\end{center}
\caption{Cumulated CPU time (Time) in seconds and upper (UB) and lower (LB) bounds computed by StoDCuP, IStoDCuP, and
IStoDCuP CS for some iterations and the 6 instances.}
\label{tablebounds}
\end{table}

\section{Conclusion}

We introduced StoDCuP, a variant of 
SDDP which builds linearizations of
some or all nonlinear
constraint and objective
functions along the iterations
of the method, as well as an inexact
variant of StoDCuP which is able
to cope with approximate primal-dual
solutions of the subproblems solved
along the iterations.
We have shown the convergence of 
StoDCuP and of Inexact StoDCuP
for vanishing error
terms $\varepsilon_t^k$.

Our numerical experiments have
illustrated on a difficult
nonlinear nondifferentiable
multistage stochastic program
that StoDCuP can be an alternative
solution method to SDDP and that
its inexact variant can converge
quicker than StoDCuP. An interesting
feature of the inexact variant
is its flexiblity, able to cope
with any approximate primal-dual
solution to the subproblems, allowing to further study 
the impact of the calibration of error
terms $\varepsilon_t^k$ on
the performance of Inexact StoDCuP.
For DCuP, the calibration seems
simpler, see for instance Remark 2
in \cite{guiguesinexact2018}
on the calibration of the error
terms for Inexact DDP which also
applies to Inexact DCuP.


\section*{Acknowledgments} The research of the first author was 
partially supported by an FGV grant, CNPq grants 401371/2014-0, 311289/2016-9, 204872/2018-9,
and FAPERJ grant E-26/201.599/2014. 
Research of the second author was partially supported by CNPq grant 401371/2014-0.

\addcontentsline{toc}{section}{References}
\bibliographystyle{plain}
\bibliography{Decomp_Algo_MSP_1-15-2020_Submitted}

\section*{Data availability statement}

All (simulated) data generated or analysed during this study can be obtained
following the steps
given in Section \ref{sec:numexp}
of  this article.

\section*{Appendix}

\if{

------------------------------------------

\[
\bar  v( \bar y) = \inf \{ f(x) : A(x) + B(\bar y) \le 0, \, x \in X \} \quad (1)
\]
\[
v(b) = \inf \{ f(x) : A(x) + b \le 0, \, x \in X \} \quad (2)
\]
Clearly,
\[
\tilde v = v \circ B
\]
So,
\[
B^* \partial v(B\bar y) \subset \partial (v \circ B)(\bar y) = \partial {\bar v}(\bar y)
\]
How to compute $\partial v(b)$? Assume that $b$ is such that there exists $x \in \ri X$ such that
$A(x) +  b \le 0$. Then, $\partial v(b) \ne \emptyset$ and, if $\bar x$ is an optimal solution
of (2),  then $\partial v(b)$  consists of all $\lambda$ satisfying
\[
0 \in \nabla f(\bar x)  + A^* \lambda +N_X(\bar x), \quad \lambda \ge 0, \quad \lambda^T[A(\bar x)+b]=0
\]
How about $B^* \partial v(B\bar y)$? Assume that $\bar y$ is such that there exists $x \in \ri X$ such that
$A(x) +  B(\bar y) \le 0$.  If $\bar x$ is an optimal solution
of (1),  Then,  $B^* \partial v(B\bar y)$ consists of all $B^*\lambda$ such that $\lambda$ satisfies
\[
0 \in \nabla f(\bar x)  + A^* \lambda +N_X(\bar x), \quad \lambda \ge 0, \quad \lambda^T[A(\bar x)+B(\bar y)]=0
\]
Is the above also a characterization of $\partial \bar  v (\bar y)$? This would be true if
$Im(B) \cap \ri (\dom v) \ne \emptyset$, or equivalently,
there exists $x \in \ri X$ such that
$A(x) +  B(\bar y) < 0$.

}\fi


To prove \eqref{proofg0} and \eqref{limitQNotinSn}, we will need the following lemma 
(the proof of (ii) of this lemma was given in \cite{lecphilgirar12} for a more general sampling scheme
and the proof of (i), that we detail,  is similar to the proof of (ii)):
\begin{lemma}\label{lemmaappendix} Assume that Assumptions (H0), (H1)-Sto, and (H2) hold 
for StoDCuP. Define random variables 
$y_n^k = 1(k \in \mathcal{S}_n)$.

(i) Let $\varepsilon>0$, $t \in \{1,\ldots,T\}$, $n \in {\tt{Nodes}(t-1)}$, $m \in C(n)$, $i \in \{1,\ldots,p\}$ and set 
$$
K_{\varepsilon, m , i}=\left\{k \geq 1 : g_{t i}( x_m^{k}, x_{n}^{k} , \xi_m ) - g_{t i j_t(m)}^{k-1}( x_m^{k}, x_{n}^{k} ) \geq \varepsilon \right\}.
$$
Let 
$$
\Omega_0( \varepsilon )=\{ \omega \in \Omega : |K_{\varepsilon, m , i}(\omega)|  \mbox{ is infinite}  \}
$$
and assume that $\Omega_0( \varepsilon ) \neq \emptyset$. Define on the sample space 
$\Omega_0 ( \varepsilon )$ the random variables $\mathcal{I}_{\varepsilon, m , i}(j), j \geq 1$, where
$\mathcal{I}_{\varepsilon, m , i}(1)=\min\{k \geq 1: k \in  K_{\varepsilon, m , i}(\omega) \}$
and for $j \geq 2$
$$
\mathcal{I}_{\varepsilon, m , i}(j)=\min\{ k > \mathcal{I}_{\varepsilon, m , i}(j-1) : k \in  K_{\varepsilon, m , i}(\omega)\},
$$
i.e., $\mathcal{I}_{\varepsilon, m , i}(j)(\omega)$ is the index of $j$th iteration $k$
such that $g_{t i}( x_m^{k}, x_{n}^{k} , \xi_m ) - g_{t i j_t
(m)}^{k-1}( x_m^{k}, x_{n}^{k} ) \geq \varepsilon$. Then random 
variables $(y_n^{\mathcal{I}_{\varepsilon, m , i}(j)})_{j \geq 1}$ defined on sample space 
$\Omega_0(\varepsilon)$ are independent, have the distribution of $y_n^1$ and therefore by the Strong 
Law of Large numbers we have 
\begin{equation}\label{sllngi}
\mathbb{P}\left( \lim_{N \rightarrow +\infty} \frac{1}{N} \sum_{j=1}^N  y_n^{\mathcal{I}_{\varepsilon, m , i}(j)}= \mathbb{E}[y_n^1]  \right) =1.
\end{equation}
(ii) Let $\varepsilon>0$, $t \in \{1,\ldots,T\}$, $n \in {\tt{Nodes}(t-1)}$,  and set 
$$
K_{\varepsilon, n}=\left\{k \geq 1 : \mathcal{Q}_t (  x_{n}^{k} ) - \mathcal{Q}_t^k ( x_{n}^{k} ) \geq \varepsilon \right\}.
$$
Let 
$$
\Omega_1( \varepsilon )=\{\omega \in \Omega : |K_{\varepsilon, n}(\omega)|  \mbox{ is infinite}  \}
$$
and assume that $\Omega_1( \varepsilon ) \neq \emptyset$. Define on the sample space 
$\Omega_1( \varepsilon )$ the random variables $\mathcal{I}_{\varepsilon, n}(j), j \geq 1$, where
$\mathcal{I}_{\varepsilon, n}(1)=\min\{k \geq 1: k \in  K_{\varepsilon, n}(\omega) \}$
and for $j \geq 2$
$$
\mathcal{I}_{\varepsilon, n}(j)=\min\{ k > \mathcal{I}_{\varepsilon, n}(j-1) : k \in  K_{\varepsilon, n}(\omega)\},
$$
i.e., $\mathcal{I}_{\varepsilon, n}(j)(\omega)$ is the index of $j$th iteration $k$
such that $\mathcal{Q}_t (  x_{n}^{k} ) - \mathcal{Q}_t^k ( x_{n}^{k} ) \geq \varepsilon$.
Then random variables $(y_n^{\mathcal{I}_{\varepsilon, n}(j)})_{j \geq 1}$ defined on sample space $\Omega_1( \varepsilon )$ are independent, have the distribution of
$y_n^1$ and therefore by the Strong Law of Large numbers we have 
\begin{equation}\label{sllngi2}
\mathbb{P}\left( \lim_{N \rightarrow +\infty} \frac{1}{N} \sum_{j=1}^N  y_n^{\mathcal{I}_{\varepsilon, n}(j)}= \mathbb{E}[y_n^1]  \right) =1.
\end{equation}
\end{lemma}
\begin{proof}
(i) Define on the sample space 
$\Omega_0( \varepsilon )$ the random variables $(w_{\varepsilon, m , i}^k)_k$  by 
$$
w_{\varepsilon, m , i}^k(\omega) =\left\{ 
\begin{array}{ll}
1  &   \mbox{if }k \in K_{\varepsilon, m , i}(\omega) \\
0 & \mbox{otherwise.}
\end{array}
\right.
$$ 
To alleviate notation ($\varepsilon, m , n, i$ being fixed), let us
put $w^k := w_{\varepsilon, m , i}^k$, $\mathcal{I}(j):=\mathcal{I}_{\varepsilon, m , i}(j)$,
For ${\overline{y}}_j \in \{0,1\}$, we have 
\begin{equation}\label{firsteq}
\mathbb{P}\Big(y_n^{\mathcal{I}(j)} = {\overline{y}}_j \Big)
= \sum_{{\overline{\mathcal{I}}}_j =1}^{\infty} \mathbb{P}\Big(y_n^{{\overline{\mathcal{I}}}_j} = {\overline{y}}_j ;  \mathcal{I}(j) =  {\overline{\mathcal{I}}}_j    \Big). 
\end{equation}
Observe that the event $\mathcal{I}(j) =  {\overline{\mathcal{I}}}_j$ can be written 
as the union $\bigcup_{1 \leq {\overline{\mathcal{I}}}_1 < {\overline{\mathcal{I}}}_2 < \ldots < {\overline{\mathcal{I}}}_j} E({\overline{\mathcal{I}}}_1,\ldots,{\overline{\mathcal{I}}}_j)$   of events 
$$
E({\overline{\mathcal{I}}}_1,\ldots,{\overline{\mathcal{I}}}_j):=
\left\{ 
\begin{array}{l}
w^{\overline{\mathcal{I}}_1} =\ldots =w^{\overline{\mathcal{I}}_j}=1,\\
w^{\ell} = 0, 1 \leq \ell < \overline{\mathcal{I}}_j, \ell \notin \{ \overline{\mathcal{I}}_1,\ldots, \overline{\mathcal{I}}_j \}
\end{array}
\right\}.
$$
Due to Assumption (H2) observe that random variable $y_n^{{\overline{\mathcal{I}}}_j}$ is independent of 
random variables $w^{i}, i=1,\ldots,\overline{\mathcal{I}}_j$, and therefore events 
$\{y_n^{{\overline{\mathcal{I}}}_j} = {\overline{y}}_j \}$ and $\{ \mathcal{I}(j) =  {\overline{\mathcal{I}}}_j \}$ are independent which gives
\begin{equation}\label{disty1}
\begin{array}{lll}
\mathbb{P}\Big(y_n^{\mathcal{I}(j)} = {\overline{y}}_j \Big)
&=&\displaystyle  \sum_{{\overline{\mathcal{I}}}_j =1}^{\infty} \mathbb{P}\Big(y_n^{{\overline{\mathcal{I}}}_j} = {\overline{y}}_j ;  \mathcal{I}(j) =  {\overline{\mathcal{I}}}_j    \Big) \\
&=&\displaystyle  \sum_{{\overline{\mathcal{I}}}_j =1}^{\infty} \mathbb{P}\Big(y_n^{{\overline{\mathcal{I}}}_j} =  {\overline{y}}_j \Big) \mathbb{P}\Big(  \mathcal{I}(j) =  {\overline{\mathcal{I}}}_j    \Big)
= \displaystyle \mathbb{P}\Big(y_n^{1} =  {\overline{y}}_j \Big)  \sum_{{\overline{\mathcal{I}}}_j =1}^{\infty}  \mathbb{P}\Big(  \mathcal{I}(j) =  {\overline{\mathcal{I}}}_j    \Big)= \mathbb{P}\Big(y_n^{1} =  {\overline{y}}_j \Big)
\end{array}
\end{equation}
where we have used the fact that $y_n^{1}$ and $y_n^{{\overline{\mathcal{I}}}_j}$ have the same distribution (from (H2)). 

Next for ${\overline{y}}_1,\ldots,{\overline{y}}_p \in \{0,1\}$, we have
$$
\begin{array}{lll}
 \mathbb{P}\Big(y_n^{\mathcal{I}(1)} = {\overline{y}}_1, \ldots, y_n^{\mathcal{I}(p)} = {\overline{y}}_p \Big)
 & = & \displaystyle \sum_{1 \leq {\overline{\mathcal{I}}}_1 < {\overline{\mathcal{I}}}_2 <  \ldots < {\overline{\mathcal{I}}}_p}^{\infty}
 \mathbb{P}\Big(y_n^{{\overline{\mathcal{I}}}_1} = {\overline{y}}_1 ; \ldots;y_n^{{\overline{\mathcal{I}}}_p} = {\overline{y}}_p ;
 \mathcal{I}(1) =  {\overline{\mathcal{I}}}_1   ;\ldots;\mathcal{I}(p) =  {\overline{\mathcal{I}}}_p    \Big).
\end{array} 
$$
By the same reasoning as above, the event 
$$
\left\{
y_n^{{\overline{\mathcal{I}}}_1} = {\overline{y}}_1 ; \ldots;y_n^{{\overline{\mathcal{I}}}_{p-1}} = {\overline{y}}_{p-1} ;
 \mathcal{I}(1) =  {\overline{\mathcal{I}}}_1   ;\ldots;\mathcal{I}(p) =  {\overline{\mathcal{I}}}_p \right\}
$$
can be expressed in terms of random variables $y_n^{{\overline{\mathcal{I}}}_1},\ldots,y_n^{{\overline{\mathcal{I}}}_{p-1}},w_n^{{\overline{\mathcal{I}}}_{1}},\ldots,w_n^{{\overline{\mathcal{I}}}_{p}}$,
and is therefore independent of
event $\{ y_n^{{\overline{\mathcal{I}}}_p} = {\overline{y}}_p \}$. It follows that  
\begin{equation}
\begin{array}{l}
 \mathbb{P}\Big(y_n^{\mathcal{I}(j)} = {\overline{y}}_j, 1 \leq j \leq p \Big)
 =   \displaystyle \sum_{1 \leq {\overline{\mathcal{I}}}_1 < {\overline{\mathcal{I}}}_2 <  \ldots < {\overline{\mathcal{I}}}_p}^{\infty}
\mathbb{P}\Big(y_n^{{\overline{\mathcal{I}}}_p} = {\overline{y}}_p  \Big) 
 \mathbb{P}\Big(y_n^{{\overline{\mathcal{I}}}_j} = {\overline{y}}_j, 1 \leq j \leq p-1 ;
 \mathcal{I}(j) =  {\overline{\mathcal{I}}}_j, 1 \leq j \leq p    \Big)\\
 =  \mathbb{P}\Big(y_n^{1} = {\overline{y}}_p  \Big)  \displaystyle \sum_{1 \leq {\overline{\mathcal{I}}}_1 < {\overline{\mathcal{I}}}_2 <  \ldots < {\overline{\mathcal{I}}}_p}^{\infty}
 \mathbb{P}\Big(y_n^{{\overline{\mathcal{I}}}_j} = {\overline{y}}_j, 1 \leq j \leq p-1 ;
 \mathcal{I}(j) =  {\overline{\mathcal{I}}}_j, 1 \leq j \leq p    \Big)\\
 =  \mathbb{P}\Big(y_n^{1} = {\overline{y}}_p  \Big)  \displaystyle \sum_{1 \leq {\overline{\mathcal{I}}}_1 < {\overline{\mathcal{I}}}_2 <  \ldots < {\overline{\mathcal{I}}}_{p-1}}^{\infty}
 \mathbb{P}\Big(y_n^{{\overline{\mathcal{I}}}_j} = {\overline{y}}_j, 1 \leq j \leq p-1 ;
 \mathcal{I}(j) =  {\overline{\mathcal{I}}}_j, 1 \leq j \leq p-1    \Big)\\
 =  \mathbb{P}\Big(y_n^{1} = {\overline{y}}_p  \Big)  \mathbb{P}\Big(y_n^{\mathcal{I}(j)} = {\overline{y}}_j, 1 \leq j \leq p-1 \Big).
 \end{array} 
\end{equation}
By induction this implies
\begin{equation}
\begin{array}{lll}
 \mathbb{P}\Big( y_n^{\mathcal{I}(j)} = {\overline{y}}_j, 1 \leq j \leq p \Big) 
& = & \displaystyle \prod_{j=1}^p \mathbb{P}\Big(y_n^{1} = {\overline{y}}_j  \Big) \stackrel{\eqref{disty1}}{=} \displaystyle \prod_{j=1}^p \mathbb{P}\Big( y_n^{\mathcal{I}(j)}  = {\overline{y}}_j  \Big)
\end{array} 
\end{equation}
which shows that random variables $(y_n^{\mathcal{I}(j)})_{j \geq 1}$ are independent.

The proof of (ii) is similar to the proof of (i).\hfill
\end{proof}

\par {\textbf{Proof of \eqref{proofg0} and \eqref{limitQNotinSn}.}} 
As in \cite{lecphilgirar12}, we can now use the previous lemma to prove \eqref{proofg0} and \eqref{limitQNotinSn}.
Let us prove \eqref{proofg0}. By contradiction, assume that \eqref{proofg0} does not hold. Then there is 
$\varepsilon>0$ such that the set $\Omega_0( \varepsilon  )$ defined in Lemma \ref{lemmaappendix}
is nonempty. By Lemma \ref{lemmaappendix}, this implies that \eqref{sllngi} holds.
But due to \eqref{finalgbef0}, only a finite number of indices ${\mathcal{I}_{\varepsilon, m , i}(j)}$ can be in $\mathcal{S}_n$
(with corresponding variable $y_n^{{\mathcal{I}_{\varepsilon, m , i}(j)}}$ being one) and therefore 
$\mathbb{P}\left( \lim_{N \rightarrow +\infty} \frac{1}{N} \sum_{j=1}^N  y_n^{\mathcal{I}_{\varepsilon, m , i}(j)}= 0  \right) =1$, which is a contradiction with \eqref{sllngi}.

The proof of \eqref{limitQNotinSn} is similar to the proof of \eqref{proofg0}, by contradiction and using 
\eqref{limitQinSn} and Lemma \ref{lemmaappendix}-(ii) (see also \cite{lecphilgirar12}, \cite{guiguessiopt2016}). $\hfill \square$

\end{document}

%% file: def.tex
\usepackage{framed}

\newtheorem{thm}{Theorem}[section]

\newtheorem{lemma}[thm]{Lemma}

\newtheorem{rem}[thm]{Remark}
\newtheorem{remark}[thm]{Remark}

\newcommand{\beq}{\begin{equation}}
\newcommand{\eeq}{\end{equation}}
\newcommand{\beqa}{\begin{eqnarray}}
\newcommand{\eeqa}{\end{eqnarray}}
\newcommand{\beqas}{\begin{eqnarray*}}
\newcommand{\eeqas}{\end{eqnarray*}}
\newcommand{\bi}{\begin{itemize}}
\newcommand{\ei}{\end{itemize}}

\newcommand{\R}{\mathbb{R}}

\newcommand{\inner}[2]{\langle #1,#2\rangle}

\newcommand{\ri}{\mathrm{ri}\,}
\newcommand{\inte}{\mathrm{int}\,}

\newcommand{\dom}{\mathrm{dom}\,}
\newcommand{\Dom}{\mathrm{Dom}\,}
